\documentclass[letterpaper,11pt]{article}

\usepackage[margin=1in]{geometry}  % set the margins to 1in on all sides
\usepackage{xspace,exscale,relsize}
\usepackage{fancybox,shadow}
\usepackage{graphicx}
\usepackage{color}
\usepackage{amsthm,amsfonts}
\usepackage{amssymb}
\usepackage{authblk}
\usepackage[title]{appendix}

\usepackage{extarrows}

\usepackage[noadjust]{cite}

\usepackage{amsmath}
	\makeatletter
	\let\over=\@@over \let\overwithdelims=\@@overwithdelims
	\let\atop=\@@atop \let\atopwithdelims=\@@atopwithdelims
  	\let\above=\@@above \let\abovewithdelims=\@@abovewithdelims
  	\makeatother
\usepackage{rotating}

\usepackage{ifpdf}

\usepackage{subfigure}
\usepackage{psfrag}

\usepackage{prettyref}
\usepackage{enumitem}

\usepackage{tikz}
\usetikzlibrary{arrows}
\tikzstyle{int}=[draw, fill=blue!20, minimum size=2em]
\tikzstyle{dot}=[circle, draw, fill=blue!20, minimum size=2em]
\tikzstyle{init} = [pin edge={to-,thin,black}]

\usepackage[%dvips,
            CJKbookmarks=true,
            bookmarksnumbered=true,
            bookmarksopen=true,
            colorlinks=true,
            citecolor=red,
            linkcolor=blue,
            anchorcolor=red,
            urlcolor=blue,
            pdfauthor={Polyanskiy and Wu}
            ]{hyperref}

\usepackage[all]{xy}

\usepackage{mathtools}
\usepackage{ifthen}
\newboolean{aos}
\setboolean{aos}{FALSE}

\newcommand{\matx}{\ensuremath{\mathcal{X}}}

\newcommand{\mate}{\ensuremath{\mathcal{E}}}

\newcommand{\matn}{\ensuremath{\mathcal{N}}}
\newcommand{\matf}{\ensuremath{\mathcal{F}}}
\newcommand{\matg}{\ensuremath{\mathcal{G}}}
\newcommand{\matp}{\ensuremath{\mathcal{P}}}

\newcommand{\mreals}{\ensuremath{\mathbb{R}}}

\ifx\eqref\undefined
	\newcommand{\eqref}[1]{~(\ref{#1})}
\fi
\ifx\mod\undefined
	\def\mod{\mathop{\rm mod}}
\fi

\usepackage{bm}
\def\vect#1{\bm{#1}}

\def\dim{\mathop{\rm dim}}

\def\exp{\mathop{\rm exp}}
\def\lspan{\mathop{\rm span}}

\def\EE{\Expect}
\DeclareMathOperator\sign{\rm sign}

\def\Var{\mathrm{Var}}
\def\Cov{\mathrm{Cov}}

\def\PP{\mathbb{P}}

\def\eqdef{\triangleq}

\def\simiid{\stackrel{iid}{\sim}}

\newcommand{\zeros}{\mathbf{0}}

\newcommand{\stepa}[1]{\overset{\rm (a)}{#1}}
\newcommand{\stepb}[1]{\overset{\rm (b)}{#1}}
\newcommand{\stepc}[1]{\overset{\rm (c)}{#1}}
\newcommand{\stepd}[1]{\overset{\rm (d)}{#1}}

\newcommand{\I}{\text{(I)}}
\newcommand{\II}{\text{(II)}}
\newcommand{\III}{\text{(III)}}
\newcommand{\IV}{\text{(IV)}}

\newcommand{\bTheta}{\boldsymbol{\Theta}}
\newcommand{\btheta}{\boldsymbol{\theta}}
\newcommand{\bX}{\boldsymbol{X}}

\newcommand{\deltaa}{\delta_{\rm bv}}
\newcommand{\REB}{R_{\rm iid}^*}
\newcommand{\RComp}{R_{\rm det}^*}

\newcommand{\Poi}{\mathrm{Poi}}

\newcommand{\floor}[1]{{\left\lfloor {#1} \right \rfloor}}

\newcommand{\reals}{\mathbb{R}}
\newcommand{\naturals}{\mathbb{N}}
\newcommand{\integers}{\mathbb{Z}}
\newcommand{\complex}{\mathbb{C}}
\newcommand{\Expect}{\mathbb{E}}
\newcommand{\expect}[1]{\mathbb{E}\left[#1\right]}
\newcommand{\eexpect}[1]{\mathbb{E}[#1]}

\newcommand{\prob}[1]{\mathbb{P}\left[#1\right]}
\newcommand{\pprob}[1]{\mathbb{P}[#1]}

\newcommand{\TV}{{\rm TV}}
\newcommand{\Aff}{{\rm Aff}}
\newcommand{\Mu}{M}

\newcommand{\expects}[2]{\mathbb{E}_{#2}\left[ #1 \right]}

\newcommand{\eg}{e.g.\xspace}

\newcommand{\iid}{iid\xspace}
\newcommand{\ind}{ind.\xspace}

\newcommand{\pth}[1]{\left( #1 \right)}
\newcommand{\qth}[1]{\left[ #1 \right]}
\newcommand{\sth}[1]{\left\{ #1 \right\}}

\newcommand{\iiddistr}{{\stackrel{\text{\iid}}{\sim}}}
\renewcommand{\simiid}{\iiddistr}
\newcommand{\inddistr}{{\stackrel{\text{\ind}}{\sim}}}

\newcommand{\var}{\Var}

\newcommand{\Bern}{\text{Bern}}
\newcommand{\Binom}{\text{Binom}}
\newcommand{\Bino}{\Binom}
\newcommand{\iprod}[2]{\left \langle #1, #2 \right\rangle}
\newcommand{\Iprod}[2]{\langle #1, #2 \rangle}
\newcommand{\indc}[1]{{\mathbf{1}_{\left\{{#1}\right\}}}}
\newcommand{\Indc}{\mathbf{1}}

\definecolor{myblue}{rgb}{.8, .8, 1}
\definecolor{mathblue}{rgb}{0.2472, 0.24, 0.6} % mathematica's Color[1, 1--3]
\definecolor{mathred}{rgb}{0.6, 0.24, 0.442893}
\definecolor{mathyellow}{rgb}{0.6, 0.547014, 0.24}

\newcommand{\calE}{{\mathcal{E}}}
\newcommand{\calF}{{\mathcal{F}}}

\newcommand{\calP}{{\mathcal{P}}}

\newcommand{\calX}{{\mathcal{X}}}

\newcommand{\diverge}{\to \infty}

\newrefformat{eq}{(\ref{#1})}
\newrefformat{thm}{Theorem~\ref{#1}}
\newrefformat{th}{Theorem~\ref{#1}}
\newrefformat{chap}{Chapter~\ref{#1}}
\newrefformat{sec}{Section~\ref{#1}}
\newrefformat{seca}{Section~\ref{#1}}
\newrefformat{algo}{Algorithm~\ref{#1}}
\newrefformat{fig}{Fig.~\ref{#1}}
\newrefformat{tab}{Table~\ref{#1}}
\newrefformat{rmk}{Remark~\ref{#1}}
\newrefformat{clm}{Claim~\ref{#1}}
\newrefformat{def}{Definition~\ref{#1}}
\newrefformat{cor}{Corollary~\ref{#1}}
\newrefformat{lmm}{Lemma~\ref{#1}}
\newrefformat{prop}{Proposition~\ref{#1}}
\newrefformat{pr}{Proposition~\ref{#1}}
\newrefformat{app}{Appendix~\ref{#1}}
\newrefformat{apx}{Appendix~\ref{#1}}
\newrefformat{ex}{Example~\ref{#1}}
\newrefformat{exer}{Exercise~\ref{#1}}
\newrefformat{soln}{Solution~\ref{#1}}
\def\unifto{\mathop{{\mskip 3mu plus 2mu minus 1mu%
	\setbox0=\hbox{$\mathchar"3221$}%
	\raise.6ex\copy0\kern-\wd0%
	\lower0.5ex\hbox{$\mathchar"3221$}}\mskip 3mu plus 2mu minus 1mu}}

\ifx\lesssim\undefined
\def\simleq{{{\mskip 3mu plus 2mu minus 1mu%
	\setbox0=\hbox{$\mathchar"013C$}%
	\raise.2ex\copy0\kern-\wd0%
	\lower0.9ex\hbox{$\mathchar"0218$}}\mskip 3mu plus 2mu minus 1mu}}
\else
\def\simleq{\lesssim}
\fi

\ifx\gtrsim\undefined
\def\simgeq{{{\mskip 3mu plus 2mu minus 1mu%
	\setbox0=\hbox{$\mathchar"013E$}%
	\raise.2ex\copy0\kern-\wd0%
	\lower0.9ex\hbox{$\mathchar"0218$}}\mskip 3mu plus 2mu minus 1mu}}
\else
\def\simgeq{\gtrsim}
\fi

\def\delchi{\delta_{\chi^2}}
\def\delTV{\delta_{\TV}}

\newtheorem{theorem}{Theorem}
\newtheorem{lemma}[theorem]{Lemma}
\newtheorem{corollary}[theorem]{Corollary}

\newtheorem{proposition}[theorem]{Proposition}

\theoremstyle{definition}

\newtheorem{example}{Example}
\newtheorem{remark}{Remark}

\newif\ifmapx
{\catcode`/=0 \catcode`\\=12/gdef/mkillslash\#1{#1}}
\edef\jobnametmp{\expandafter\string\csname ic_apx\endcsname}
\edef\jobnameapx{\expandafter\mkillslash\jobnametmp}
\edef\jobnameexpand{\jobname}
\ifx\jobnameexpand\jobnameapx
\mapxtrue
\else
\mapxfalse
\fi

\long\def\apxonly#1{\ifmapx{\color{blue}#1}\fi}

\newcommand{\co}{\mathop{\mathrm{co}}}

\begin{document}
\ifpdf
\DeclareGraphicsExtensions{.pgf}
\graphicspath{{figures/}{plots/}}
\fi

% paper title
%\title{Dualizing Le Cam's method}
\title{Dualizing Le Cam's method for functional estimation, with applications to estimating the unseens}
% alternative titles
%\title{Estimating functionals: Why Le Cam's method always works?}

\author{Yury Polyanskiy and Yihong Wu\thanks{Y.P. is with the Department of EECS, MIT, Cambridge, MA, email: \url{yp@mit.edu}. Y.W. is with
the Department of Statistics and Data Science, Yale University, New Haven, CT, email: \url{yihong.wu@yale.edu}.}}

\maketitle

\begin{abstract}

Le Cam's method (or the two-point method) is a commonly used tool for obtaining statistical lower bound and especially popular for functional estimation problems. 
This work aims to explain and give conditions for the tightness of Le Cam's lower bound in functional estimation from the perspective of convex duality.
Under a variety of settings it is shown that the maximization problem that searches for the best two-point lower bound, upon dualizing, becomes a minimization problem that optimizes the bias-variance tradeoff among a family of estimators. 
%While Le Cam's method can be used with arbitrary distance, our duality result applies specifically to the $\chi^2$-divergence, thus singling it out as a natural choice for quadratic risk. 
For estimating linear functionals of a distribution our work strengthens prior results of Donoho-Liu \cite{DL91} (for quadratic loss) by dropping the H\"olderian assumption on the modulus of continuity. For exponential families our results extend those of
Juditsky-Nemirovski \cite{JN09} by characterizing the minimax risk for the quadratic loss under weaker assumptions on the exponential family.

We also provide an extension to the high-dimensional setting for estimating separable functionals. Notably, coupled with
tools from complex analysis, this method is particularly effective for characterizing the ``elbow effect'' -- the
phase transition from parametric to nonparametric rates.
 %-- in various problems. 
As the main application of our methodology, we consider three problems in the area of ``estimating the unseens'', recovering the prior result of \cite{PSW17-colt} on population recovery and, in addition, obtaining two new ones:
\begin{itemize}
	\item \emph{Distinct elements problem}: Randomly sampling a fraction $p$ of colored balls from an urn containing $d$ balls in total, the optimal normalized estimation error of the number of distinct colors in the urn is within logarithmic factors of $d^{-\frac{1}{2}\min\{\frac{p}{1-p},1\}}$, exhibiting an elbow at $p=\frac{1}{2}$;
	
	\item \emph{Fisher's species problem}: Given $n$ independent observations drawn from an unknown distribution, the optimal normalized prediction error of the number of unseen symbols in the next (unobserved) $r \cdot n$ observations is within logarithmic factors of $n^{-\min\{\frac{1}{r+1},\frac{1}{2}\}}$, exhibiting an elbow at $r=1$.
\end{itemize}

\end{abstract}

%We provide simultaneous strengthening of results of Donoho-Liu and Juditsky-Nemirovski. Applications are also cool.
%We provide an extension of results of Juditsky-Nemirovski for non-parametric estimation in exponential families by
%characterizing minimax risk for a stronger loss (quadratic) and under weaker assumptions on the family.

\newpage

\tableofcontents

\newpage

\section{Introduction}
	\label{sec:intro}

%\subsection{Le Cam’s method and main results}
	%\label{sec:main}
%One of the most commonly used tools for statistical lower bound is \emph{Le Cam's method} (or the two-point method) \cite{LeCam73,Yu97}.
\emph{Le Cam's method} (or the two-point method) is a commonly used tool for obtaining statistical lower bound \cite{LeCam73,Yu97}.
The rationale is that if two hypotheses are statistically indistinguishable, then the difference of their parameters presents a lower bound to the accuracy of any estimator.
Although Le Cam's method can be loose in estimating high-dimensional parameters as it may not capture the correct dependency on the dimension, 
for estimating scalar-valued functionals it often yields the tight minimax rate even when the underlying parameter is high-dimensional.
%This work aims to explain the tightness of Le Cam's lower bound in certain settings from the perspective of convex duality.
This work aims to explain and give conditions for the tightness of Le Cam's lower bound in functional estimation from the perspective of convex duality.

Let $\Theta$ and $\matx$ be measurable spaces and $P$ a transition probability kernel from $\Theta$ to $\matx$. 
Then $\{P_\theta=P(\cdot|\theta)\colon \theta\in \Theta\}$ is a parametric family of distributions on $\calX$. 
Let $\Pi$ be a given subset of $\matp(\Theta)$, the set of all probability measures on $\Theta$. 
Let $T(\pi)$ be a real-valued affine functional of $\pi \in \Pi$. 
Denote the observations by $\bX=(X_1,\ldots,X_n)$ and the latent parameters $\btheta=(\theta_1,\ldots,\theta_n)$, where 
conditioned on $\btheta$, $X_i$'s are independent and distributed as
\begin{equation}
X_i \inddistr P_{\theta_i}, \quad i=1,\ldots,n
\label{eq:model-main}
\end{equation}
 Furthermore, we shall focus on the following settings, which are commonly assumed in empirical Bayes and compound estimation problems, respectively \cite{Robbins51,zhang1997empirical,zhang2003compound}.
\begin{itemize}
	%\item Empirical Bayes:
	\item 
	%Iid setting: we have 
	$\theta_i \iiddistr \pi$ for some $\pi \in \Pi$.
	In this case, $X_i \iiddistr \pi P$, where $\pi P \triangleq \int_{\Theta} P_\theta \pi(d\theta)$ denotes the mixture distribution induced by the mixing distribution (prior) $\pi$.
	Given $X_1,\ldots,X_n$, the goal is to estimate the functional $T(\pi)$ or $\pi$ itself.
	%We will also be interested in estimating the prior $\pi$ itself with respect to seminorm loss of the form:

	\item 
	%Compound estimation: 
	%Compound setting: 
	$\theta_i$'s are deterministic whose empirical distribution $\pi_{\btheta} \triangleq  \frac{1}{n} \sum_{i=1}^n \delta_{\theta_i}$ belongs to $\Pi$.
		Given $X_1,\ldots,X_n$, the goal is to estimate the functional $T(\pi_{\btheta})$ or the empirical distribution $\pi_{\btheta}$ itself.
		For affine functional $T(\pi)=\int h(\theta) \pi(d\theta)$, $T(\pi_{\btheta})= \frac{1}{n}\sum_{i=1} h(\theta_i)$ has a separable form and can be non-linear in the parameter $\btheta$.
	
\end{itemize}
%Following \cite{zhang1997empirical,zhang2003compound}, we shall refer to the above settings as the \emph{empirical Bayes} (EB) and \emph{compound estimation} respectively. 
%In Robbins' 1951 paper that set forth the empirical Bayes formalism, these correspond to his Problem (II) and Problem (I) respectively \cite[p.~146]{Robbins51}.
The minimax quadratic risks of estimating the functional $T$ in the iid and deterministic setting are defined respectively as:
	\begin{align}
	\REB(n) \eqdef & ~ \inf_{\hat T} \sup_{\pi \in \Pi} \EE[|\hat T(X_1,\ldots,X_n) -
	T(\pi)|^2]\,, \label{eq:REB}	\\
	\RComp(n) \eqdef & ~ 	\inf_{\hat T} \sup_{\btheta: \pi_{\btheta} \in \Pi} \EE[|\hat T(X_1,\ldots,X_n) -
	T(\pi_{\btheta})|^2]\,.\label{eq:RComp}
	\end{align}

%Alternatively, we can view $P$ as a transition probability kernel from $\Theta$ to $\matx$. 
%When $P$ is the identity kernel, $X_i$'s are simply drawn from $\pi$; otherwise, they are indirect observations.

The main result of this paper is the following: 
Under appropriate technical conditions such as the convexity and compactness of the space $\Pi$, for affine functional $T$, 
 Le Cam's lower bound is tight up to universal constant factors for both iid and deterministic settings.
More precisely, we have
\begin{equation}
\REB(n) \asymp \RComp(n) \asymp \max_{\theta,\theta'\in\Theta} \sth{|T(\pi)-T(\pi')|^2: \chi^2(\pi P \| \pi'P ) \leq \frac{1}{n}},
\label{eq:mainresult-intro}
\end{equation}
where $\chi^2(\cdot\|\cdot)$ denotes the $\chi^2$-divergence. 
In the iid setting, this result strengthens the celebrated result of Donoho-Liu \cite{DL91} for linear functionals.
In addition, we show a counterpart of this characterization also holds for exponential families for estimating functionals linear in the mean parameters, 
where the $\chi^2$-divergence in \prettyref{eq:mainresult-intro} is replaced by the squared Hellinger distance,
%\yp{(propose to remove this, or replace Jeffrey's with Hellinger at least: where the $\chi^2$-divergence in \prettyref{eq:mainresult-intro} is replaced by the
%Jeffrey's divergence,)} 
extending the result of Juditsky-Nemirovski
\cite{JN09} to the quadratic risk and relaxing the assumptions. 
See \prettyref{sec:related} for more discussion.
Throughout the paper we focus on the expected risk under the quadratic loss (for which the $\chi^2$-divergence is a natural choice). As will be shown later, these results can be easily extended to high-probability risk bounds.

%To explain the result \prettyref{eq:mainresult-intro}, we start from the lower bound part. 
We now explain the intuition behind the main result \prettyref{eq:mainresult-intro}. 
In the iid setting where $X_i \iiddistr \pi P$, the lower bound is a straightforward application of Le Cam's two point method, since $\chi^2(\pi P \| \pi'P ) \leq \frac{1}{n}$ implies that the total variation of the product distributions  $(\pi P)^{\otimes n}$ and $(\pi' P)^{\otimes n}$
is bounded away from one and thus impossible to be tested reliably given the sample $X_1,\ldots,X_n$, leading to a lower bound on the order of $|T(\pi)-T(\pi')|^2$.
In the deterministic setting, the lower bound is shown by a generalized version of Le Cam's method using two priors (also known as fuzzy hypotheses testing \cite[Sec.~2.7.4]{Tsybakov09}), where we consider $\theta_i$'s drawn from the product distribution $\pi^{\otimes n}$ or $\pi'^{\otimes n}$ with appropriate truncation.
%we consider $\theta_i \iiddistr \pi$ or $\pi'$ in conjunction with appropriate truncation.

What is more surprising is perhaps the upper bound in \prettyref{eq:mainresult-intro}, which in fact holds \emph{without additional constant factors} in both the iid and deterministic settings. 
The key observation is that the maximization in \prettyref{eq:mainresult-intro} is a convex optimization problem, whose dual corresponds to a minimization problem that optimizes the bias-variance tradeoff of estimators of the following type:
\begin{equation}
\hat T = \frac{1}{n} \sum_{i=1}^n g(X_i).
\label{eq:hatT-emp}
\end{equation}
In the absence of a duality gap, this shows the achievability of Le Cam's lower bound
 by optimizing the choice of $g$.

The duality view underlying the main result \prettyref{eq:mainresult-intro} is in fact natural. Indeed, the classical minimax theorem in decision theory states that, under regularity assumptions (cf.~e.g.~\cite[Theorem 46.5]{Strasser85}) the minimax risk and the least favorable Bayes risk coincide, namely
\begin{equation}
%\overbrace{\inf_{\hat{T}}\sup_{\theta \in \Theta} \Expect_\theta[( \hat{T}-T(\theta))^2] }^{\text{minimax risk}}
%\xlongequal[\scriptsize \text{theorem}]{\scriptsize \text{minimax}}
%\sup_{\pi \in \calP(\Theta)}  \overbrace{\inf_{\hat{T}} \Expect_{\theta \sim \pi}[( \hat{T}-T(\theta) )^2]}^{\text{Bayes risk}}.
\inf_{\hat{T}}\sup_{\theta \in \Theta} \Expect_\theta[( \hat{T}-T(\theta))^2] 
=
\sup_{\pi} \inf_{\hat{T}} \Expect_{\theta \sim \pi}[( \hat{T}-T(\theta) )^2],
\label{eq:minimax}
\end{equation}
where the supremum on the right is taken over all priors on $\Theta$.
This can also be interpreted from the duality perspective,\footnote{This follows from standard arguments in optimization by rewriting the left-hand side as $\inf_{\hat{T}} \{t: \Expect_\theta[( \hat{T}-T(\theta))^2] \leq t, \forall \theta \in \Theta \}$ and the Lagrange multipliers correspond to priors. When both $X$ and $\theta$ are finitely-valued, \prettyref{eq:minimax} is simply the duality of linear programming (LP).}
%the dual program obtained by moving $\sup_{\theta \in \Theta}$ into constraints and the priors correspond to the Lagrange multipliers
where the primal variables corresponds to (randomized) estimators and the dual variables correspond to priors.
However, the duality view of \prettyref{eq:minimax} is unwieldy except in special cases, due to the difficulty of finding the least favorable prior that maximizes the Bayes risk for large sample size.
In this vein, the more effective result of \prettyref{eq:mainresult-intro} can be viewed as approximate version of the general minimax theorem for functional estimation.
%Other work that operationalized similar duality perspectives for statistical estimation include 
%\cite{PSW17-colt}, which recognized the LP duality between the risk of the optimal linear estimator and the best Le Cam lower bound based on the total variation.
%In addition, the duality between best polynomial approximation and moment matching was leveraged in \cite{WY14,WY15,JVHW15} 
	%for estimating symmetric functionals, such as the Shannon entropy and support size, of distributions supported on large domains.

To produce concrete results of minimax rate for specific applications, one needs to evaluate the value of the maximum in \prettyref{eq:mainresult-intro}. 
Using tools from complex analysis, we do so for a number of problems and obtain new results on the optimal rate of convergence, characterizing, in particular, the ``elbow effect'', that is, the phase transition from parametric to nonparametric rates. 
As the main application of our methodology, we consider three problems in the area of ``estimating the unseens'', namely, \emph{population recovery}, \emph{distinct elements problem}, and \emph{Fisher's species problem}. In addition to recovering the prior result of \cite{PSW17-colt} on the optimal rate of population recovery, we establish the following new results:
\begin{itemize}
%\item Population recovery: Observing the erased version of $n$ strings drawn from a distribution $\pi$ on $\{0,1\}^n$, where each coordinate is erased independently with probability $\epsilon$, the optimal sup-norm estimation error of $\pi$ is with logarithmic factors of $n^{-\frac{1}{2}\min(1, {1-\epsilon\over\epsilon})}$, exhibiting an elbow at $p=\frac{1}{2}$;

	\item Distinct elements problem: Randomly sampling a fraction $p$ of colored balls from an urn containing $n$ balls in total, the goal is to estimate the number of distinct colors in the urn \cite{RRSS09,Valiant11,WY2016sample}. We show that, as $n\to\infty$,
 the optimal normalized estimation error is within logarithmic factors of $n^{-\frac{1}{2}\min\{\frac{p}{1-p},1\}}$, exhibiting an elbow at $p=\frac{1}{2}$;
	
	\item Fisher's species problem: Given $n$ independent observations drawn from an unknown distribution, the goal is to predict the number of unseen symbols in the next (unobserved) $r \cdot n$ samples \cite{FCW43,ET76,OSW15}. We show that, as $n\diverge$,	
	the optimal normalized prediction error is within logarithmic factors of $n^{-\min\{\frac{1}{r+1},\frac{1}{2}\}}$, exhibiting an elbow at $r=1$.
\end{itemize}

We emphasize that the main focus of this paper to determine the minimax rate by means of convex duality 
%such as \prettyref{eq:mainresult-intro} 
without demonstrating an explicit choice of the optimal estimator.
%We emphasize that in obtaining the above results, 
%As such, we do not demonstrate an explicit choice of the optimal estimator;
%instead, capitalizing on the duality between the minimization problem over the linear estimators and the maximization
%that produces the best Le Cam lower bound, we bound the value of the dual problem from above, thereby showing the
%achievability of the optimal rates. 
This is conceptually distinct from existing explicit construction of 
estimators such as 
%kernel density estimator \cite{Tsybakov09} or 
smoothed estimators in the context of the species problem \cite{OSW15} (which do not attain the optimal rate). 
Nevertheless, since the minimax rate in both iid and deterministic settings can be achieved by the estimator \prettyref{eq:hatT-emp} parameterized by $g$, 
the optimal choice of $g$ corresponds to the optimizer of certain convex optimization problem (cf.~\prettyref{eq:deltach_def}), which can be solved efficiently in the case of finite $\Theta$ and $\calX$ (e.g.~the population recovery problem). Even for continuous models, this infinite-dimensional problem can often be effectively discretized leading to computational efficient construction of optimal estimators. For example, for the distinct elements and species problems, the estimators can be constructed in polynomial time as solutions to certain linear programs (see \prettyref{eq:LP-de} and \prettyref{eq:species-bvLP} respectively).
 %with polynomially many variables and constraints.
%even for continuous models, the infinite-dimensional LP can be effectively ``finite-dimensionalized'' leading to computational efficient construction of optimal estimators (see Theorems \ref{thm:de} and \ref{thm:species} for examples).

%Introduce two programs: TV version and $\chi^2$-version. The main result is phrased in terms of the $\chi^2$-program which determine the minimax rate without any assumption. On the other hand, the dual for the TV program does not automatically yield a valid lower bound. 
%Nevertheless, we still discuss the TV program for two reasons:
%(a) it is easier to dualize
%(b) it is frequently easier to compute e.g.~by means of complex analysis techniques. Furthermore, a good approach to the $\chi^2$-program is to plug in the optimal dual solution from the TV program, which frequently turn out to be optimal as well.

%Before proceeding to the discussion of the related literature, 
Before discussing the related literature, 
let us mention that the duality-based method in this paper need not be limited to functional estimation. In a companion paper \cite{JPW20} we extend the methods to estimating the distribution itself (with respect to the total variation loss) in the context of the distinct elements problem. The connection to functional estimation is that 
%the total variation distance can be written as the supremum over all bounded linear functionals and 
estimating the distribution in total variation is equivalent to simultaneously estimating all bounded linear functionals; 
this view enables us to analyze Wolfowitz's \emph{minimum-distance estimators} \cite{Wolfowitz57} in the duality framework.

\subsection{Related work}
\label{sec:related}

A celebrated result of Donoho-Liu~\cite{DL91} relates the minimax rate of estimating linear functionals to the Hellinger
modulus of continuity.  For the density estimation models, under certain assumptions, it is shown that the minimax rate
coincides with the right-hand side of \prettyref{eq:mainresult-intro} with $H^2$ in place of the
$\chi^2$-divergence.\footnote{The resulting moduli of continuity are in fact the same up to constant factors, as we
show in \prettyref{prop:deltaproperty}.} However, the constant factors may not be universal and depend on the
problem or its hyper-parameters, thus
precluding the application to high-dimensional problems. More importantly, the proof (of the upper bound) in \cite{DL91}
is based on constructing an estimator via pairwise hypotheses tests, by means of a binary search on the functional
value. While this method can deal with general loss function, the limitation is that it assumes the H\"olderianity of
the modulus of continuity in order to show tightness.  We refer the readers to \prettyref{sec:dl_compare} for a detailed
comparison of the results.

The prior work that is closest to ours in spirit is that of Juditsky-Nemirovski \cite{JN09} (cf.~also the recent monograph \cite{JNbook}), where the main technology was also convex optimization and the minimax theorem.
As opposed to the squared loss, they considered the $\epsilon$-quantile loss, namely, an upper bound on the estimation accuracy that holds with probability at least $1-\epsilon$ for all parameters.
%$$ R^*_{n,\epsilon} = \inf_{\hat T} \inf\sth{r: \sup_{\theta \in \Theta} P_\theta[|\hat T - T(\theta)| > r] \le \epsilon}.$$
For exponential families, under certain convexity assumptions, it is shown that the minimax $\epsilon$-quantile risk is determined within absolute constant factors by the Hellinger modulus of continuity provided that $\epsilon$ is not too small.
%that is, replacing $\chi^2$ by $H^2$ on the RHS of \prettyref{eq:LC-intro-chi2}, 
%provided that $ \exp(-o(n)) \leq \epsilon \leq \frac{1}{4}$.
We extend this result to quadratic risk under more relaxed assumptions (see \prettyref{sec:jn_compare} for details).
%Note that the quadratic risk result cannot be obtained through the usual route of integrating the
%high-probability risk bound, since the optimal estimator for an $\epsilon$-quantile loss potentially depends on $\epsilon$. On the other
%hand, results on $\epsilon$-quantile loss in \cite{JN09} are not implied by the quadratic risk
%result in this paper except for constant $\epsilon$ (by applying the Markov inequality).
Note that the quadratic risk result cannot be obtained through the usual route of integrating the
high-probability risk bound, since the optimal estimator for an $\epsilon$-quantile loss potentially depends on $\epsilon$. On the other
hand, results on $\epsilon$-quantile loss can be obtained from the quadratic risk by sampling splitting and applying the median (although the more direct argument in \cite{JN09} achieves better constant factors).
Nevertheless, the main advantage of our approach lies in its versatility, as
witnessed, e.g., by the treatment of the deterministic setting.
%, something we are not sure is possible with the methods of~\cite{JN09}.

Finally, let us mention that while the main objective of \cite{JN09} was to obtain rate-optimal estimators by means of convex programming, in this paper and similar to the program in \cite{DL91}, the characterization of the minimax risk by convex optimization is mostly used as a mathematical tool for determining the minimax rates, although it also leads to efficient construction of estimations for specific problems.

\subsection{Organization}

%The rest of the paper is organized as follows.
%\prettyref{sec:linear} presents the main result for estimating linear functionals of a distribution (with possibly indirect observations) under a general setup.
%%, with applications to population recovery and density estimation given in Sections \ref{sec:poprec} and \ref{sec:density}.
%We provide two examples: population recovery (\prettyref{sec:poprec}) and density estimation (\prettyref{sec:density}), which are finite-dimensional and infinite-dimensional application of the main theorem respectively.
%\prettyref{sec:hd} extend the result to estimating separable functions in the deterministic setting. The methods are then applied to the distinct elements problem (\prettyref{sec:de}) and Fisher's species extrapolation problem (\prettyref{sec:species}) to determine the minimax rates of convergence up to logarithmic factors.
%Finally, in \prettyref{sec:exp} we extend the result for exponential families under weaker assumptions than those in \cite{JN09}. 
%%To keep the discussion simple, we focus on finite-dimensional models.
%To present a simple motivating example and to exhibit the duality perspective in a familiar problem, in \prettyref{sec:whitenoise} we revisit the
%classical Gaussian white noise model and re-derive the classical result of Ibragimov and
%Has'minskii~\cite{IH84}. For readers unfamiliar with this type of argument, it might be helpful to \textit{start} with
%\prettyref{sec:whitenoise}.

The rest of the paper is organized as follows.
\prettyref{sec:linear} presents the main result for the iid setting. 
We provide two examples: population recovery (\prettyref{sec:poprec}) and interval censoring (\prettyref{sec:ic}), which are finite-dimensional and infinite-dimensional application of the main theorem respectively.
\prettyref{sec:hd} extend the result to the deterministic setting for estimating separable functionals. The methods are then applied to the distinct elements problem (\prettyref{sec:de}) and Fisher's species extrapolation problem (\prettyref{sec:species}) to determine the minimax rates of convergence up to logarithmic factors.
Finally, in \prettyref{sec:exp} we extend the result to exponential families. 

\ifthenelse{\boolean{aos}}{
Due to space constraints, the proofs of \prettyref{prop:deltaproperty} and Theorems \ref{thm:de}--\ref{th:exp} are provided in \prettyref{seca:pf} of the supplementary material \cite{dual2-supp};  further technical results and proofs are collected in Sections \ref{app:lmm} and \ref{app:species-poisson} of \cite{dual2-supp}.
\prettyref{app:classical} discusses applications to classical problems of density estimation and the Gaussian white noise model, the latter constituting a simple example where the dualization can be carried out explicitly.
}{
\prettyref{seca:pf} contains the proofs of 
%Propositions \ref{prop:deltaproperty}, \ref{prop:} and 
Theorems \ref{thm:de}--\ref{th:exp}; further technical results and proofs are collected in Appendices \ref{app:lmm} and \ref{app:species-poisson}. 
\prettyref{app:classical} discusses applications to classical problems of density estimation and the Gaussian white noise model, the latter constituting a simple self-contained example where the dualization of Le Cam's lower bound can be carried out explicitly. 
%\prettyref{sec:whitenoise}.
%classical Gaussian white noise model and re-derive the classical result of Ibragimov and
%Has'minskii~\cite{IH84}. For readers unfamiliar with this type of argument, it might be helpful to \textit{start} with
%\prettyref{sec:whitenoise}.
}

%\subsection{Notations}
	%\label{sec:notation}
%%$\Bern(p)$ stands for Bernoulli distribution with parameter $p$ and 
%
%Let $\Binom(n,p)$ stand for the binomial distribution with $n$ independent trials and success probability $p$.
%For any sequences $\{a_n\}$ and $\{b_n\}$ of positive numbers, we write $a_n \gtrsim b_n$ if $a_n\geq cb_n$ holds for all $n$ and some absolute constant $c > 0$, $a_n\lesssim b_n$ if $a_n \gtrsim b_n$, and $a_n \asymp b_n$ if both $a_n\gtrsim b_n$ and $a_n\lesssim b_n$ hold.
%In addition, we use $\asymp_k$ to indicate that the constant depends only on $k$.

%$\co(A)$ denotes the convex hull of a set $A$.
%
%Let $\matp(\Theta)$ denote the collection of all probability measures on a measurable space $\Theta$.

%\section{Linear functionals}\label{sec:linear}

\section{Iid setting}\label{sec:linear}

%In this section ...
%We provide two examples: population recovery (\prettyref{sec:poprec}) and density estimation (\prettyref{sec:density}), which are finite-dimensional and infinite-dimensional application of \prettyref{th:linear}, respectively.

%Let $\Theta$ and $\matx$ be measurable spaces and $P$ a transition probability kernel from $\Theta$ to $\matx$. 
%Denote by $\matp(\Theta)$ the set of all probability distributions on $\Theta$ and let $\Pi$ be a (given) subset of
%$\matp(\Theta)$. Let $T(\pi)$ be a real-valued affine functional of $\pi \in \Pi$. The minimax quadratic risk of estimating $T$
%based on $X_1,\ldots,X_n \simiid \pi P$ is defined as:
	%$$ \REB(n) \eqdef \inf_{\hat T} \sup_{\pi \in \Pi} \EE[|\hat T(X_1,\ldots,X_n) -
	%T(\pi)|^2]\,.$$
	%When $P$ is the identity kernel, $X_i$'s are simply drawn from $\pi$; otherwise, they are indirect observations.

	Recall that in the iid setting of \prettyref{eq:model-main}, the sample consists of 
$X_1,\ldots,X_n \simiid \pi P$. When $P$ is the identity kernel, 
%which is the case for density estimation problems, 
$X_i$'s are directly drawn from $\pi$; otherwise, they are indirect observations.
The minimax quadratic risk of estimating $T(\pi)$ over $\pi\in\Pi$ is denoted by $\REB(n)$ in \prettyref{eq:REB}.

Define the modulus of continuity of functional $T$ with respect to various distances (and quasi-distances)
between distributions $\pi P$:
	\begin{align} \delta_{\chi^2}(t) &= \sup\{T(\pi')-T(\pi): \chi^2(\pi'P\|\pi P)\le t^2, \pi,\pi'\in\Pi\} 
		\label{eq:delchi2_def}\\
	   \delta_{H^2}(t) &= \sup\{T(\pi')-T(\pi):  H^2(\pi'P,\pi P)\le t^2, \pi,\pi'\in\Pi\} \\
	   \delta_{\TV}(t) &= \sup\{T(\pi')-T(\pi):  \TV(\pi'P,\pi P)\le t, \pi,\pi'\in\Pi\} \label{eq:deltv_def}
	\end{align}
where $\TV(F,G) = \sup_E |F(E)-G(E)|$ is the total variation, $H^2(F,G) = \int d\nu \left(\sqrt{{dF\over d\nu}} -
\sqrt{{dG\over d\nu}}\right)^2$ is the squared Hellinger distance (with $\nu$ being any dominating measure s.t. $F\ll \nu$ and
$G\ll \nu$, e.g. $\nu=F+G$). Finally, the $\chi^2$-divergence is defined as 
$\chi^2(F\|G) = \infty$ if $F\not \ll G$ and otherwise $\chi^2(F\|G) = \int dG 
\left( {dF\over dG}\right)^2-1$. We note that $\TV(F,G)$ and $H(F,G)$ are distances on $\matp$. For a signed measure
$\mu$ its total variation norm is denoted $\|\mu\|_{\TV}$, so that $\TV(F,G)=\|F-G\|_{\TV}$.

%\subsection{Main result: Minimax rate for linear functionals}
%\label{sec:linear-main}

Our main result is the following:
\begin{theorem} \label{th:linear}
Suppose that $(\Theta,\matx,P,T, \Pi)$ satisfy the following assumptions:
\begin{enumerate}
\item[A1] The functional $\pi \mapsto T(\pi)$ is affine;
\item[A2] The set $\Pi$ is convex;
\item[A3] There exists a vector space of functions $\matf$ on $\matx$ such that $\matf$ contains constants and 
is dense in $L_2(\matx, \pi P)$ for every $\pi \in \Pi$;
\item[A4] There exists a topology on $\Pi$ such that:
	\begin{enumerate}
		\item[A4a] It is coarse enough that $\Pi$ is compact;
		\item[A4b] It is fine enough that $T(\pi)$, $\pi Pf$ and $\pi P(f^2)$ are continuous in $\pi\in \Pi$ for all $f\in \matf$.
	\end{enumerate}
\end{enumerate}
Then
	\begin{equation}\label{eq:linthm}
		{1\over (1+\sqrt{e})^2}\delta_{\chi^2}(\tfrac{1}{\sqrt{n}})^2 \le {R^*(n)} \le \delta_{\chi^2}(\tfrac{1}{\sqrt{n}})^2,
\end{equation}	
where the upper bound can be achieved by an estimator of the form $\hat T = \frac{1}{n} \sum_{i=1}^n g(X_i)$ for some $g\in\calF$.
\end{theorem}

\begin{corollary}\label{cor:linthm} There exist an absolute constant $c>0$ with the following property. In the setting
of the previous Theorem for all $e^{-2n} \le \epsilon \le
{1\over 16}$ we have 
	\begin{equation}\label{eq:clt_0}
		{1\over c} \delchi\left(\sqrt{{1\over n}\ln {1\over \epsilon}}\right) \le \REB(n,\epsilon) \le c 
			\delchi\left(\sqrt{{1\over n}\ln {1\over \epsilon}}\right)\,, 
	\end{equation}			
where the $\epsilon$-risk (confidence interval) is defined as
\begin{equation*}
	\REB(n,\epsilon) \eqdef  \inf_{\hat T} \sup_{\pi \in \Pi} \inf\{\rho: \PP[|\hat T(X_1,\ldots,X_n) -
	T(\pi)| > \rho] \le \epsilon\}\,.%\label{eq:riid_eps}
\end{equation*}
\end{corollary}

Some remarks are in order:
\begin{enumerate}
\item If $\Theta$ and $\matx$ are finite, then $\matf$ can be taken to be all functions on $\matx$ and assumptions
A3 and A4 are automatic.
\item If $\matx$ is a normal topological space, then every probability measure $\nu$ is regular~\cite[IV.6.2]{DS58} and the
set $\matf$ of all bounded continuous functions is dense in $L_2(\matx, \nu)$, cf.~\cite[IV.8.19]{DS58}.
Other convenient choices of $\matf$ are all Lipschitz functions (and Wasserstein $W_1$-convergence), all polynomials, trigonometric
polynomials or sums of exponentials.
\item The continuity of $\pi P f$ under the weak topology on $\Pi$ 
can be assured by demanding the following (strong Feller) property for kernel $P$: For any bounded measurable $f$, $Pf$ is
	bounded continuous.
\end{enumerate}

To prove \prettyref{th:linear}, we start with several general properties and comparisons of various moduli of continuity.
 %(see \prettyref{seca:pf-deltaproperty} for a proof).
\begin{proposition} 
\label{prop:deltaproperty}
Let $T(\pi)$ be affine in $\pi$. 
Assume that $\Pi$ is convex.
Then
\begin{enumerate}
	\item (Concavity) $\delta_{H^2}(\cdot)$, $\delta_{\TV}(\cdot)$ and $\delta_{\chi^2}(\sqrt{\cdot})$ are concave.
	\item (Subadditivity) For any $c\in[0,1]$ and $t\ge 0$ we have:
	\begin{align} \delta_{\TV}(ct)&\ge c\delta_{\TV}(t)\\
		   \delta_{H^2}(ct)&\ge c\delta_{H^2}(t)\label{eq:delH2_sub} \\
			\delta_{\chi^2}(ct) &\ge c^2 \delchi(t)\label{eq:delchi_sub} 
	\end{align}			
\item (Comparison of various $\delta$'s) For all $t\ge 0$ we have
\begin{equation}\label{eq:delta_all}
		{1\over 2}\delta_{H^2}(t) \le \delta_{\chi^2}(t) \le \delta_{H^2}(t) \le \delta_{\TV}(t) \le
	\delta_{H^2}(\sqrt{2t})\,. 
\end{equation}
\item (Superlinearity) Let $\Delta_{\max} \eqdef \sup\{T(\pi')-T(\pi):  \pi,\pi'\in\Pi\}$, then
\begin{equation}\label{eq:delta_superl}
		\delta_{H^2}(t) \ge \Delta_{\max} {t\over \sqrt{2}}\,.
\end{equation}	
\end{enumerate}
\end{proposition}

\begin{proof}%[Proof of \prettyref{prop:deltaproperty}]
%The first property follows from the fact that $\TV(P,Q)$ and $H(P,Q)$ are both convex in the pair $(P,Q)$ (in fact they are distances). The second one for TV
%and $H^2$ follows from the first and the fact that $\delta(0)=0$, while for $\chi^2$ it follows from the convexity of $(P,Q)\mapsto \chi^2(P\|Q)$ and hence the concavity of $s \mapsto\delchi(\sqrt{s})$. 
The first property follows from the convexity of $\TV(P,Q)$, $H(P,Q)$ and $\chi^2(P\|Q)$ in the pair $(P,Q)$. The second one follows from the first and the fact that $\delta(0)=0$.
For the third, we recall standard bounds (cf.~e.g.~\cite[Sec.~2.4.1]{Tsybakov09}): For any pair of distributions $P,Q$ we have
	\begin{equation}\label{eq:helTV}
		H^2(P,Q)/2 \le \TV(P,Q) \le H(P,Q)\,,
	\end{equation}	
	and
	\begin{equation}\label{eq:helchi}
		H^2(P,Q) \le 2 - {2\over \sqrt{1+\chi^2(P\|Q)}} \le \chi^2(P\|Q)\,.  
	\end{equation}
	Together~\eqref{eq:helTV} and~\eqref{eq:helchi} establish all inequalities in~\eqref{eq:delta_all} except the left-most
	one. For the latter we recall from~\cite[p.~48]{Lecam86}:
	\begin{equation}\label{eq:helLe Cam}
		{1\over 2}H^2(P,Q) \le \chi^2\pth{P\Big\|{P+Q\over2}} \le H^2(P,Q)\,.
	\end{equation}	
	%(Another useful fact is that $\chi(P\|{P+Q\over2})$ is a distance).
	Thus, for any $(\pi,\pi')$ that are feasible for the $\delta_{H^2}(t)$ problem, $\pi_0 \eqdef
	{\pi+\pi'\over 2}$ and $\pi_0'\eqdef \pi'$ are feasible for the $\delchi(t)$ problem,
	since $\chi^2(\pi_0' P\| \pi_0 P)\le t^2$ according to~\eqref{eq:helLe Cam}, and satisfy 
	$|T(\pi_0)-T(\pi_0')| = {1\over 2} |T(\pi)-T(\pi')| $.

Finally, \eqref{eq:delta_superl} follows from \prettyref{eq:delH2_sub} and the observation that $\delta_{H^2}(\sqrt{2}) = \Delta_{\max}$ since
$H^2 \le 2$ by definition.
%\nb{According to the def, should be $\delta_{H^2}(\sqrt{2})$, so should be $\Delta_{\max} {t\over \sqrt{2}}$ in \prettyref{eq:delta_superl}.}
\end{proof}

\begin{proof}[Proof of \prettyref{th:linear}]
The lower bound simply follows from the $\chi^2$-version of Le Cam's method. 
%Consider a pair of distributions $\pi,\pi'$ such that $\chi^2(\pi'P\| \pi P) \le {a\over n}$ for some $a\in(0,1)$ to be optimized. 
Consider a pair of distributions $\pi,\pi'$ such that $\chi^2(\pi'P\| \pi P) \le {1\over n}$ to be optimized. From the tensorization property of $\chi^2$-divergence we have
\begin{equation}\label{eq:tv_chi2_tenso}
		\chi^2((\pi'P)^{\otimes n}\| (\pi P)^{\otimes n}) = (1+\chi^2(\pi'P\| (\pi P)))^n-1 \le e -1.
\end{equation}		
Using Brown-Low's two-point lower bound \cite{BL96} and 
optimizing over the pair $\pi,\pi'$, we have
\begin{equation}
	\REB(n) \geq 
	\sup_{\pi,\pi'\in\Pi: \chi^2(\pi'P\| \pi P)
\le {1\over n}}	
	\frac{(T(\pi)-T(\pi'))^2}{\pth{1+\sqrt{1+\chi^2((\pi'P)^{\otimes n}\| (\pi P)^{\otimes n})	} }^2} \geq \frac{\delchi\pth{\frac{1}{\sqrt{n}}}^2}{(1+\sqrt{e})^2 }.
	\end{equation}
	%Using \prettyref{eq:delchi_sub} and optimizing over $0<a<1$, we obtain
%\begin{equation}
	%\REB(n) \geq 
	%\delchi\pth{ \frac{1}{\sqrt{n}} }^2  \underbrace{\sup_{0<a<1} \frac{a}{\pth{1+e^{a/2}}^2 }}_{= 1/(1+\sqrt{e})^2}.
	%\end{equation}
	
To prove an upper bound we consider estimators of the form
\begin{equation}\label{eq:emp_est}
		\hat T_g = {1\over n} \sum_{i=1}^n g(X_i)\,,
\end{equation}	
where $g\in \matf$.
 %and $\bX = (X_1,\ldots,X_n)$. 
We analyze the quadratic risk of this estimator by decomposing it into bias and variance part:
\begin{equation}\label{eq:emp_bv}
	\EE_{X_i \simiid \pi P} [|\hat T_g - T(\pi)|^2] \le {1\over n} \Var_{\pi
	P}[g] + |T(\pi) - \pi Pg|^2.
\end{equation}	
Taking worst-case $\pi$ and optimizing over $g$ we get
	$$ \sqrt{\REB(n)} \le \inf_{g\in\calF} \sup_{\pi \in \Pi}
		\left\{{1\over \sqrt{n}} \sqrt{\Var_{\pi P}[g]} +
		|T(\pi) - \pi Pg|\right\} = \deltaa(\tfrac{1}{\sqrt{n}})\,,$$ 
where 
	\begin{equation}\label{eq:deltach_def}
		\deltaa(t) \eqdef \inf_{g\in\calF} \sup_{\pi\in\Pi} 
		\left\{t \sqrt{\Var_{\pi P}[g]} +
		|T(\pi) - \pi Pg|\right\}.
\end{equation}		
The proof is completed by applying the next proposition.
\end{proof}

\begin{proposition}\label{prop:delminmax} Under the conditions of Theorem~\ref{th:linear}, we have
\begin{equation}
\deltaa(t) \le \delchi(t) \qquad \forall t\ge 0.
\label{eq:delachi}
\end{equation}
Furthermore, the supremum over $\pi,\pi'$ in the definition of $\delchi$ is achieved: There exist $\pi_*, \pi'_* \in
\Pi$ s.t. $ \delchi(t) = T(\pi'_*) - T(\pi_*)$ and $\chi^2(\pi'_* P \| \pi_* P) \le t^2$.
\end{proposition}

Before proving the proposition, we recall the minimax theorem due to Ky Fan \cite[Theorem 2]{Fan53}:\footnote{There it is
stated for Hausdorff $X$, but this condition is not necessary, e.g.,~\cite{borwein1986fan}. Note that in defining
convex-concave-like property we mandate it hold for all $0\le t \le 1$ in~\eqref{eq:concon}, but it is also known that
minimax theorem holds for functions that only satisfy, e.g., $t=1/2$, see~\cite{konig1968neumannsche}.}
\begin{theorem}[Ky Fan]
\label{thm:minimax}
	Let $X$ be a compact space and $Y$ an arbitrary
set (not topologized). Let $f: X\times Y \to \reals$ be such that
for every $y \in Y$, $x\mapsto f(x, y)$ is upper semicontinuous on $X$. If $f$ is concave-convex-like on $X\times Y$,
then
\[
\max_{x\in X} \inf_{y \in Y} f(x,y) = \inf_{y \in Y} \max_{x\in X} f(x,y).
\]
\end{theorem}
We recall that the function $f$ is concave-convex-like on $X\times Y$ if a) for any two $x_1,x_2\in X$ and $\lambda\in[0,1]$
there exists $x_3 \in X$ such that for all $y\in Y$:
\begin{equation}\label{eq:concon}
		\lambda f(x_1,y) + (1-\lambda) f(x_2,y) \le f(x_3,y)  
\end{equation}	
and b) for any two $y_1,y_2\in Y$ and $\lambda\in[0,1]$
there exists $y_3 \in Y$ such that for all $x\in X$:
$$ 
		\lambda f(x,y_1) + (1-\lambda) f(x,y_2) \ge f(x,y_3)  .
$$

\begin{proof}[Proof of Proposition~\ref{prop:delminmax}]
We aim to apply the minimax theorem in order to get a more convenient expression for $\deltaa(t)$. The function
$$ (\pi, g) \mapsto \sqrt{\Var_{\pi P}[g]} + |T(\pi) - \pi Pg| $$
satisfies all the conditions for applying \prettyref{thm:minimax}
except for the concavity in $\pi$ due to the last term (it is convex instead of concave). To mend this consider the following upper bound
	$$ |T(\pi) - \pi Pg| \le \sup_{\xi \in [0,2], \pi' \in \Pi}  T(\pi)-\pi Pg - \xi(T(\pi')-\pi'
	Pg)\,. $$
	Indeed, if $T(\pi) - \pi Pg > 0$, take $\xi=0$; otherwise, take $\pi'=\pi,\xi=2$.

So letting $u=(\pi,\pi',\xi)\in U \eqdef \Pi \times \Pi \times [0,2]$ we consider the following function on $U\times \matf$:
	$$ F_t(u,g) \eqdef T(\pi)-\pi Pg - \xi(T(\pi')-\pi' Pg) + t \sqrt{\Var_{\pi P}[g]}  $$
We claim it is concave-convex-like. Convexity in $g$ is easy: the term $|T(\pi) - \pi Pg|$ is clearly convex, whereas the convexity of $g \mapsto \sqrt{\Var_\mu[g]}$ follows from observation that without loss of generality we may assume $\EE_\mu[g]=0$ and then $\sqrt{\Var_\mu[g]} = \sqrt{\int g^2 d\mu} \triangleq \|g\|_{L_2(\mu)}$ is a norm (hence convex).

We proceed to checking the concave-like property of $F_t(u,g)$ in $u$. Define for convenience, 
	$$ a(\pi) \eqdef T(\pi) - \pi Pg, \qquad b(\pi) = t\sqrt{\Var_{\pi P}[g]} $$
It is clear that $a(\pi)$ is affine, whereas $b(\pi)$ is concave. Indeed,  $\sqrt{\cdot}$ is a concave and increasing
scalar function, whereas $\Var_\mu[g] = \mu(g^2) - (\mu g)^2$ is concave in $\mu$. So for $u=(\pi,\pi',\xi)$ we have
\begin{equation}\label{eq:ft_eq}
		F_t(u,g) = a(\pi) - \xi a(\pi') + b(\pi)\,.
\end{equation}	

Consider $u_1 = (\pi_1, \pi_1', \xi_1)$ and $u_2 = (\pi_2, \pi_2', \xi_2)$ and $\lambda \in [0,1]$. First, suppose that
$\xi_1=\xi_2=0$. We see that in this case
	$$ \lambda F_t(u_1,g)  + (1-\lambda) F_t(u_2,g) \le F_t(\lambda u_1 + (1-\lambda)u_2, g) $$
since from~\eqref{eq:ft_eq} we see that $F_t$ is concave in $\pi$. Then, taking $u_3 = \lambda u_1 + (1-\lambda) u_2$
satisfies~\eqref{eq:concon}. Next, assume that either $\xi_1 >0$ or $\xi_2 > 0$. Then define 
$$ \pi_3 \eqdef \lambda \pi_1 + (1-\lambda) \pi_2\,,\quad \pi_3' \eqdef {\lambda\xi_1\over \xi_3} \pi_1' + {(1-\lambda)\xi_2\over
\xi_3} \pi_2', \quad \xi_3 \eqdef \lambda \xi_1 + (1-\lambda) \xi_2 \,.$$
And set $u_3 = (\pi_3, \pi_3', \xi_3)$. We claim that 
\begin{equation}\label{eq:ft_cc}
		\lambda F_t(u_1,g)  + (1-\lambda) F_t(u_2,g) \le F_t(u_3, g)\,. 
\end{equation}	
Indeed, we have from affinity of $a(\cdot)$:
	$$ a\left({\lambda \xi_1\over \xi_3} \pi_1' + {(1-\lambda) \xi_2\over \xi_3} \pi_2' \right) = 
		{\lambda \xi_1\over \xi_3} a(\pi_1') + {(1-\lambda) \xi_2\over \xi_3} a(\pi_2')\,.$$
Therefore, we have
\begin{align*} \lambda a(\pi_1) + (1-\lambda) a(\pi_2) &= a(\pi_3)\\
	   \lambda \xi_1 a(\pi_1') + (1-\lambda) \xi_2 a(\pi_2') &= \xi_3 a(\pi_3')\\
	   \lambda b(\pi_1) + (1-\lambda) b(\pi_2) &\le b(\pi_3)\,.
\end{align*}
These three statements together with~\eqref{eq:ft_eq} prove~\eqref{eq:ft_cc}.

Knowing that $F_t$ is concave-convex-like, for applying the minimax theorem we only need to check that $u \mapsto F_t(u,g)$ is continuous for all $g$ and that $U$
is compact. This is satisfied by the assumption $A4$ of Theorem~\ref{th:linear}. Applying \prettyref{thm:minimax}, we have 
\begin{equation}\label{eq:ft_i2}
		\deltaa(t) \le \inf_{g\in \matf} \sup_{u \in U} F_t(u,g) = \inf_{g\in \matf} \max_{u \in U} F_t(u,g) =
	\max_{u \in U} \inf_{g\in \matf}  F_t(u,g)\,.
\end{equation}	
(Note that the rightmost maximum exists thanks to \prettyref{thm:minimax}.)

Next, to evaluate the rightmost term, fix $u=(\pi,\pi',\xi)\in U$ and consider the optimization
\begin{equation}\label{eq:ft_psidef}
		\psi_t(u) = \inf_{g \in \matf} (\xi \pi'  -\pi) Pg  +  t \sqrt{\Var_{\pi P}[g]}\,.
\end{equation}	
We claim that 
\begin{equation}\label{eq:ft_psi}
		\psi_t(u) = \begin{cases} -\infty, &\qquad  \xi \neq 1\\
				   -\infty, & \quad \xi = 1, \chi^2(\pi' P\| \pi P) > t^2\\
				   0, & \mbox{otherwise}
			\end{cases} 
\end{equation}
which implies the desired \prettyref{eq:delachi} by continuing~\eqref{eq:ft_i2}:
	$$ \max_{u \in U} \inf_{g\in \matf}  F_t(u,g) = \max \{T(\pi') - T(\pi): \chi^2(\pi' P\| \pi P) \le t^2, \pi \in
	\Pi, \pi' \in \Pi\}\,.$$

To prove~\eqref{eq:ft_psi}, we first recall that $\matf$ contains constants. Thus if $\xi \neq 1$, we have that the
first term in~\eqref{eq:ft_psidef} can be driven to $-\infty$, while keeping the second term zero,  by taking $g= c 1$
and $c\to \pm\infty$. So fix $\xi = 1$. Recall a variational characterization of the $\chi^2$-divergence:\footnote{For completeness,
here is a short proof of~\eqref{eq:chi_va}. First, assume $\chi^2(\mu\|\nu)<\infty$. Denoting $f={d\mu\over d\nu}$ and assuming without loss of
generality that $\EE_\nu g=0$ we have $|\EE_{\mu}[g] - \EE_{\nu} [g]|^2 = (\EE_\nu [fg])^2 \le \Var_{\nu}g \Var_\nu f$,
which completes the proof since $\Var_\nu f = \chi^2(\mu\|\nu)$. For the other direction, simply approximate $f$ by
elements of $\mathcal{G}$. If $\chi^2(\mu\|\nu)=\infty$, set $f_n=\min(f,n)$ and let $n\to\infty$.}
\begin{equation}\label{eq:chi_va}
		\chi^2(\mu\|\nu) =  \sup_{g\in \matg} \{ |\EE_\mu[g] - \EE_{\nu}[g]|^2: \Var_{\nu}[g] \le 1\}\,,
\end{equation}	
where $\matg$ is any subset that is dense in $L_2(\nu)$. Thus, if $\chi^2(\pi' P\| \pi P) > t^2$ (in particular, if $\pi' P
\not \ll \pi P$) there must exists $g_0\in \matf$ such that 
		$$ \pi' P g_0 - \pi P g_0 < -t \qquad \Var_{\pi P}[g_0] \le 1$$
		Thus taking $g = c g_0$ and $c\to\infty$ in~\eqref{eq:ft_psidef} we again obtain that
		$\psi_t(u)=-\infty$. In the remaining case, $\chi^2(\pi ' P\| \pi P) \le t^2$ and again
		from~\eqref{eq:chi_va} we have that for any $g \in \matf$
			$$ (\pi'  -\pi) Pg \ge -t \sqrt{\Var_{\pi P}[g]}\,,$$
		and thus $\psi_t(u) \ge 0$, while $0$ is achievable by taking $g=0$.
\end{proof}

\begin{proof}[Proof of Corollary~\ref{cor:linthm}] Consider two distributions $\pi$ and $\pi'$ such that 
	$H^2((\pi P)^{\otimes n}, (\pi' P)^{\otimes n}) =2-2\beta$ then from~\cite[Theorem 2.2]{Tsybakov09} we have that
	\begin{equation}\label{eq:clt_1}
		\REB(n,\epsilon) \ge {1\over 2} |T(\pi) - T(\pi')| 
\end{equation}	
	provided that $\beta^2 > 1-(1-2\epsilon)^2$. Thus, taking $\beta > \sqrt{4\epsilon}$ suffices. Recall the
	tensorization identity for $H^2$: 
		$$ 1-{1\over 2}H^2((\pi P)^{\otimes n}, (\pi' P)^{\otimes n}) = \left(1-{1\over 2}H^2(\pi P, \pi'
		P)\right)^n\,.$$
	Consequently, the bound~\eqref{eq:clt_1} holds whenever $H^2(\pi P, \pi'P) \le t_n^2$ with 
	$t_n^2 = 2-2(4\epsilon)^{1\over 2n}$. Note that for $e^{-2n} \le \epsilon \le
{1\over 16}$ we always have $2-2(4\epsilon)^{1\over 2n} \ge {1\over 4n} \ln {1\over \epsilon}$, implying
	$$ \REB(n,\epsilon) \ge {1\over 2} \delta_{H^2}({1\over 4n} \ln {1\over \epsilon})\,,$$
	from which the left-hand bound in~\eqref{eq:clt_0} follows due to~\eqref{eq:delH2_sub} and~\eqref{eq:delta_all}.

	The right-hand bound in~\eqref{eq:clt_0} follows from applying Theorem~\ref{th:linear} to the following generic
	observation:
		\begin{equation}\label{eq:clt_2}
			\REB(2 n L, \epsilon) \le 2\sqrt{\REB(n)}\,, \qquad \forall L\ge 8 \ln {2\over \epsilon}\,.
\end{equation}		
	This follows from the standard ``median trick'': Consider a sample of size $n_1 = 2nL$ and denote by  $\hat T_1, \ldots, \hat T_{2L}$ the result of
	evaluating best quadratic-risk estimator on $2L$ independent subsamples, each of size $n$. Let $\rho =
	2\sqrt{\REB(n)}$. Then from Chebyshev's inequality we have for each $1\le j \le 2L$
		$$ p \eqdef \PP[|T(\pi) - \hat T_j| \ge \rho] \le {1\over 4}\,.$$
	Define the estimator $\hat T$ to be the median of $(\hat T_1,\ldots,\hat T_{2L})$. Then from Hoeffding's
	inequality we have
		$$ \PP[|T(\pi) - \hat T| > \rho] \le \sum_{k\ge L} {2L \choose k} p^k (1-p)^{2L-k} \le e^{-L/8}\,.$$
	From the last statement, we conclude that~\eqref{eq:clt_2} must hold.
\end{proof}

\subsection{Comparison to Donoho-Liu \cite{DL91}}
\label{sec:dl_compare}
Theorem~\ref{th:linear} is very similar to a celebrated result of Donoho-Liu~\cite{DL91}, who showed
that in the same setting, as $n\to\infty$, one has
\begin{equation}\label{eq:donoho_liu}
	C_0 \delta_{H^2}(\tfrac{1}{\sqrt{n}})^2 \leq 	\REB(n) \leq C_1 \delta_{H^2}(\tfrac{1}{\sqrt{n}})^2\,,
\end{equation}
for some constants $C_0,C_1$, i.e.~that the minimax rate for estimating the linear functionals $T$ coincides with \textit{modulus of continuity} of $T$ with respect
to Hellinger distance.
In view of~\eqref{eq:delta_all}, $\delta_{H^2} \asymp \delta_{\chi^2}$ and thus~\eqref{eq:donoho_liu} seems like 
exactly what Theorem~\ref{th:linear} claims. 

The differences, however, are three-fold. First, the technical assumptions required in~\cite{DL91} are: A1, A2 (from Theorem~\ref{th:linear}), boundedness
$\sup_{\pi \in \Pi} |T(\pi)| < \infty$ and \textit{H\"olderianity} of $\delta_{H^2}$:
	$$ \delta_{H^2}(t) = C t^r + o(t^r) $$
	for some $C,r>0$ as $t\to 0$. Barring the latter, the assumptions are weaker than in \prettyref{th:linear}.

The second, and crucial, difference is the fact
that~\eqref{eq:donoho_liu} only holds for a fixed statistical problem $(\Theta, \calX, P, T, \Pi)$ and as $n\to \infty$,
i.e.~the proportionality constants in~\eqref{eq:donoho_liu} are not uniform and can be \textit{problem dependent}. This
precludes one to analyze questions where the problem size (e.g.~dimension) varies with the sample size $n$, etc. For example, in
the population recovery problem considered in Section~\ref{sec:poprec} for any fixed $d$ and $n\to\infty$ we get parametric rate $\REB(n) \asymp {1\over
n}$. To get interesting phase-transitions one needs to let $d$ slowly grow -- and this cannot be handled in the setup
of~\cite{DL91} where the problem is first fixed and then analyzed in the large-sample asymptotics of $n\to\infty$. 

The third difference is the method of proof. While we (indirectly, via duality) show the existence of a good linear
estimator, Donoho and Liu construct an estimator via binary search, which entails decomposing the problem into a dyadic sequence of testing problems between two composite hypotheses of the form $\{\pi: T(\pi) < a\}$ vs $\{\pi: T(\pi) > b\}$. The advantage of their
method is that it can handle loss functions other than the quadratic loss. The advantage of our method is that 
 our estimator is simply an empirical average of a certain function, that, in discrete cases, can be efficiently
 pre-computed by convex or linear programming. Furthermore, even for continuous models, the infinite-dimensional LP can be effectively ``finite-dimensionalized'' leading to computational efficient construction of optimal estimators (see Theorems \ref{thm:de} and \ref{thm:species} for examples).

To sum up, in the iid setting, the chief advantage of our method is in explicit universal constants
comparing $\REB(n)$ and $\delta_{\chi^2}(\frac{1}{\sqrt{n}})$. However, perhaps, the main advantage is 
(as we show in \prettyref{sec:linear}) that our methods extend to the deterministic setting, when one's goal is to estimate a functional
of high-dimensional parameters.

\subsection{Application: Population recovery}\label{sec:poprec}
\def\delchii{\delchi^{(i)}}
\def\delchif{\delchi^{(1)}}
\def\delchis{\delchi^{(2)}}
\def\delchit{\delchi^{(3)}}

%Consider the following problem known as lossy population recovery:
%Let $\pi$ be a probability distribution on the $d$-dimensional Hamming space $\{0,1\}^n$. 
%Let $X=(X^{(1)}, \ldots, X^{(d)})$ be a random binary string drawn from $\pi$ and instead of observing $X$ directly, we observe its erased version $Y=(X^{(1)}, \ldots, Y^{(d)}) \{0,1,?\}^n$, where each $Y^{(i)}$ is obtained independently by passing $X^{(i)}$ through the binary erasure channel, that is, 
%$Y^{(i)}=X^{(i)}$ with probability $1-\epsilon$ and $Y^{(i)}=?$ with probability $\epsilon$.
%Let $X_1,\ldots,X_n$ be iid samples from $\pi$ and we observe their randomly erased version $Y_1,\ldots,Y_n$.
%The goal is to estimate the distribution $\pi$ in the sup norm, which can be reduced to estimating the weight $\pi(0^n)$ of the all-zero string $0^n$.
%Denote the minimax quadratic risk by $\REB(n,d)$. 
%%Clearly, this problem is 
%
%
%The problem of population recovery was initially considered in~\cite{DRWY12,WY12} in the context of learning DNFs with partial observations and further
%investigated in \cite{batman2013finding,MS13,LZ15,DST16}. This problem can also be viewed as a special instance of
%learning mixtures of discrete distributions in the framework of \cite{kearns1994learnability}. 

In the problem of lossy population recovery \cite{DRWY12,WY12}, let $\mu$ denote an unknown distribution on the $d$-dimensional Hamming space $\{0,1\}^d$.
For $n$ iid random binary strings $A_1,\ldots,A_n$ drawn from $\mu$, we observe their erased version $B_1,\ldots,B_n \in \{0,1,?\}^d$., where each bit is erased with probability $\epsilon$, and the goal is to estimate the distribution $\mu$. 
It has been shown in \cite{DRWY12} (cf.~\cite[Appendix A]{PSW17-colt})
estimating the entire distribution $\mu$ in the sup norm can be reduced to estimating the weight of the all-zero string $\mu(\zeros)$ in terms of both sample and time complexity.
Furthermore, for large $d$ it is shown in \cite{PSW17-colt} that summarizing each string into its number of 1's is ``almost sufficient'' (in the sense that it changes the minimax rate by no more than logarithmic factors), as the number of 0's provides negligible information for estimating $\mu(\zeros)$.

After these reductions, we arrive at a specialization of the setup in Theorem~\ref{th:linear}. 
Let $\theta_i$ denote the number of 1's in the unerased string $A_i$. Then $\theta_i \iiddistr \pi$ for some distribution $\pi$ on $\calX=\Theta =\{0,\ldots, d\}$, where $T(\pi)\triangleq \pi(0)=\mu(\zeros)$ is the quantity to be estimated on the basis of $X_1,\ldots,X_n$, where $X_i$ denotes the number of 1's in the erased string $B_i$. Note that conditioned on $\theta$, we have
	\begin{equation}
	 X_i \iiddistr P_\theta = \Bino(\theta,1-\epsilon)\,.
	\label{eq:P3}
	\end{equation}

The minimax risk of population recovery, denoted by $R^*(n,d)$, has been characterized within logarithmic factors in~\cite{PSW17-colt}. 
Next we deduce this result from the general \prettyref{th:linear}, which boils down to characterizing the corresponding $\chi^2$-modulus of continuity $\delchi(t) = \delchi(t,d)$. 
The following result can be distilled from \cite{PSW17-colt} (a proof is given in \prettyref{app:lmm} for completeness):
\begin{lemma}\label{lmm:horo} For any $t\ge 0, d\ge 1$ we have
\begin{equation}\label{eq:horo1}
		\delchi(t,d) \le t^{\min(1, {1-\epsilon\over\epsilon})}\,.
\end{equation}	
	Conversely, for $\epsilon \le {1\over 2}$, $\delchi(t,d) \ge {t\over 2\sqrt{2}}$; 
	for $\epsilon > 1/2$ there exists $t_0 = t_0(\epsilon)$ and $C=C(\epsilon)$ such that
\begin{equation}\label{eq:horo2}
			\delchi(t,d) \ge C \left(t\over \ln {1\over t}\right)^{1-\epsilon \over \epsilon},
\end{equation}		
	provided that $t \le t_0$ and $d \ge C \ln^2 {1\over t}$. 
	%Furthermore, if $d\ge C t^{-2} \ln^{4}{1\over t}$ then	also
%\begin{equation}\label{eq:horo3}
				%\delchis(t,d) \ge C \left(t\over \ln {1\over t}\right)^{1-\epsilon \over \epsilon} .
%\end{equation}			
\end{lemma}

Applying the general \prettyref{th:linear} together with \prettyref{lmm:horo}, we obtain the following characterization of the minimax risks, where the rate of convergence exhibits an elbow effect at erasure probability $\epsilon=\frac{1}{2}$:
\begin{corollary}[\cite{PSW17-colt}] 
%For every $\epsilon \in (0,1)$ there exists a constant $K=K(\epsilon)>0$ such that we have 
		%$$ {1\over K} n^{-\min({1-\epsilon\over \epsilon}, 1)} \log^{-K} n  \le  \sup_{d\ge 1} R^*_{(i)}(n,d) \le
		%n^{-\min({1-\epsilon\over \epsilon}, 1)} $$
		%for all three minimax risks $i=1,2,3$.
%For all three minimax risks $i=1,2,3$, the following holds:
~~~~~
\begin{itemize}
	\item If $\epsilon \in (0,\frac{1}{2}]$, then for any $d\geq 1$,
		$$ {1\over 8 (1+\sqrt{e})^2 n}  \le  R^*(n,d) \le \frac{1}{n}.$$
		
\item If $\epsilon \in (\frac{1}{2},1)$, then there exists a constant $C=C(\epsilon)>0$ such that we have 
		$$ {1\over C} (n \log^2 n)^{-{1-\epsilon\over \epsilon}} \le  R^*(n,d) \le
		n^{-{1-\epsilon\over \epsilon}}, $$
	where the lower bound holds provided that $d\ge C n \log^{4} n$.
		\end{itemize}
	
\end{corollary}
%\begin{proof} Follows from general Theorem~\ref{th:linear} and the technical result above.
%\end{proof}

\subsection{Application: Interval censoring}
\label{sec:ic}

Consider the following setup, commonly occurring in biostatistics. Subjects are arriving according to a Poisson process
starting from time $t_0$.
Upon arrival a treatment is administered to every subject. Following the treatment after a random delay $\theta_i
\simiid
\pi$ the $i$th subject  experiences an event. Statistician terminates the experiment at time $t_1$ and counts all subjects for
which the event has occurred by time $t_1$. The goal is to estimate the CDF or the median of $\pi$. (Note that
statistician does not observe the time of event, only whether it happened or not, which makes it different from the
right-censoring model of Kaplan-Meier~\cite{kaplan1958nonparametric}. It can be seen that conditioned on the number $n$ of total subjects arrived
between $t_0$ and $t_1$ the random time $A_{(i)}$ passed between the $i$th subject's arrival and experiment termination $t_1$
is simply an $i$-th order statistic of the iid uniform sample $A_i \simiid \mathrm{Unif}[0,t_1-t_0]$. This
motivates the formal setting of ``Case 1'' below. See~\cite{huang1997interval} for a survey and more details.

Fix $s_0 \in (0,1)$ and class of distributions $\Pi$ on $\Theta = [0,1]$.
For any distribution $\pi$ we denote by $F_\pi$ its CDF. We
consider two functionals
$$ T_c(\pi) \eqdef F_{\pi}(s_0)=\pi([0,s_0]), \qquad T_m(\pi) \eqdef \int_{[0,1]} \theta \pi(d\theta) \,.$$
For a given $\pi\in \Pi$ we generate $\theta_i \simiid \pi$, $i\in [n]$ and consider two cases of observations:
\begin{enumerate}
\item ``Interval censoring, Case 1'' (also known as ``current status model''), see~\cite[Section
2.3]{groeneboom2014nonparametric}. 
%The observation space is $\calX=[0,1]\times \{0,1\}$, so that a random element $X\in\calX$ we write as a two-component
%vector $X=(A,\Delta)$. 
The observations are $X_1,\ldots,X_n$, where $X_i=(A_i,\Delta_i)$, given by $A_i \simiid G_1$ and $\Delta_{i} = 1\{\theta_i \le A_i\}$. That is,
we only see whether $\theta_i$ has occurred before an independent $A_i$. Here $G_1$ is a fixed (known) distribution on $[0,1]$. 
%In this case the observation space is $\calX=[0,1]\times \{0,1\}$. 
We denote the corresponding Markov kernel acting from $[0,1]$ to $\calX=[0,1]\times \{0,1\}$ by $P^{(1)} = \{P^{(1)}_\theta(\cdot): \theta
\in [0,1]\}$.
\item ``Interval censoring, Case 2'', see~\cite[Section 4.7]{groeneboom2014nonparametric}. This time observations are
given by $X_i=(A_i,B_i,\Delta,\tilde \Delta)$ with $(A_i,B_i)\simiid G_2$, where $G_2$ is some fixed distribution on $[0,1]^2$,
and $\Delta_i = 1\{\theta_i \le A_i\}$, $\tilde \Delta_i = 1\{\theta_i \le B_i\}$. The Markov kernel for this case is denoted by $P^{(2)}$.
\end{enumerate}

We denote the minimax risks for estimating $T_j, j\in \{c,m\}$ in case $i$, $i\in \{1,2\}$ by
$R^{(i,j)}_n(\Pi)$. Under appropriate conditions on $G_1$, $G_2$ and $\Pi$, the following was shown in a series of works, cf.~\cite[Section 5]{GL95} and~\cite[Example 3.1]{groeneboom1992information}:
\begin{align} 
   R^{(1,m)}_n(\Pi) &\asymp R^{(2,m)}_n \asymp {1\over n}\,,\label{eq:ic_rm}\\
   R^{(1,c)}_n(\Pi) &\asymp {1\over n^{2/3}}\,,\label{eq:ic_rc1}\\
   R^{(2,c)}_n(\Pi) &\asymp {1\over (n\log n)^{2/3}} \label{eq:ic_rc2}\,.
\end{align}

We denote the $\chi^2$-moduli of continuity for functionals $T_c$ and $T_m$ under the two models as follows:
$$ \delta^{(i,j)}_{\chi^2}(t; \Pi) \eqdef \sup\{T_j(\pi) - T_j(\pi'): \pi,\pi' \in \Pi, \chi^2(\pi P^{(i)} \| \pi'
P^{(i)}) \le t^2\}\,, \quad i=1,2,\quad j=c,m.$$
The next result shows that the minimax rates are determined by these moduli of continuity. 
The proof is given in \prettyref{app:lmm}, which amounts to verifying the assumptions of \prettyref{th:linear}.

\begin{proposition}
\label{prop:ic}
 Suppose that the set $\Pi$ is convex and weakly closed. For the case of estimating $T_c$, in
addition, we assume that $s_0$ is a point of continuity of $F_\pi(\cdot)$ for every $\pi \in \Pi$. Suppose that distributions $G_1$ and $G_2$
both have densities $g_1$ and $g_2$.\footnote{For estimating $T_c$ we could also demand only the existence of density in small interval around
$s_0$ or $(s_0,s_0)$.} Then we have for all $i\in\{1,2\}, j\in\{c,m\}$
		$$ {R^{(i,j)}_n} \asymp \delta^{(i,j)}_{\chi^2}\pth{{1\over \sqrt{n}}; \Pi}^2 $$
		Furthermore, in each case the upper bound is attained by an estimator of the form $\sum_{k=1}^n \phi(X_k)$ 
		for some continuous function $\phi$ on $\calX$.
		%$$ \hat T_j = {1\over n} \sum_{k=1}^n \phi(X_k)\,,$$
		%where $\phi$ is some continuous function on $\calX$.
\end{proposition}

The above general characterization of the minimax risk by \prettyref{prop:ic} can be converted into explicit rate of convergence. We give two examples:
\begin{enumerate}
\item Suppose that the distribution $G_1$ has density $g_1$ such that $g_1(a) \ge \epsilon_1>0$ for all $a\in(0,1)$. Then we
have for any weakly closed and convex $\Pi$:
	$$ R_n^{(1,m)} \asymp \delta^{(1,m)}_{\chi^2}\pth{{1\over \sqrt{n}}; \Pi}^2
	\asymp {1\over n}\,.$$

	Indeed, consider any two distribution $\pi_1$ and $\pi_2$ with corresponding CDFs given by $F_1$ and $F_2$. A
	simple calculation shows
		\begin{equation}\label{eq:ic_j2}
			\chi^2(\pi P^{(1)} \| \pi'P^{(1)}) = \int_0^1 g_1(a) {(F_1(a) - F_2(a))^2\over F_2(a)(1-F_2(a))}\,.
	\end{equation}		
	Upper-bounding the denominator by ${1\over 4}$ and lower-bounding $g_1$ by $\epsilon_0$ we obtain
		\begin{align} \chi^2(\pi P^{(i)} \| \pi'P^{(i)}) &\ge 4\epsilon_0 \int_0^1 da (F_1(a) - F_2(a))^2\\
				&\ge 4\epsilon_0 \left(\int_0^1 da (F_1(a) - F_2(a))\right)^2\,, \label{eq:ic_j5}
		\end{align}				
	where the last step is via Jensen's inequality. Since the integral equals $T_m(\pi_2) - T_m(\pi_1)$ we get that
		$$ \delta^{(1,m)}_{\chi^2}(t) \le {t\over 2\sqrt{\epsilon_0}}\,.$$
	Note also that from~\eqref{eq:delta_all} and~\eqref{eq:delta_superl} we conclude that 
	$\delta^{(1,m)}_{\chi^2}(t) \asymp t$. This recovers~\eqref{eq:ic_rm} under most general conditions. (Note that
	if density $g_1$ is zero on some interval $[a,b]$ then the model becomes unidentifiable.)

\item Now suppose that $G_1$ has density $g_1$ that is continuous and positive at $s_0$. Consider the class $\Pi(\gamma,\epsilon) \eqdef  \{\pi:
F_\pi(s) \mbox{ is $\gamma$-Lipschitz for~} s\in(s_1,s_2)\}$, where $s_1 = s_0-\epsilon, s_2 = s_0+\epsilon$. Then we claim
	\begin{equation}\label{eq:ic_j3}
		\delta^{(1,c)}_{\chi^2}(t) \asymp t^{1/3}\,.
\end{equation}	
Without loss of generality, we will assume that $\epsilon_0 \le g_1(s) \le {1\over \epsilon_0}$ for all $s\in (s_1,s_2)$. (Otherwise, we simply reduce $\epsilon$.) 
Notice that $\Pi$ is indeed convex and weakly closed. For the lower bound consider any pair of CDFs $F_1,F_2$ such that
(a) they both belong to $\Pi$, (b) $1/4 < F_2(s_1) < F_2(s_2) < 3/4$, and (c)
$$ F_1(s) - F_2(s) = \begin{cases} 0, & s<s_3 \mbox{~or~} s>s_4\\
				   s_5 + {\gamma\over 2} |s-s_0|, & s_3 < s < s_4
			\end{cases}\,,$$
where $s_3 = s_0-\tau_1, s_4=s_0+\tau_1$ and $s_5 = -{\gamma\tau_1/2}$. From~\eqref{eq:ic_j2} we obtain:
	$$ \chi^2(\pi_1 P^{(1)} \| \pi_2P^{(1)}) = \int_{s_3}^{s_4} ds g_1(s) {(F_{1}(s) -
	F_2(s))^2 \over F_2(s)(1-F_2(s))}  \lesssim \tau_1^3\,.$$
	Thus, this demonstrates $\delta_{\chi^2}(\tau_1^3) \ge \tau_1$, as claimed by~\eqref{eq:ic_j3}. 

	For the upper bound, arguing as in~\eqref{eq:ic_j5} we get for any $\tau_1<\epsilon$ that
	$$ t^2 \ge \chi^2(\pi_1 P^{(1)} \| \pi_2P^{(1)}) \gtrsim \int_{s_0-\tau_1}^{s_0+\tau_1} |F_1(s_0+x) -
	F_2(s_0+x)|^2 ds\,.$$
	Now, if we set $\delta = F_1(s_0)-F_2(s_0)>0$ then from Lipschitzness we get that for $\tau_1 = {\delta \over
	2\gamma}$ we have that $|F_1(s_0+x)-F_2(s_0+x)|\ge {\delta\over 2}$ and thus 
	$$ t^2 \gtrsim \delta^3\,,$$
	implying that $\delta_{\chi^2}(t) \lesssim t^{1/3}$, finishing the proof of~\eqref{eq:ic_j3}.
\end{enumerate}

These applications hopefully demonstrate the utility of Theorem~\ref{th:linear}. Indeed, the minimax rates~\eqref{eq:ic_rm}
and~\eqref{eq:ic_rc1} were obtained with a lot less effort compared to the existing literature \cite{groeneboom2014nonparametric}\footnote{Results in \cite{groeneboom2014nonparametric}  are stated with the extra assumption on the lower bound on the density of $\pi$ at $s_0$, but this constraint seems not necessary for establishing rates.} In
particular, the previous upper bounds were derived by a lengthy analysis of the nonparametric maximum
likelihood~\cite{groeneboom1992information}, or certain ad hoc histogram estimator~\cite{GL95}.
On the other hand, establishing~\eqref{eq:ic_rc2} by computing $\delta_{\chi^2}$ is more involved and will be presented
elsewhere.

%\section{Extension 1: High-dimensional functional estimation}
\section{Deterministic setting}
\label{sec:hd}

In this section we consider the deterministic setting as described in \prettyref{sec:intro}. 
Namely, the observations are $\bX=(X_1,\ldots,X_n)$, where $X_i\inddistr P_{\theta_i}$.
Here the unknown parameter $\vect{\theta}=(\theta_1,\ldots,\theta_n)$ is deterministic and belongs to the following constraint set
%$\bTheta_c$ is defined as 
	$$ \bTheta_c = \left\{ \vect{\theta} \in \Theta^{\otimes n}: {1\over n} \sum_{i=1}^n c(\theta_i) \le 1 \right\}\,,$$
	for some cost function $c:\Theta \to \mreals$.
Equivalently, $\bTheta_c$ consists of those $\btheta$ whose empirical distribution $\pi_{\btheta} = \frac{1}{n} \sum_{i=1}^n \delta_{\theta_i}$ 
 belongs to the convex set
\begin{equation}
\Pi = \sth{\pi \in \calP(\Theta): \int c(\theta) \pi(d\theta) \leq 1}.
\label{eq:Pic}
\end{equation}
Let $h:\Theta\to\reals$ and define the following affine functional 
\begin{equation}
T(\pi)=\int h(\theta) \pi(d\theta). 
\label{eq:Th}
\end{equation}
		Given $\bX=(X_1,\ldots,X_n)$, the goal is to estimate  $T(\pi_{\btheta})= \frac{1}{n}\sum_{i=1} h(\theta_i)$, which is a symmetric separable function of the parameter $\btheta$.
	The minimax quadratic risk $\RComp(n)$ is defined in \prettyref{eq:RComp}, namely,
	\begin{equation}\label{eq:rd_def}
		\RComp(n) = \inf_{\hat T} \sup_{\vect{\theta} \in \bTheta_c} \EE_{\btheta}[|\hat T(\vect{X}) -
	T(\pi_{\btheta})|^2] 
\end{equation}	

%
%A $d$-dimensional parameter $\vect{\theta}=(\theta_1,\ldots,\theta_n) \in \bTheta_c$ is given. The constraint set
%$\bTheta_c$ is defined as 
	%$$ \bTheta_c = \left\{ \vect{\theta} \in \Theta^{\otimes n}: {1\over n} \sum_{i=1}^n c(\theta_i) \le 1 \right\}\,,$$
	%for some cost function $c:\Theta \to \mreals$.
	%%Let $P:\Theta \to \matx$ be a  transition kernel. 
	%Given the data $\vect{X}=(X_1,\ldots,X_n)$, where $X_i \sim P_{\theta_i}$ independently, the goal is to estimate a separable functional $T(\pi_{\btheta})$:
	%$$ T(\pi_{\btheta}) = {1\over n} \sum_{i=1}^n h(\theta_i)\,,$$
	%where $T:\Theta \to \mreals$.
	%The minimax quadratic risk is defined as
	%\begin{equation}\label{eq:rd_def}
		%R^*(n) = \inf_{\hat T} \sup_{\vect{\theta} \in \bTheta_c} \EE[|\hat T(\vect{X}) -
	%T(\pi_{\btheta})|^2] 
%\end{equation}	
%where the infimum is taken over all measurable function $\hat T: \calX^n \to \reals$.

Many problems studied in the high-dimensional functional estimation literature are of or can be reduced to questions of the above type. For example, in the Gaussian model where $X_i\sim N(\theta_i,1)$, estimation of linear ($h(\theta)=\theta$) and quadratic functional ($h(\theta)=\theta^2$) has been well-studied and more recently under sparsity assumptions which correspond to adding further constraints with $c(\theta)=\indc{|\theta|>0}$ or $c(\theta)=|\theta|^q$ \cite{CCTV16,CCT17}. Estimation of non-smooth functional such as the $\ell_1$-norm ($h(\theta)=|\theta|$) has been studied in \cite{LNS99,CL11}.

%We define the following convex set of probability distributions
	%$$ \Pi = \{\pi \in \matp(\Theta): \EE_\pi[c(\theta)] \le 1\}\,.$$
%Identifying $\theta$ with $\delta_\theta$, we extend $T$ to $\pi \in \Pi$ by linearity:
	%$$ T(\pi) = \int_{\Theta}  T(\theta) \pi(d\theta)\,.$$
%Our technical assumptions below will imply this integral indeed exists. Finally, with $\Pi$ and $T:\Pi \to \mreals$
%defined, we also define $\delta_{\chi^2}(\cdot)$ via~\eqref{eq:delchi2_def}.

%The main idea of this section is that the stated minimax problem is very similar to a problem where, instead of
%adversarially selecting a vector $\vect{\theta}$, one generates each coordinate $\theta_i$ independently from some prior $\pi\in \Pi$, and instead of
%$T(\pi_{\btheta})$ one estimates $\EE[T(\pi_{\btheta})] = T(\pi)$, which is a linear functional of $\pi$. The latter problem falls into the purview
%of Section~\ref{sec:linear} and hence its minimax rate is given by $\delta_{\chi^2}(\tfrac{1}{\sqrt{n}})$. Thus, it seems natural
%to expect that
%\begin{equation}\label{eq:hd_asymp_wrong}
		%\RComp(n) \asymp \delta_{\chi^2}\pth{{1\over \sqrt{n}}}^2
%\end{equation}	
%up to universal constants. Alas, such statement cannot hold without conditions, as the next example demonstrates. However, the good news is that such counterexamples only occur in the
%``uninteresting'' case of $R^*_{\rm det}(n)=0$ or $R^*(n) \asymp {1\over n}$ (parametric rate). 

The main idea of this section is that the minimax problem in the deterministic setting is similar to the iid setting studied in \prettyref{sec:linear}, where, instead of
adversarially selecting a vector $\vect{\theta}$ from $\bTheta_c$, one generates each coordinate $\theta_i$ independently from some prior $\pi\in \Pi$ such that $\Expect_{\theta \sim \pi}[c(\theta)] \leq 1$. By concentration, we expect the constraint $\frac{1}{n}\sum_{i=1}^n c(\theta_i) \leq 1$ to be fulfilled approximately and indeed this product prior can be made valid with appropriate truncation. Furthermore, we expect $T(\pi_{\btheta}) = \frac{1}{n} \sum_{i=1}^n h(\theta_i)$ to be concentrated near its mean $T(\pi) = \int h(\theta)\pi(d\theta)$, which is a linear functional of $\pi$. Estimating the latter falls under the purview of Section~\ref{sec:linear} and hence its minimax rate is given by $\delta_{\chi^2}(\tfrac{1}{\sqrt{n}})$. Thus, it seems natural
to expect that
\begin{equation}\label{eq:hd_asymp_wrong}
		\RComp(n) \asymp \delta_{\chi^2}\pth{{1\over \sqrt{n}}}^2
\end{equation}	
up to universal constants. Alas, such statement cannot hold without conditions, as the next example demonstrates. However, the good news is that such counterexamples only occur in the
``uninteresting'' case of $\RComp(n)=0$ or $\RComp(n) \asymp {1\over n}$ (parametric rate).

\begin{example} Let $\Theta = \matx = \{0,1\}$, $c(\theta)=0$, so that $\Pi=\{\Bern(p): 0 \leq p \leq 1\}$. Let 
$h(\theta)=\theta$ and consider the observation model $\PP[X=\theta] = 1-\PP[X=1-\theta] =
\tau$ (the binary symmetric channel). From~\eqref{eq:delta_all} and~\eqref{eq:delta_superl} we obtain that for any
$\tau\ge 0$ (including $\tau=0$!): $ \delta_{\chi^2} (t) \ge {t\over 2\sqrt{2}}$.
At the same time, a simple unbiased estimator $\hat T(X_1,\ldots, X_n) = {1\over n(1-2\tau)} \sum_{i=1}^n (1\{X_i=1\} -
\tau)$ achieves
	$$ \RComp(n) \le {\tau (1-\tau) \over (1-2\tau)^2} {1\over n}\,.$$
One immediate conclusion is that at $\tau=0$ we have $\RComp(n)=0$ while $\delchi(t)>0$ for all $t>0$. Furthermore, even
when $\tau>0$ and $\RComp(n) \asymp \delchi(1/\sqrt{n})^2 \asymp \tfrac{1}{n}$, the proportionality constant in the
first relation is not uniform in $\tau$, as $\lim_{\tau \to 0} {\RComp(n)\over \delta_{\chi^2}(1/\sqrt{n})^2} = 0 $. Therefore, we cannot expect
the relation~\eqref{eq:hd_asymp_wrong} to hold universally.
\end{example}

\begin{remark}[Parametric lower bound] 
\label{rmk:lb-parametric}
Consider the setting where the constraint function $c$ and the function $h$ are both fixed and the sample size $n$ grows.
There is a general \emph{dichotomy}: either risk $\RComp(n)=0$ or $\RComp(n) = \Omega(\frac{1}{\sqrt{n}})$. Indeed, either there exists a pair $\theta_a,\theta_b \in \Theta$ s.t. $h(\theta_a)\neq h(\theta_b)$ and
$\TV(P_{\theta_a},P_{\theta_b)})<1$, or there is no such pair. In the latter case, we have $h(\theta)=g(X_1)$ (i.e.~$h(\theta)$ is a deterministic function of a single sample), and thus $\RComp(n)=0$ for any $n\geq 1$. In the former case, we can lower bound $\RComp(n)$ by
the Bayes risk when $\vect{\theta}$ has iid components with $\PP[\theta_i = \theta_a] = \PP[\theta_i = \theta_b] =
{1\over 2}$.\footnote{This prior needs to be modified if $c(\theta_a)>1$ or $c(\theta_b)>1$. 
Specifically, choose an arbitrary $\theta_0$ such that $c(\theta_0)<1$. Then we can choose $\vect{\theta}$ iid from $\pi = (1-\epsilon) \delta_{\theta_0} + \frac{\epsilon}{2} (\delta_{\theta_a} + \delta_{\theta_b})$ for sufficiently small constant $\epsilon$.
}
 Clearly, the corresponding Bayesian risk is $\Omega(1/\sqrt{n})$.
\end{remark}

The main result of this section is:
 %(see \prettyref{seca:pf-highdim} for a proof):
\begin{theorem}\label{th:highdim}
Suppose that $(\Theta,\matx,P,T, \Pi)$, with 
 $\Pi$ and $T$ given in \prettyref{eq:Pic} and \prettyref{eq:Th} respectively, satisfy conditions A1-A4 of Theorem~\ref{th:linear}.
Then
\begin{equation}\label{eq:hd_ach}
		\RComp(n) \le \delchi\left({1\over \sqrt{n}}\right)^2\,,
\end{equation}	
achieved by an estimator of the form $\hat T = \frac{1}{n} \sum_{i=1}^n g(X_i)$ for some $g\in\calF$.
Furthermore, suppose the following extra conditions are satisfied
\begin{enumerate}
\item[A5] $K_V = \sup_{\pi \in \Pi} \Var_{\theta \sim \pi}[T(\theta)] < \infty$;
\item[A6] Cost function $c \ge 0$ and there exists $\theta_0 \in \Theta$ with $c(\theta_0)=0$.
\end{enumerate}
Then
\begin{equation}\label{eq:hd_conv}
		\RComp(n) \ge {1\over 2400} \delchi\left({1\over \sqrt{n}}\right)^2 - {K_V\over 2n}\,.
\end{equation}	
\end{theorem}

%\subsection{Proof of \prettyref{th:highdim}}
	%\label{seca:pf-highdim}
\begin{proof} 
Recall that $T(\pi)=\int h(\theta)\pi(d\theta)$ and  $T(\pi_{\btheta}) = \frac{1}{n}\sum_{i=1}^n h(\theta_i)$.
To prove~\eqref{eq:hd_ach}, consider an estimator $\hat T_g$ of the form~\eqref{eq:emp_est} and, similarly
to~\eqref{eq:emp_bv}, let us analyze its risk by decomposing into bias and variance parts:
\begin{equation}\label{eq:hd_1}
		\sqrt{\EE_{\vect{\theta}} [|\hat T_g(\vect{X}) - T(\pi_{\btheta})|^2]} \le {1\over n} \sqrt{\sum_{i=1}^n
	\Var_{P_{\theta_i}}[g]} + \left|{1\over n} \sum_{i=1}^n (P_{\theta_i} g - h(\theta_i)) \right| 
\end{equation}	
Recall that the empirical distribution of $\vect{\theta}=(\theta_1,\ldots,\theta_n)$ is denoted by $  \pi_{\btheta} = {1\over n} \sum_{i=1}^n \delta_{\theta_i}$, so that 
$\pi_{\btheta} P = \frac{1}{n} \sum_{i=1}^n P_{\theta_i}$.
By the concavity of $\mu \mapsto \Var_\mu[g]$, upper-bounding 
	$$ \sum_{i=1}^n	\Var_{P_{\theta_i}}[g] \le n \cdot \Var_{ \pi_{\btheta} P}[g], $$
we continue~\eqref{eq:hd_1} to get
\begin{equation}
		\sqrt{R^*(n)} \le \inf_g \sup_{\btheta} {1\over \sqrt{n}} \sqrt{\Var_{\pi_{\btheta} P}[g]} + |T(\pi_{\btheta}) -
	\pi_{\btheta} P g|\,, 
\end{equation}	
where the supremum is taken over all $\btheta$ whose empirical measure $\pi_{\btheta}$ belongs to $\Pi$.
Thus we can extend the inner supremum to $\hat \pi$ ranging over all of $\Pi$,
concluding 
	$$ \sqrt{R^*(n)} \le \deltaa(\tfrac{1}{\sqrt{n}})$$
with $\deltaa$ defined in~\eqref{eq:deltach_def}. Applying \prettyref{prop:delminmax} we get~\eqref{eq:hd_ach}.

To prove~\eqref{eq:hd_conv}, fix $\gamma  \in (0,1)$ (to be specified later) and 
consider $\pi_0,\pi'_0 \in \Pi$ such that $\chi^2(\pi_0'P\|\pi_0 P) \le {1\over n}$ and
$T(\pi_0') - T(\pi_0) = \delta$. Next define distributions
	$$ \pi_1 = \gamma \pi_0 + (1-\gamma)  \delta_{\theta_0}, \quad \pi'_1 = \gamma \pi'_0 + (1-\gamma)
	\delta_{\theta_0}\,,$$
	where $\theta_0$ is from Assumption A6 such that $c(\theta_0)=0$.
From the convexity of $\chi^2(\cdot\|\cdot)$, we get 
	$$ \chi^2(\pi_1' P\| \pi_1 P) \le {\gamma\over d}, \qquad T(\pi_1') - T(\pi_1) = \gamma \delta\,.$$
Denote $\mu' = T(\pi_1'), \mu = T(\pi_1)$. Define distributions $\nu = \pi_1^{\otimes n}, \nu' = \pi_1'^{\otimes n}$ and 
note that $(\pi_1 P)^{\otimes n} = \nu P^{\otimes n}$. 
Then 
\begin{align*}
\TV(\nu P^{\otimes n}, \nu' P^{\otimes n})
= \TV((\pi_1 P)^{\otimes n}, (\pi_1' P)^{\otimes n}))
\stepa{\leq} & ~ \frac{1}{2}\sqrt{\chi^2((\pi_1 P)^{\otimes n} \| (\pi_1' P)^{\otimes n}) } \\
\stepb{=} & ~ \frac{1}{2}\sqrt{(1+\chi^2(\pi_1 P\| \pi_1' P))^n-1 } \\
\stepc{\leq} & ~ \frac{1}{2}\sqrt{(1+\gamma/n)^n-1 } \le \frac{1}{2}\sqrt{e^{\gamma} - 1},
\end{align*}
where (a) follows from the fact that $\TV \leq \frac{1}{2} \sqrt{\chi^2}$ \cite[Section 3]{gibbs2002choosing}; 
(b) is from the tensorization identity in \eqref{eq:tv_chi2_tenso};
(c) is from the convexity $\chi^2(\pi_1 P\| \pi_1' P) \leq \gamma \chi^2(\pi_0 P\| \pi_0' P)$.

Next define sets $A,A' \subset \bTheta_c$:
	\begin{align} A &= \sth{\vect{\theta} \in \Theta^{\otimes n}: \frac{1}{n} \sum_{i=1} c(\theta_i) \le 1, T(\pi_{\btheta}) \le \mu + {\gamma
	\delta \over 3}}\\
	A' &= \sth{\vect{\theta} \in \Theta^{\otimes n}: \frac{1}{n} \sum_{i=1} c(\theta_i) \le 1, T(\pi_{\btheta}) \ge \mu' - {\gamma
		\delta \over 3}}\,.
\end{align}
	From the Chebyshev and Markov inequalities we have
		$$ \nu(A^c), \nu'(A'^c) \le \gamma + {9 K_V\over n \gamma^2 \delta^2} $$
Next, decompose distributions $\nu, \nu'$ as convex combinations:
	$$ \nu = \nu(A) \nu_{|A} + \nu(A^c) \nu_{|A^c}\,, \nu' = \nu'(A') \nu'_{|A'} + \nu'(A'^c) \nu'_{|A'^c}\,,$$
where $\nu_{|B}(\cdot) \eqdef \nu(\cdot\cap B)/\nu(B)$ is the conditional version of the distribution $\nu$.

By the triangle inequality and the data processing inequality of total variation, we get
	$$ \TV(\nu_{|A} P^{\otimes n}, \nu'_{|A'} P^{\otimes n}) \le \nu(A^c) + \nu'(A^c) + \TV(\nu P^{\otimes n}, \nu'
	P^{\otimes n})\,. $$
Altogether, we have a pair of distributions $\nu_1 \triangleq \nu_{|A}$ and $\nu_1' \triangleq \nu'_{|A'}$ both supported on $\bTheta_c$ such that $T(\pi_{\btheta}) \le \mu - {\gamma \delta\over 3}$ for $\nu_1$-a.e.~$\btheta$ and $T(\pi_{\btheta}) \ge \mu + {\gamma \delta\over 3}$ for $\nu_1'$-a.e.~$\btheta$. Applying the $\TV$ version of Le Cam's method for quadratic risk (see \cite[Lemma 1]{Yu97}) yields the following minimax lower bound:
	$$ R^*(n) \ge {1\over 4} \left(\gamma \delta\over 3\right)^2 (1-t)\,,$$
	where $t \triangleq 2\gamma + {18 K_V \over n \gamma^2 \delta^2} + \sqrt{e^{\gamma} -1}/2$. Choosing $\gamma \in (0,1)$ to maximize 
	the function $\gamma^2-2\gamma^3-\sqrt{e^\gamma-1}/2$, we obtain
	$$ R^*(n) \ge {1\over 2400} \delta^2 - {K_V\over 2n}\,.$$
	Optimizing over the choice of $\pi_0,\pi'_0$ thus yields \eqref{eq:hd_conv}.
	%$$ R^*(n) \ge {1\over 6^5} \delchi^2(\sqrt{7\over 4 d}) - {K_V\over 4d}\,.$$
	%Applying~\eqref{eq:delchi_sub} 
	%\nb{This seems incorrect because $\sqrt{7/4}>1$ so we can't apply \eqref{eq:delchi_sub}.}
	%we obtain
	%$$ R^*(n) \ge {7^2\over 2^9 3^5} \delchi^2(\sqrt{1\over n}) - {K_V\over 4d} \ge \left({1\over 62}
	%\delchi(\sqrt{1\over n})\right)^2 - {K_V\over 4d}\,.$$
	%Finally, using $\sqrt{a}-\sqrt{b} < \sqrt{a-b}$ (when the bound is non-trivial), we get~\eqref{eq:hd_conv}.
\end{proof}

\begin{remark}

	Before presenting new results obtained from \prettyref{th:highdim}, as a quick application, consider the problem of estimating the $\ell_1$-norm of a vector in the Gaussian location model \cite{LNS99,CL11}, where $X_i\sim N(\theta_i,1)$, $h(\theta)=|\theta|$ and $T(\pi_{\btheta})=\frac{1}{n}\|\vect{\theta}\|_1$, and $\Theta=[-1,1]$. 
	Using the method of polynomial approximation and moment matching, it was shown in \cite{CL11} that $\RComp(n) = \Theta((\frac{\log\log n}{\log n})^2)$. (In fact, the sharp constant as $n\diverge$ was also found). 
To see how this result follows from \prettyref{th:highdim}, note that $K_V=1$, we have $
c\delchi^2(\frac{1}{\sqrt{n}}) - \frac{1}{4n} \leq \RComp(n) \leq \delchi^2(\frac{1}{\sqrt{n}})$ for constant $c$, where 
\begin{align}
\delchi(t) =  \sup\sth{\int |\theta| \pi(d\theta) -\int |\theta| \pi'(d\theta) : \chi^2(\pi'* N(0,1)\|\pi * N(0,1))\le t^2}.
\label{eq:delchi-CL-def}
\end{align}
Here $*$ denotes convolution, and the supremum is taken over $\pi,\pi' \in\calP([-1,1])$. 
The speed of convergence of $\delchi(t)$ when $t\to0$ is extremely slow and thus its behavior governs the minimax rate. 
Indeed, 
%using again polynomial approximation and moment matching, 
one can show that (see \prettyref{app:lmm})
\begin{equation}
\delchi(t) = \Theta\pth{\frac{\log\log\frac{1}{t}}{\log\frac{1}{t}}},
\label{eq:delchi-CL}
\end{equation}
 recovering the result of \cite{CL11}.

However, if the parameter space is unbounded with $\Theta=\reals$ we have $K_V=\infty$ and lower bound in
\prettyref{th:highdim} is not applicable. (In fact it is easy to see that $\delta_{\chi^2}(t)=\infty$ for any $t$.)
Nevertheless, applying a truncation argument, it was shown in \cite{CL11} that $\RComp(n) \asymp \frac{1}{\log n}$.	
\end{remark}

\subsection{Application: Distinct Elements problem}
\label{sec:de}

\apxonly{\textbf{Juditsky asks:} what is the risk of estimating from just the \# of distinct elements among the observed balls?}

In the \emph{distinct elements} problem, given a sample randomly drawn from an urn containing multiple colored balls, the goal is to estimate the total number of distinct colors in the urn. 
This problem has been thoroughly investigated in both statistics and computer science under various formulations and sampling models. 
We refer the readers to the comprehensive survey \cite{BF93}, \cite{CCMN00,RRSS09,Valiant11,GV-thesis,WY2016sample} for more recent work, and  \cite[Table 1]{WY2016sample} for a summary of the state of the art. 
In this section, we consider the following version of the distinct elements problem, where the number of balls in the urn is at most $n$ and unknown a priori. We shall work with the so-called \emph{Bernoulli sampling model} with sampling ratio $p$, a specific version of sampling without replacement, where the color of each ball is observed independently with probability $p$; see \cite[Appendix A]{WY2016sample} for connections and near equivalence to other sampling models.

Most of the recent theoretical results aim at the sublinear regime of $p=o(1)$. In particular, it is known that the optimal sample complexity for consistency (in normalized error) is $p = \Theta(\frac{1 }{\log n})$.  
%however, these results typically aim at the sublinear regime. 
For the linear regime, say, 1\% of the balls are observed, existing results do not yield tight characterization of the optimal estimation accuracy.
Next, we will apply the general \prettyref{th:highdim} to determine the minimax risk up to logarithmic factors in the linear regime, and reveal an elbow effect in the optimal rate of convergence that precisely occurs at sampling ratio $\frac{1}{2}$. 
%The result will also recover the $\frac{d}{\log n}$ result in the sublinear regime.
%that can Also, \prettyref{thm:de} recovers the $\frac{d}{\log n}$ result with tight exponent and good multiplicative constants.
%It turns out there is an elbow effect that occurs precisely at $p > \frac{1}{2}$.

%Specifically, let us consider the following version of the distinct elements problem, where the number of balls in the urn is at most $d$ and unknown a priori.
%Consider the following problem. We are given an urn with $\le d$ colored balls. For each ball we make an independent
%decision and with probability $1-\epsilon$ put it on the table. The observer inspects collection of balls on the table
%and tries to predict the total number $N$ of distinct colors present in the urn. Note that in our variation, the total
%number of balls is not known apriori (i.e. observer does not know how many balls are remaining in the urn).\footnote{Another
%variation (with same minimax risk) is to fix the total number of balls to be $d$ and reveal a random fraction
%$(1-\epsilon)$ of the colors.} 
Without loss of generality, assume that the number of colors in the universe (not necessarily in the urn) is $n$, and indexed by $[n]=\{1,\ldots,n\}$.
Let $\theta_i \in \mathbb{Z}_+$ be the number of balls of the $i$th color, $i=1,\ldots,n$. Thus, the parameter 
$\btheta=(\theta_1,\ldots,\theta_n)$ is constrained  to belong to the set
	$$ \bTheta_c = \sth{\btheta \in \integers_+^n: {1\over n}\sum_{i=1}^n \theta_i \le 1 }.$$
	We shall work with the Bernoulli sampling model with sampling ratio $p$, where the color of each ball is observed independently with probability $p$. 
	Denote by $N_i$ the number of observed balls of the $i$th color. Then we have $X_i \inddistr \Binom(\theta_i)$.
	Given $(N_1,\ldots,N_n)$, the goal is to estimate the (normalized) number of distinct colors:
	$$ T(\pi_{\btheta}) \eqdef {1\over n} \sum_{i=1}^n \indc{\theta_i \ge 1}. $$
	This problem is exactly in the deterministic setting of \prettyref{th:highdim} and the minimax quadratic risk $R^*(n) \equiv \RComp(n)$ is defined as in~\eqref{eq:RComp}.

%We shall work with the Bernoulli sampling model with sampling ratio $p$, where the color of each ball is observed independently with probability $p$. 
%To conform to the notations in the previous section, instead of estimating the number of distinct colors $N \triangleq \sum_{i=1}^n \indc{\theta_i \ge 1}$, we estimate a normalized quantity
	%$$ T(\pi_{\btheta}) \eqdef {1\over n} \sum_{i=1}^n \indc{\theta_i \ge 1}. $$
%The minimax quadratic risk $R^*(n)$ is defined as in~\eqref{eq:rd_def}.

%Recall previous results \cite{WY2016sample}. See \cite[Table 1]{WY2016sample} for a summary of the state of the art. In particular, 
%however, these results typically aim at the sublinear regime. 
%For the linear regime, say, 1\% of the balls are observed, do not yield precise results that can determine the phase transition in the optimal rate of convergence (i.e.~the elbow effect).
%Also, \prettyref{thm:de} 
The following theorem (proved in \prettyref{seca:pf-de}) determines the sharp minimax risk up to logarithmic factors in the linear sampling regime ($p$ being a constant). 
Note that the upper bound is explicit and non-asymptotic, which allows us to recover the prior result on the optimal sampling complexity $\Omega(\frac{n}{\log n})$, i.e.~$p = \Omega(\frac{1}{\log n})$,  for consistent estimation.
\begin{theorem}
\label{thm:de} Fix $p \in (0,1)$. There exists a constant $c=c(p)>0$ such that
\begin{itemize}
	\item if $p \geq {1\over 2}$, then 
	\begin{equation}
	{c \over n} \le R^*(n) \le {1\over n}
	\label{eq:de1}
	\end{equation}	
	
	\item if $p< {1\over 2}$, then
	\begin{equation}
	{c\over \log^2 n} n^{-\frac{p}{1-p}} \le {R^*(n)} \le n^{-\frac{p}{1-p}},
	\label{eq:de2}
	\end{equation}
\end{itemize}
where the upper bound holds for all $n$ and the the lower bound holds for all $n \geq n_0=n_0(p)$.
Furthermore, this upper bound can be achieved within a constant factor by an estimator of the form 
\begin{equation}
\hat T = \sum_{i \in [n]} g(N_i)
\label{eq:linear-estimator}
\end{equation}
with $g(0)=0$, where the coefficient $g$ can be found by solving an LP of $O(n)$ variables and $O(n)$ constraints.
%constructed in time $O(n^C)$ for some absolute constant $C$.
\end{theorem}
%\textcolor{red}{\textbf{\Large TODO:} Lower bound need update pending horo's fix.}
\begin{remark}[Linear estimator]
%\label{rmk:}	
%One particular consequence of \prettyref{th:highdim} is that it shows the optimality of the following empirical-mean estimator:
%where $N_i$ is the observed number of balls of the $i$th color.
Estimators of the form \prettyref{eq:linear-estimator} are commonly known as \emph{linear estimators}, since they can be equivalently expressed as linear combinations of \emph{profiles} (also known as fingerprints) \cite{OSW15,VV11-focs}:
\[
\hat T = \sum_{j \geq 0} g(j) \Phi_j,
\]
where 
\begin{equation}
\Phi_j \triangleq \sum_i \indc{N_i = j},
\label{eq:Phi}
\end{equation}
 called the $j$th profile, denotes the number of colors that occurred exactly $j$ times in the sample.
Since \prettyref{thm:de} guarantees we can choose $g(0)=0$ in \prettyref{eq:linear-estimator}, the resulting estimator $\hat T = \sum_{j \geq 1} g(j) \Phi_j$ 
is fully data-driven and oblivious to the total number of possible colors, a desirable property in practice.
%In practice, it is desirable to have $g(0)=0$, 
%%since it is not realistic to assume the total number of possible colors is known.
%in which case the estimator is fully data-driven and adaptive to the total number of possible colors. 
\end{remark}

\subsection{Application: Fisher's species problem}
\label{sec:species}

Dating back to Fisher \cite{FCW43}, 
\emph{predicting the unseen species} is a classical question in statistics, where given a sample of $n$ iid
 observations $X_1,\ldots,X_n$ drawn from an unknown probability discrete distribution $P=(p_x)$ on some countable alphabet $\calX$, the goal is to estimate the number of hitherto unobserved symbols that would be observed if a new sample of $X_1', \ldots, X_{m}'$ were collected, i.e.,
%$X_{n+1}, \ldots, X_{n+m}$:
 \[
U = U_{n,m}  \triangleq |\{X_1', \ldots, X_m'\} \backslash \{X_1,\ldots,X_n\}|.
\]
In particular, the sequence $m\mapsto U_{n,m}$ is called the species discovery curve, which provides guidance on how many new species would be observed were $m$ more data points to be collected. For this reason, extrapolating the species discovery curve is of significant interest in various fields such as ecology \cite{FCW43,CL92}, computational linguistics \cite{ET76}, genomics \cite{ICL09}, etc. Clearly, the more future data we want to extrapolate, the more difficult it is to obtain a reliable prediction.

%To be consistent with the existing literature as well as for the sake of technical simplicity,
In order to frame the problem in the deterministic estimation setting, we consider the Poissonized version of the problem as studied in \cite{FCW43,ET76,OSW15}, where the sizes of the available and future (unobserved) samples are $N\sim\Poi(n)$ and $M\sim\Poi(m)$, respectively.
Due to the concentration of the Poisson distribution, standard arguments (see \prettyref{app:species-poisson}) show that the minimax risk bounds proved next apply to the model with fixed sample sizes with little change. Denote the histogram in the observed and unobserved sample by $N_x = \sum_{i\in[N]} \indc{X_i = x}$  and $N_x' = \sum_{i\in[M]} \indc{X_i' = x}$, respectively.
Then $\{N_x\} \inddistr \Poi(n p_x)$ and $\{N_x'\} \inddistr \Poi(m p_x)$ are independent of each.
In terms of histograms, the number of unseen species can be expressed as
\begin{equation}
U = \sum_x \indc{N_x = 0, N_x' > 0}.
\label{eq:U}
\end{equation}
Let $r \triangleq \frac{m}{n}$ denote the extrapolation ratio.
Denote the normalized minimax mean squared error of estimating $U$ by
\[
\calE_n(r) \triangleq \inf_{\hat U} \sup_P \frac{1}{m^2}\Expect_P[(\hat U-U)^2],
\]	
where the expectation is with respect to both the original and the future samples.
We emphasize that this problem is fully non-parametric and no assumptions are imposed on the distribution $P$.

It is known since Good and Toulmin \cite{GT56} that an unbiased estimator for $U$ is
\[
\hat U_{\rm GT} = - \sum_x (-1)^{N_x} \indc{N_x > 0} = 
\sum_{j\geq 1} -(-r)^j \Phi_j,
\]
where $\Phi_j$ is the $j$th profile defined in \prettyref{eq:Phi}.
If $r \leq 1$, that is, we extrapolate no more than what have been observed, this unbiased estimator achieves the (optimal) parametric rate
\begin{equation}
%\calE_n(r)\leq 
\frac{1}{m^2} \Expect[(U - \hat U_{\rm GT})^2] \lesssim \frac{1}{n}.
\label{eq:species-GT}
\end{equation}
However, 
for $r>1$, the variance of $\hat U$ is unbounded due to the exponential growth of the coefficients. 
Based on a technique called \emph{smoothing} that modifies the unbiased estimator to obtain a good bias-variance tradeoff, 
Orlitsky et al \cite{OSW15} constructed a family of estimators that 
%are also linear in the profiles and 
encompass previous heuristics of Efron and Thisted \cite{ET76} and provably achieve the following prediction risk:
\begin{equation}
\calE_n(r) \lesssim n^{-\log_3(1+\frac{2}{r}) }.
\label{eq:species-OSW}
\end{equation}
Conversely, the following lower bound is also shown in \cite{OSW15}:
\[
\calE_n(r) \gtrsim  n^{-C/r }.
\]
for some absolute constant $C$.
Thus, one can extrapolate with a vanishing risk provided that $r = o(\log n)$, and this condition is the best possible. However, for fixed $r$, the optimal rate remains open. In particular, the above achievable results \prettyref{eq:species-GT} and \prettyref{eq:species-OSW} seem to suggest an ``elbow effect'' in 
the optimal convergence rate, which transitions from parametric rate to nonparametric rate when the extrapolation ratio $r$ exceeds $1$.
The following result resolves this question in the positive:
%\begin{equation}
%\calE_n(r) =
%\begin{cases}
%\Theta\pth{n^{-1}}  &  r \leq 1 \\
 %\tilde\Theta\pth{n^{-\frac{2}{r+1}}}  &  r > 1\\
%\end{cases}
%\label{eq:species-risk}
%\end{equation}
%More precisely, we will prove the following:
%, demonstrating an elbow effect of the optimal convergence rate at $r=1$:
\begin{theorem}[Optimal rate for predicting the unseen]
\label{thm:species}
Let $r>0$ be a constant. 
There exist constants $c_0,c_1$ that depend only on $r$, such that the following holds.
\begin{itemize}
	\item If $r \leq 1$, then
\begin{equation}
\frac{c_0}{n} \leq \calE_n(r) \leq \frac{c_1}{n};
\label{eq:species-risk1}
\end{equation}
\item If $r>1$, then
\begin{equation}
\frac{c_0 n^{-\frac{2}{r+1}}}{\log^2 n}  \leq \calE_n(r) \leq  c_1 n^{-\frac{2}{r+1}} \log^4 n.
\label{eq:species-risk2}
\end{equation}
\end{itemize}
Furthermore, an estimator achieving the upper bound can be constructed and evaluated in time $O(n^a)$ for some absolute constant $a$.
\end{theorem}

It is worth mentioning that, unlike \prettyref{thm:de}, \prettyref{thm:species} does not directly follow from the general result in \prettyref{th:highdim} because of the infinite-dimensional nature of the species problem (the number of distinct species is potentially unbounded), which requires extra reduction argument. Furthermore, analyzing the behavior of the modulus of continuity (as a linear program) relies on delicate complex analysis, in particular, Hadamard's three-lines theorem and the Paley-Wiener theorem. The proof of \prettyref{thm:species} is provided in \prettyref{seca:pf-species}.

\begin{remark}[Species versus distinct elements problem] 
%\label{rmk:}	
There is an obvious connection between the species problem considered here and the distinct elements problem considered in \prettyref{sec:de}: 
Treating the union of observed and unobserved samples $\{X_1,\ldots,X_n,X_1',\ldots,X_m'\}$ as the content of an urn, the former can be viewed as a special case of the latter with the urn size being $n+m$ and the fraction of observation being $p = \frac{n}{m+n} = \frac{1}{1+r}$.
Thus, for the interesting case of $r>1$, applying \prettyref{thm:de} yields 
the upper bound $\calE_n(r) \leq O(n^{-\frac{1}{r}})$.
Perhaps surprisingly, this strategy turns out to be suboptimal in view of \prettyref{thm:species}. This suggests that the optimal
estimator for the species problem is able to exploit the special structure in the color configuration arising from iid sampling.
\end{remark}

\section{Exponential families}
\label{sec:exp}

Let us revisit the setting of Theorem~\ref{th:linear} in the special case when $\Theta$ and $\calX$ are both
finite. Given a closed convex set $\Pi \subset \Theta$ we define 
	$$ M_0 \eqdef \{\mu: \mu = \pi P, \pi \in \Pi\}\,,$$
which again is closed and convex.
Any (identifiable) linear functional $T(\pi)$ can in turn be represented as a linear functional of $\mu$. Hence, the
statistical problem at hand becomes: Given $X_i \simiid \mu \in M_0$ estimate $T(\mu) \eqdef \sum_{x\in\calX} \mu(x)
h(x)=\Iprod{\mu}{h}$, where $h:\calX \to \mreals$ is a fixed function. 

Let us restate the problem in the language of exponential families. Without loss
of generality, let $\calX = [d] \eqdef \{1,\ldots,d\}$. Then for $j\in [d]$ let $\phi_j(x) = \Indc\{x=j\}$,
and denote $\phi(x) = (\phi_1(x)\,\ldots,\phi_m(x))$.
For every $\gamma = (\gamma_1,\ldots,\gamma_d) \in \mreals^d$ we define a distribution on $\calX$
	$$ P_\gamma(x) = \exp\{\langle \gamma, \phi(x)\rangle - C(\gamma)\}\,,$$
where $C(\gamma)$ is chosen from normalization; more explicitly, $P_\gamma(x)=\frac{\exp(\gamma_x)}{\sum_{x\in\calX} \exp(\gamma_x)}$. We can see that $P_\gamma$ forms an exponential family with
\textit{natural parameters} $\gamma$ and \textit{mean parameters} $\mu_f(\gamma) = \EE_{X\sim P_\gamma}[\phi(X)]$.
Theorem~\ref{th:linear} then shows that there exists $g$ such that the empirical-mean estimator
\begin{equation}\label{eq:exp_estim}
		\hat T(X_1,\ldots,X_n) = {1\over n} \sum_{i=1}^n g(X_i)
\end{equation}	
that achieves the minimax rate for estimating $T(P_\gamma) = \langle \mu_f(\gamma), h\rangle$ from
$X_i\simiid P_\gamma$ over the class $\{P_\gamma: \mu_f(\gamma) \in M_0\}$.

It turns out that this result can be extended: many other exponential families (i.e.~different choices of $\phi:\calX \to
\mreals^d$) still enjoy the same property of (near) optimality of empirical-mean estimators. This extends the result
of~\cite{JN09} to square loss and a wider class of exponential families (see discussion in Section~\ref{sec:jn_compare}). 
We proceed
to formal definitions.

A $d$-dimensional exponential family $\{P_\gamma\}_{\gamma \in \Gamma}$ of probability distributions on a measurable
space $\calX$ is given by a triplet $(\nu, \phi, \Gamma)$, where $\nu$ is a reference measure on $\calX$,
$\phi:\calX \to \mreals^d$ is a measurable map, $\Gamma \subset \mreals^d$ and 
	$$ P_\gamma(dx) = \exp\{\iprod{\gamma}{\phi(x)} - C(\gamma)\} \nu(dx)\,, $$
with $\gamma \in \mreals^d$ called the \emph{natural parameter}. 
Let $\matf$ be the finite-dimensional linear space spanned by basis functions $\phi_{i}$, i.e., 
	$\matf = \{\iprod{h}{\phi}: h\in\reals^d\}$. We make two standing assumptions on
the exponential family:\footnote{Note that the second assumption is without loss of generality: if there is a linear
relation between coordinates of $\phi$,
then by reducing the dimension $d$ we eventually will make the second assumption hold.}
\begin{enumerate}
\item \label{E1}
The set $\Gamma$ is open and convex; $C(\gamma) < \infty$ for all $\gamma \in \Gamma$.\apxonly{\\I only need
convexity of $\Gamma$ to conclude that $\gamma \mapsto \mu_f(\gamma)$ is one-to-one.}
\item \label{E2}
For some $\gamma_0 \in \Gamma$ (and hence for all $\gamma$ by absolute continuity $P_\gamma \ll P_{\gamma_0}$), the 
functions $\phi_{1},\ldots, \phi_{d}$ are linearly independent, i.e.
\begin{equation}\label{eq:exp_nondeg}
		\Var_{X\sim P_{\gamma_0}} (\iprod{\phi(X)}{h}) > 0 \qquad \forall h \in \mreals^d \setminus \{0\}\,.
\end{equation}	
\end{enumerate}

In addition to the natural parameter $\gamma$, we define the \emph{mean parameter} $\mu$ via the forward map
	$$\mu_f(\gamma) \eqdef \EE_{X\sim P_\gamma}[\phi(X)]\,. $$
It is well known (see \eg \cite{Brown.EF}) that inside $\Gamma$ the function $\gamma\mapsto C(\gamma)$ is infinitely differentiable, whose first two derivatives 
give the mean and covariance of $\phi(X)$:
%. Indeed, a simple calculation shows that
	\begin{equation}
	\mu_f(\gamma) = \nabla C(\gamma), \qquad {\partial \mu_f \over \partial \gamma} = \mathrm{Hess}\, C(\gamma) =
	\Cov_{P_\gamma}[\phi(X)] \eqdef
	\Sigma(\gamma)\,.
	\label{eq:moments-exp}
	\end{equation}	
%where the Jacobian matrix ${\partial \mu_f \over \partial \gamma}$ was denoted $\Sigma(\gamma)$. The same matrix is also
%the covariance matrix of $X$:
	%$$ \Sigma(\gamma) = \Cov_{P_\gamma}[X]\,.$$
The non-degeneracy assumption~\eqref{eq:exp_nondeg} implies 
\begin{equation}\label{eq:exp_nondeg2}
		\Sigma(\gamma) \succ 0 \qquad \forall \gamma \in \Gamma\,.
\end{equation}
Since $C(\gamma)$ is, thus, strictly convex on $\Gamma$, the map $\gamma \mapsto \mu_f = \nabla C(\gamma)$ is one-to-one. 
Since the Jacobian of this map is non-zero everywhere
on $\Gamma$, by the inverse function theorem the image $M\eqdef \mu_f(\Gamma)$ is an open set in $\mreals^d$ and, furthermore, there is an
infinitely-differentiable inverse map $\gamma_r$ such that
	$$ \mu_f(\gamma_r(\mu)) = \mu \qquad \forall \mu\in M\,.$$
It is also known that Jacobian of $\gamma_r$ can be computed as
\begin{equation}\label{eq:exp_jacinv}
		{\partial \gamma_r(\mu)\over \partial \mu} = {\Sigma^{-1}(\gamma_r(\mu))}\,.
\end{equation}	
For convenience we denote $\tilde P_\mu = P_{\gamma_r(\mu)}$ and $\tilde \Sigma(\mu) = \Sigma(\gamma_r(\mu))$. 

For a given constraint set $\Gamma_0 \subset \Gamma$ and a functional $T(\gamma)$,  we define the minimax
square-loss as usual
\begin{equation}
R_n^*(\Gamma_0) = \inf_{\hat T} \sup_{\gamma\in\Gamma_0} \EE_{X_i \simiid P_\gamma}[|\hat T(X_1,\ldots, X_n) -
	T(\gamma)|^2]\,.
\label{eq:Rn-exp}
\end{equation}
 
The main finding in this section is that for estimating linear functionals of the mean parameter $\mu$, 
under certain convexity assumptions (that are strictly weaker than those in \cite{JN09}), the minimax quadratic risk is characterized by 
certain moduli of continuity within universal constant factors. 
To this end, let $\omega_H$
 denote the modulus of continuity of $T$ on $M_0$ with respect to the Hellinger distance, i.e.
 \begin{align}
\omega_H(t) &\eqdef \sup_{\gamma,\gamma' \in \Gamma_0} \{T(\gamma)-T(\gamma'): H(P_\gamma,
		P_{\gamma'}) \le t\},
\end{align}

\begin{theorem}\label{th:exp} There exist absolute constants $c_0>0$ and $c_1>0$ with the following property. Fix any
dimension $d\ge 1$ and any exponential family $(\nu, \phi, \Gamma)$ satisfying regularity assumptions \ref{E1} and \ref{E2} above. Consider a subfamily of an exponential family corresponding to mean parameters $\mu
\in M_0 \subset M \subset \mreals^d$, where $M_0$ is compact and convex. Assume that the subfamily $M_0$ satisfies the key
condition
\begin{equation}\label{eq:exp_a2}
	\mu \mapsto \sqrt{\Var_{P_\mu}[\phi]} \mbox{ is concave in $\mu \in M_0$ for all $\phi \in \matf$} .
\end{equation}
Let the functional $T(\gamma)$ be \emph{linear in the mean parameter}, i.e.,
\begin{equation}\label{eq:tlin}
		T(\gamma) = \iprod{h}{\mu_f(\gamma)}
\end{equation}	
for some $h\in\reals^d$, and define the constraint set $\Gamma_0 = \gamma_r(M_0)$. Then  we have 
\begin{equation}\label{eq:exp_main}
			c_0 \omega_H(1/\sqrt{n}) \le \sqrt{R_n^*(\Gamma_0)} \le c_1 \omega_H(1/\sqrt{n}),,
\end{equation}
and this rate is achieved by the estimator of the type~\eqref{eq:exp_estim} with $g\in\calF$.
\end{theorem}

The proof is given in \prettyref{seca:pf-exp}. 
We stress that constants $c_0,c_1$ in~\eqref{eq:exp_main} do not depend on dimension of the exponential family, and thus
as in~\cite{JN09} we can think of the above result as essentially non-parametric.

\begin{remark} Note that in the setting of the preceding \prettyref{th:exp}, we have 
	$$ T(\gamma) = \EE_{X\sim P_\gamma}[\phi_0(X)]\,,$$
	for some $\phi_0 \in \matf$. Thus, it may appear that a good estimator would arise from taking $g =\phi_0$
	in~\eqref{eq:exp_estim}. Indeed, it gives an unbiased estimator by design. However, the subtlety here is that
	$\Var_{P_\gamma}[\phi_0(X)]$ might be prohibitively large
	(such as in population recovery in \prettyref{sec:poprec}). The main discovery here is that the concavity condition~\eqref{eq:exp_a2}
	guarantees existence of some other $g \not= \phi_0$ such that the empirical average of $g$ is minimax rate-optimal.
\end{remark}

\begin{remark} To shed some light on how assumption~\eqref{eq:exp_a2} relates to the tightness of empirical-mean estimators, we observe that the Fisher information matrix for
parameter $\gamma$ is given by $ I_F(\gamma) = \Sigma(\gamma)\,,$
while for parameter $\mu$ we get $ I_F(\mu) = \tilde\Sigma^{-1}(\mu)\,.$
In one dimension $d=1$, we see that~\eqref{eq:exp_d2x} shows that $R_n^*(M_0) \le {1\over n \min_{\mu} I_F(\mu)}$.
From the Bayesian Cram\'er-Rao lower bound (van Trees inequality) \cite{GL95}, we expect a similar lower bound to hold, unless $I_F(\mu)$ grows very rapidly around its minimum.
The latter situation is prohibited by the assumption~\eqref{eq:exp_a2}, as shown by the key
inequality~\eqref{eq:exp_key}. Thus, assumption~\eqref{eq:exp_a2} enters our proof in two crucial ways: for the
applicability of the minimax theorem and for taming the behavior of Fisher information. Because of the latter, it is unclear
whether~\eqref{eq:exp_a2} can be extended from concavity to, say, quasi-concavity.
\end{remark}

\subsection{Comparison to Juditsky-Nemirovski~\cite{JN09}}\label{sec:jn_compare}

	As opposed to the squared loss \prettyref{eq:Rn-exp}, 
Juditsky-Nemirovski~\cite{JN09} considered the $\epsilon$-quantile loss and the corresponding minimax risk:
\begin{align*}
R^*_{n,\epsilon}(\Gamma_0) %&\triangleq \inf_{\hat T} \inf\sth{r: \sup_{\gamma \in \Gamma_0} P_\gamma[|\hat
%T(X_1,\ldots,X_n) - T(\gamma)| > r] \le \epsilon},\\
			   &\triangleq \inf_{\hat T} \sup_{\gamma \in \Gamma_0} \inf\sth{r: P_\gamma[|\hat
T(X_1,\ldots,X_n) - T(\gamma)| > r] \le \epsilon}.
\end{align*}
Nevertheless, Theorem~\ref{th:exp} proved for the quadratic risk can be translated to the $\epsilon$-quantile loss similarly as done in Corollary~\ref{cor:linthm}.
\begin{corollary}\label{cor:exp} 
	In the setting of Theorem~\ref{th:exp}, whenever $e^{-2n} \le \epsilon \le 2^{-8} $ we have up to absolute
	constants of proportionality
\begin{equation}\label{eq:rne_exp}
		R^*_{n,\epsilon}(\Gamma_0) \asymp \omega_H\left(\sqrt{{1\over n} \ln {1\over \epsilon}}\right)\,. 
\end{equation}	
\end{corollary}
\begin{proof} The lower bound is proved by a Hellinger-based two-point argument as in Corollary~\ref{cor:linthm}. The upper bound follows from the same median trick as in Corollary~\ref{cor:linthm}. Bounding $\omega_H(ct)$ by $\omega_H(t)$ from
above and below is done exactly as in the proof of Theorem~\ref{th:exp}.
\end{proof}

We now discuss results of~\cite{JN09}. The following assumptions are made in \cite{JN09} (later called \textit{a simple
observation schemes} in~\cite[Section 2.4.2]{JNbook})
\begin{enumerate}% [label=JN\arabic*]
\item 
%\label{JNass1}
The exponential family $(\nu, \phi,\Gamma)$ has $\Gamma=\mreals^d$, i.e.~the natural parameters 
	$\gamma$ can range over the entire space $\mreals^d$.
\item 
%\label{JNass2}
The functional $T(\gamma) = T(A(\xi))$ is affine in $\xi$, where $\gamma = A(\xi)$ is a reparametrization such that the map
\begin{equation}\label{eq:exp_ajn}
		\xi \mapsto C(A(\xi)+a) - C(A(\xi)) \,\mbox{ is concave for every $a\in\mreals^d$.}
\end{equation}
\end{enumerate}
	Under these assumptions, it is shown that	(cf.~\cite[Theorem 3.1 and Proposition 3.1]{JN09})
	\begin{equation}
	\frac{1}{2} \omega_H\pth{\sqrt{2 \pth{1 - e^{-\frac{1}{2n} \log \frac{1}{4\epsilon} }}}}
	\leq 	R^*_{n,\epsilon} \leq \frac{1}{2} \omega_H\pth{\sqrt{2 \pth{1 - e^{-\frac{1}{n} \log \frac{2}{\epsilon} }}}};
	\label{eq:JN09-main1}
	\end{equation}
in particular, whenever $ \exp(-2n) \leq \epsilon < \frac{1}{5}$, we have~\eqref{eq:rne_exp} (also within absolute
constants).
%	
%To compare with the quadratic risk characterization in \prettyref{th:exp}, first of all, in terms of results, since
%$R^*_{n,\epsilon} \geq \sqrt{R^*_{n} / \epsilon}$ by the Markov inequality, comparing \prettyref{eq:exp_main} with
%\prettyref{eq:JN09-main2} shows that under the assumption of \cite{JN09},	the modulus of continuity with respect
%to the Hellinger distance and the Jeffrey's divergence are equivalent up to constant factors.  
%
Thus, to compare our results with~\cite{JN09} we need to compare the
assumptions. It turns out (see \prettyref{seca:pf-exp} for a proof)
that~\eqref{eq:exp_ajn} is equivalent to the following requirement
	\begin{equation}\label{eq:exp_a2star}
	\xi \mapsto \mu_f(A(\xi)) \text{ is affine and }
	\mu \mapsto \Var_{P_\mu}[\phi(X)] \mbox{ is concave in $\mu \in M_0$ for all $\phi \in \matf$}.
	\end{equation}	
	This equivalence shows that our condition~\eqref{eq:exp_a2} is strictly weaker than~\eqref{eq:exp_a2star}. Let
	us consider a simple example showing difference between~\eqref{eq:exp_a2} and~\eqref{eq:exp_a2star}.

\begin{example}[Exponential distribution]
\label{ex:exp}
Let $\calX = \mreals^d_+$, $\gamma \in \mreals^d_+$ and take $P_\gamma(dx) = \prod_{i=1}^d e^{-\gamma_i x_i} \indc{x_i>0} dx_i$, i.e.
$X\sim P_\gamma$ has $d$ independent components, each exponentially distributed. In this case $\phi(x) = x$. 
The mean parameters are $\mu = (\gamma_1^{-1}, \ldots, \gamma_d^{-1})$. 
Our goal is to estimate $T(\gamma) = \sum_{i=1}^d \mu_i$ over the $\ell_p$-ball in $\mreals^d$:
$M_0 =\{\mu: \sum_{i=1}^d \mu_i^p \leq 1\}$, where $p\ge 1$. A simple calculation shows
	\begin{equation}\label{eq:exp_exp}
		\omega_H(t) \asymp t d^{\max(\tfrac{1}{2}-\tfrac{1}{p},0)} 
\end{equation}	
	up to absolute (i.e. $p$-independent) constants.
	\apxonly{More precisely, $$ {1\over 2} t d^{\max(\tfrac{1}{2}-\tfrac{1}{p},0)} \le \omega_J(t) \le t
	d^{\max(\tfrac{1}{2}-\tfrac{1}{p},0)}\,.$$}% 
	(For $p \le 2$ the worst pair $(\mu,\mu')$ are $1$-sparse (with a single nonzero), whereas for $p>2$ they are scaled
constant vectors.) From \prettyref{th:exp} we conclude that the minimax quadratic risk is $\Theta({1\over n}
d^{\max(1-\frac{2}{p},0)})$. In this simple case the empirical mean $\hat T = {1\over n} \sum_{t=1}^n \sum_{i=1}^d
(X_t)_i$ achieves the optimal rate for all $p,d$, suggesting that the problem is rather simple. 
\apxonly{By
taking $\min((X)_1,\ldots,(X)_d)$ we can estimate $\sum_{i=1}^n \gamma_i$ with parametric rate ${1\over \sqrt{n}}$
independent of dimension. }
However, while our condition~\eqref{eq:exp_a2} holds, the condition~\eqref{eq:exp_a2star} imposed by~\cite{JN09} does
not. In addition, the natural parameter ranges over a subset of $\mreals^n$, not all of $\mreals^n$ (again in violation
of~\cite{JN09}).
\end{example}

	The above example thus shows that the extension from~\eqref{eq:exp_a2star}  to~\eqref{eq:exp_a2} is not vacuous. Another example
	is the normal scale model $X \sim \matn(0, \sigma^2), \sigma^2 >0$ with $\phi(x)=x^2$. For this
	family, again~\eqref{eq:exp_a2} holds but not~\eqref{eq:exp_a2star}. For larger dimension $d>1$, the family $X \sim
	\matn(0,\Sigma)$ with $\Sigma$ a $d\times d$ positive definite matrix, does not satisfy either~\eqref{eq:exp_a2star}
	or~\eqref{eq:exp_a2}. (However, in this case a partial remedy is possible -- see~\cite[Section 3.4.1]{JNbook}.)

	\apxonly{Another idea was to check $\omega \sim \matn(0, \mu)$, where $\mu$ is a covariance matrix,
	$X=(X_{i,j})$ with $X_{i,j} = \omega_i \omega_j$. Turns out this does not satisfy~\eqref{eq:exp_a2}. For
	example, for  case of $\dim = 2$ the functional $\sqrt{\Var_{P_{\mu}}[X_{1,2}]}$ is
		$$ \sqrt{\mu_{1,1} \mu_{2,2} + \mu_{1,2}^2}\,,$$
		which, unfortunately is not concave (check on the line of $\mu_{1,1}=\mu_{2,2}=1$). 
		\textbf{TODO:} Try to use some geodesic convexity tricks in PSD spaces?}

We point out, however, an important case of where our methods fail, but methods of~\cite{JN09,JNbook} succeed. Namely,
in~\cite[Section 3.1.4]{JNbook} it is shown that (in the notation of Theorem~\ref{th:exp}) the minimax risk of
estimating a linear (in $\mu=\mu_f(\gamma)$) functional satisfies~\eqref{eq:rne_exp} also when $M_0$ is a finite union
of convex sets (and in fact, the linear functional $T(\gamma)$ is allowed to be different on different convex sets). 
Unfortunatelly, this elegant result does not extend to the quadratic risk as the following example demonstrates.

\begin{example} Consider the goal of estimating the bias $p$ of $X_i\simiid \Bern(p)$ where $p\in\{1/4,1/3\}$.
 Then the modulus of continuity $\omega_H(t) = 0$ for $t<H(\Bern(1/4),\Bern(1/3))$, but this does
not contradict~\eqref{eq:rne_exp} since $R_{n,\epsilon}^* = 0$ for all sufficiently large $n$. At the same time, it is
clear that quadratic risk $R_n^* \ge c_1 e^{-c_2 n}$ for some constants $c_1,c_2>0$. Consequently, in the setting when
$M_0$ is a union of convex sets, characterization $R_n^* \asymp \omega_H(1/\sqrt{n})$ is not possible.
\end{example}

%\section{Extension 3 (?): Estimating distribution globally}
%
%\textbf{TODO:} YP proposal: Survey the main idea but point to the "in preparation" paper where it will be worked out. What
%does YW think?
%
%
%
%
%The idea here is to consider estimators that return $\hat T_\alpha(X^n)$ for a whole family of $T_\alpha(\pi)$ and then
%define risk as $\sup_\alpha \sup_\pi \EE|T_\alpha - \hat T_\alpha|^2$. This is equivalent to estimating $\pi$ itself
%w.r.t. norm
	%$$ d(\pi,\hat\pi) = \sup_\alpha |T_\alpha(\pi) - T_\alpha(\hat \pi)|$$
%
%\textbf{TODO: } I am not sure this stuff will work out
%
%\subsection{Application: estimating urn distribution in TV-norm from a subset}
%
%Skipping many steps, I think I showed that there exists estimator of color distribution of $d$-ball urn 
%from $\epsilon$-erasure observation of it with precision
	%$$ \TV(\hat \pi , \pi) \lesssim {1\over \log {d}}$$
%This rate is sharp.

\section{Additional proofs}
\label{seca:pf}

%\subsection{Proof of \prettyref{prop:deltaproperty}}
%\label{seca:pf-deltaproperty}

%\subsection{Proof of \prettyref{prop:ic}}
	%\label{sec:pf-propic}

\subsection{Proof of \prettyref{thm:de}}
	\label{seca:pf-de}
\begin{proof}
Clearly the sufficient statistic is the histogram of the observed colors, that is, $\{N_i: i \in [n]\}$, where $N_i$ is the number of observed balls of the $i$th color. Thus we have $N_i\inddistr \Binom(\theta_i,p)$.
Therefore, the setting of \prettyref{thm:de} is a particularization of the general \prettyref{th:highdim}, with $\Theta=\calX=\integers_+$, 
$P_\theta=\Binom(\theta,p)$, $c(\theta)=\theta$, $\Pi=\{\pi\in\calP(\integers_+):  \int \theta\pi(d\theta)\leq 1 \}$ (which is weakly compact), 
and $h(\theta) = \indc{\theta\geq 1}$ so that $T(\pi) = 1-\pi_0$, where 
we identify $\pi$ with its PMF $\pi_k \equiv \pi(\{k\})$.
 %or equivalently, $T(\theta)=\indc{\theta=0}$.
Furthermore, the assumptions of \prettyref{th:highdim} are fulfilled (with $K_V \leq \frac{1}{4}$ and $\theta_0=0$).
Applying \prettyref{th:highdim}, it remains to characterize the behavior of $\delchi(t)$.
Note that $\delchi(t)$ is closely related to $\delchi(t,n)$ previously studied for the population recovery problem in \prettyref{sec:poprec} (with $\epsilon=1-p$). 
Both dealing with the binomial model, the only difference is the additional moment constraint in $\delchi$ and the difference in the domain ($\integers_+$ versus $\{0,\ldots,n\}$).
%Applying the superlinearity property \prettyref{eq:delta_superl} in \prettyref{prop:delminmax}, we have $\delchi(t) \geq t/4$.
Indeed, we have
\begin{align}
\delchi(t) 
%&= \sup\{\pi\{\theta \geq 1\} - \pi'\{\theta \geq 1\}:  \chi^2(\pi P \| \pi' P)\le t^2, \pi, \pi'\in\Pi\} \nonumber \\
&= \sup\{\pi_0 - \pi'_0:  \chi^2(\pi P \| \pi' P)\le t^2, \pi, \pi'\in\Pi\} \nonumber \\
&\leq \sup\{\pi_0 - \pi'_0:  \chi^2(\pi P \| \pi' P)\le t^2, \pi, \pi'\in \calP(\integers_+)\} \triangleq \delchi'(t) \nonumber \\
&\leq \sup\{\pi_0 - \pi'_0:  \TV(\pi P, \pi' P)\le t, \pi, \pi'\in \calP(\integers_+)\} \triangleq \delTV'(t) \label{eq:de-tv1} \\
&\leq t^{\min(1, {p\over1-p})}, \label{eq:de-tv}
\end{align}
where the last inequality follows from \prettyref{lmm:horo} (in particular \prettyref{eq:horo-tv} for $d=\infty$). 
Substituting $t=1/\sqrt{n}$, this completes the proof of the upper bound $\sqrt{R^*(n)} \leq n^{-\frac{1}{2} \min(1, {p\over1-p})} \triangleq \epsilon_n$ as in \prettyref{eq:de1} and \prettyref{eq:de2}.

For the constructive part, consider an estimator of the form \prettyref{eq:linear-estimator}, namely $\hat T = {1\over n} \sum_{i=1}^n g(N_i)$. 
Choose $g$ to be the solution $g^*$ to the following LP (below $h=(0,1,\ldots,1)$):
\begin{equation}
\min_{g\in\reals^{n+1}} \|Pg - h\|_\infty + \frac{1}{\sqrt{n}} \|g\|_\infty,
\label{eq:LP-de}
\end{equation}
which is equal to the dual LP
\[
\max_{\Delta\in\reals^{n+1}} \{\Iprod{\Delta}{h}:  \|\Delta P\|_1 \le t, \|\Delta\|_1 \leq 1\}.
\]
By \cite[Lemma 7]{PSW17-colt}, this LP is upper bounded by twice the value of the \prettyref{eq:de-tv1} with $t=1/\sqrt{n}$, which shows the choice of $g^*$ achieves the quadratic  risk $4 \epsilon_n^2$. The LP \prettyref{eq:LP-de} (with $O(n)$ variables and $O(n)$ constraints) can be solved in time that is polynomial in $n$.

To finish the proof of the upper bound, we show that we can impose the constraint that $g(0)=0$ and still achieve the upper bounds \prettyref{eq:de1}--\prettyref{eq:de2} within a constant factor. Indeed, from \prettyref{eq:LP-de} we conclude that for all $\theta =0,\ldots,n$, $|(Pg^*)(\theta)-h(\theta)| = |\Expect_{N \sim \Binom(\theta,p)} [g^*(N)] - h(\theta)| \leq 2 \epsilon_n$. Particularizing to $\theta=0$, we have $|g^*(0)| \leq 2 \epsilon_n$.  
Consider the modified estimator $\tilde g$ given by $\tilde g(0)=0$ and $\tilde g(j)=g^*(j)$ for all $j \geq 1$.
We have $\|\tilde g\|_\infty\leq \|g^*\|_\infty$ and $\|P \tilde g-h\|_\infty\leq \|P g^*-h\|_\infty + \|P (\tilde g-g)\|_\infty \leq  \|P g^*-h\|_\infty + 2 \epsilon_n$. This shows $\tilde g$ achieve a quadratic risk of at most $16 \epsilon_n^2$. 

%Recall the definition of $\deltaa$ in \prettyref{eq:deltach_def}, which gives the best bias-variance tradeoff among linear estimators. 
%In view of \prettyref{prop:delminmax}, we have the universal relation, thanks to duality, $\deltaa(t) \leq \delchi(t)$, the latter of which is upper-bounded in \prettyref{eq:de-tv}. Thus, dropping the variance term, we conclude that there exists $g: \integers_+\to\reals$, such that
%$\sup_{\theta \in \integers_+} |\Expect_{N \sim \Binom(\theta,p)} [g(N)] - T(\theta)| \leq n^{-\frac{1}{2} \min(1, {p\over1-p})}$,
%where $T(\theta)=\indc{\theta \geq 1}$.
%Particularizing to $\theta=0$, we have $|g(0)| \leq n^{-\frac{1}{2} \min(1, {p\over1-p})}$.
%This shows that the modified estimator $\tilde g$ given by $\tilde g(0)=0$ and $\tilde g(j)=g(j)$ for all $j \geq 1$ continues to achieve the optimal rate in \prettyref{thm:de}.

Next we proceed to the lower bound. The parametric lower bound in \prettyref{eq:de1} follows from
\prettyref{rmk:lb-parametric}. To complete the proof of \prettyref{eq:de2}, it remains to show the lower bound: for any
$p \leq \frac{1}{2}$ and all $t\le t_0(p)$ we have
\begin{equation}
\delchi(t) \geq c t^{\frac{p}{1-p}}  (\log t)^{-2}
\label{eq:de-chi}
\end{equation}
for some constant $c=c(p)>0$. To this end, we demonstrate a pair of feasible $\tilde \pi,\tilde \pi' \in \Pi$ by modifying the construction in the proof of \cite[Lemma 12]{PSW17-colt} to satisfy the additional moment constraints.
	Therein,\footnote{Original version of~\cite{PSW17-colt}, as published in the proceedings, contained an error in
	this derivation, see \textit{arXiv:1702.05574v3} for correction.} it was shown that there exist probability distributions $\pi,\pi'$ on $\integers_+$, such that
	$|\pi(0) - \pi'(0)|\ge \delta$ and
\begin{equation}
%\label{eq:bhel1}
	H^2(\pi P, \pi' P) \le 4 \pth{e^2 \delta_1 \log \frac{1}{\delta_1} }^{\frac{2(1-p)}{p}}\,,
\end{equation}
whenever $\delta_1 = {\delta \over p} $ satisfies $\delta_1 < e^{-1}$. 
More precisely, $\pi$ and $\pi'$ are obtained as follows: 
Let $\alpha=1-\frac{p}{\log \frac{1}{\delta_1}}$, $\beta  = \delta_1\log \frac{1}{\delta_1}$.
Define $g:\complex\to\complex$ by $g(z) = \beta^{\frac{1+z}{1-z}}$.
%$\eta=1-\alpha1-p$, $\beta=t\log\frac{1}{\delta}$, and $r=\frac{1-\alpha1-p}{\alpha\bar1-p}$.
Set
$$ f(z) = (1-\alpha) g(\alpha z) - (1-\alpha) g(\alpha)\,.$$
Define a sequence $\{\Delta_k: k\in\integers_+\}$ via the coefficients of the Taylor expansion of $f$, i.e., $\Delta_k \eqdef [z^k] f(z)$.
Then $\Delta_k = (1-\alpha) \alpha^k [z^k] g(z)$ for $k\geq 1$.
%Then $t_0=(1-\alpha) g(0)-(1-\alpha) g(\alpha) = t - (t\log\frac{1}{\delta})^{2\log\frac{1}{\delta}-1}$
%and $\Delta_k = (1-\alpha) \alpha^k [z^k] g(z)$ for $k\geq 1$.
%Since $t\leq 1/(2e)$ by assumption, we have 
%$$ t > t_0 \ge {t\over 2}\,.$$
%The following estimates are shown in \cite[Eq.~(49)]{PSW17-colt}:
%\begin{equation}\label{eq:bh1}
	%|\Delta_k| \le \bar \alpha \alpha^k\,,\qquad k \ge 1\,.
%\end{equation}
Define the following geometric distribution $\mu$ on $\mathbb{Z}_+$ by $ \mu_k \eqdef \bar \alpha \alpha^k $.
Define now $\pi$ and $\pi'$ via
$$ \pi_k \eqdef \mu_k + \Delta_k\,, \quad \pi'_k \eqdef \mu_k - \Delta_k\,.$$
As shown in~\cite[Lemma 12]{PSW17-colt} we have $\pi_0 - \pi'_0 = 2\Delta_0 \ge \delta$.

Now we estimate the mean of $\pi,\pi'$. 
Note that the mean of the geometric distribution $\mu$ is 
$\sum_{k\geq 0} k \mu_k = \frac{1}{1-\alpha}$.
Furthermore, since the generating function of $\Delta$ is $f$, 
using the facts that $f'(z) = \alpha (1-\alpha) g'(\alpha z)$ and $g'(z) = \frac{2 \log \beta}{(1-z)^2}  \beta ^{\frac{1+z}{1-z}}$,
we have 
\[
\sum_{k\geq 0} k \Delta_k = f'(1) = \alpha (1-\alpha) g'(\alpha) = \frac{2 \alpha \log \beta}{1-\alpha}  \beta ^{\frac{1+\alpha}{1-\alpha}}
\]
Since $\delta_1 \le e^{-1}$ we have $\delta_1 \le \beta \le e^{-1}$ and $\bar\alpha < 1/2$, implying that 
$$ \left|\frac{2 \alpha \log \beta}{1-\alpha}  \beta ^{\frac{1+\alpha}{1-\alpha}}\right | \le {2\over \bar
\alpha} e^{1-{2\over \bar \alpha}} \log {1\over \delta_1} \,.$$
Since $xe^{-x}\le e^{-1}$ on $x\ge 1$ we conclude
$$ \left|\sum_{k\ge 0} \Delta_k\right| \le\log {1\over \delta_1}  $$
and therefore, the first moments of $\pi,\pi'$ are both bounded by $1/\eta$, where we set $\eta \eqdef {\bar \alpha\over
p+1} \le 1/2$.
Finally, define 
\[
\tilde \pi = (1-\eta) \delta_0+ \eta \pi , \quad 
\tilde \pi' = (1-\eta)\delta_0+ \eta  \pi',
\]
From previous estimates we have $|\sum_k k \tilde \pi_k | \leq 1$ and $|\sum_k k \tilde \pi'_k | \leq 1$ whenever
$\delta < {p\over e}$. 
 By convexity, we have $H^2(\tilde\pi P, \tilde \pi' P) \leq \eta H^2(\pi P, \pi' P) \le \tfrac{1}{2} H^2(\pi P, \pi'
 P)$.
In summary, we have constructed $\tilde \pi,\tilde \pi' \in \Pi$ such that 
$|\tilde \pi(0)-\tilde \pi'(0)| \geq \eta \delta = \tfrac{1}{2}\delta \bar \alpha$
and 
$$
H^2(\tilde\pi P, \tilde \pi' P) \leq
{C\over 2} \pth{e^2 \delta_1 \log \frac{1}{\delta_1} }^{\frac{2(1-p)}{p}}.
$$
Finally, choosing $\delta_1$ so that the RHS of the previous display is $t^2$, 
i.e.,
$\delta_1 = \Theta(t^{\frac{p}{1-p}} / \log \frac{1}{t})$, we 
	have $|\tilde \pi(0)-\tilde \pi'(0)| \geq \Omega((t)^{\frac{p}{1-p}} / (\log \frac{1}{t})^2)$.
	This completes the proof of \prettyref{eq:de-chi} and the theorem.
\end{proof}

\subsection{Proof of \prettyref{thm:species}}
	\label{seca:pf-species}
We first present a key lemma, the proof of which requires delicate complex analysis and is postponed till the end of this subsection.

%\begin{theorem}[Restricted species] 
\begin{lemma}
\label{lmm:restricted}
%Fix $\epsilon \in (0,1)$ and consider $\Theta = \mathbb{R}_+, \matx = \mathbb{Z}_+$,
%$P(\cdot|\theta) = \Poi(\bar\epsilon \theta)$, $c(\theta)=\theta$ and $T(\theta) = e^{-\theta}$. We have for all $t\le
%1$
%\begin{equation}\label{eq:rs_ach}
		%\delchi(t) \le 2 t^{\min(2-2\epsilon, 1)} 
%\end{equation}	
%For $\epsilon \ge {1\over 2}$ there exist positive constants $c=c(\epsilon), t_1=t(\epsilon)$ such that for all $t \le
%t_1$
%\begin{equation}\label{eq:rs_conv}
		%\delchi(t) \ge {1\over c} t^{2-2\epsilon} \log^{-2} {1\over t}
%\end{equation}		
%\begin{equation}\label{eq:rs_ach}
		%\delchi(t) \le 2 t^{\min(2-2\epsilon, 1)} 
%\end{equation}	
Consider the Poisson kernel $P(\cdot|\theta) = \Poi(\theta)$. For $s,t>0$, define
%\footnote{It is easy to show (cf.~\cite[Lemma 4]{PSW17-colt}) that $\delta(s,t)$ is related to $\delTV(t)$ in \prettyref{eq:deltv_def} with $T(\theta)=e^{-s\theta}$ 
%via $\frac{1}{2} \delTV(t) \leq \delta(s,t) \leq \delTV(t)$.\yp{I did not understand this, $\delTV$ of what? Is this
%part of the proof of~\eqref{eq:rs_ach}?}}
\begin{equation}
\delta(s,t) \triangleq \sup_{\Delta}\sth{ \int e^{-s\theta} \Delta(d \theta): \|\Delta P\|_{\TV} \le t, \|\Delta\|_{\TV} \le 1}.
\label{eq:deltast}
\end{equation}
where the supremum is taken over all finite signed measure $\Delta$ on $\reals_+$. Then for any $s>0$ and $0 \leq t \leq 1$,
\begin{equation}
\delta(s,t) \leq t^{\min\{1,\frac{2}{s}\}}.
\label{eq:rs_ach}
\end{equation}

Furthermore, fix $s\ge 2$ and consider $\delchi(t)$ in \prettyref{eq:delchi2_def} with $\Theta = \mathbb{R}_+, \matx = \mathbb{Z}_+$,
$P(\cdot|\theta) = \Poi(\theta)$, $\Pi = \{\pi\in\calP(\reals_+): \int \theta \pi(d\theta) \leq 1\}$ and $T(\pi) = \int e^{-s \theta} \pi(d\theta)$. 
 There exist positive constants $c=c(s), t_1=t(s)$ such that for all $t \le t_1$,
\begin{equation}\label{eq:rs_conv}
		%\delta(s,t) \geq 
	c t^{\frac{2}{s}} \log^{-2} {1\over t} \leq \delchi(t) \leq 2 t^{\frac{2}{s}} .
\end{equation}		
\end{lemma}

Before proving \prettyref{thm:species}, we note that the species problem does not completely fall within the purview of
\prettyref{th:highdim}, because the number of distinct species can be infinite.
%unbounded and the parameter $\{p_x: x \in \calX\}$ can be infinite-dimensional. 
However, if the total number of species is restricted to $O(n)$, then the minimax rate readily follows from the general
\prettyref{th:highdim} coupled with the characterization of the modulus of continuity in \prettyref{eq:rs_conv},
cf.~\eqref{eq:res_1}-\eqref{eq:res_2} below.
To deal with the full species problem without restriction, some extra argument is needed, which involves the auxiliary
LP \prettyref{eq:deltast} and introduces extra logarithmic factors in the upper bound of \prettyref{eq:species-risk2}.

\begin{proof}%[Proof of \prettyref{thm:species}]
The result \prettyref{eq:species-risk1} for $r\leq 1$ simply follows from using Good-Toulmin's unbiased estimator and a parametric lower bound (cf.~\cite{GT56,OSW15}). Next we focus on proving \prettyref{eq:species-risk2} for $r>1$.

%We first prove the lower bound part. 
\paragraph{Lower bound.}
We begin with some easy reductions.
By \prettyref{eq:U}, $U = \sum_x \indc{N_x = 0} - V$, where $V \triangleq \sum_x \indc{N_x = 0, N_x' = 0}$, and hence
%Since $\{N_x\}$ are observed, 
estimating $U$ and $V$ are equivalent. 
Next, since $V$ is concentrated near its mean, estimating $V$ and $\Expect[V]$ are essentially equivalent. Indeed, by \prettyref{eq:U} and independence, we have
\[
\Var(U) = \sum_x \Var(\indc{N_x = 0, N_x' > 0}) \leq \Expect[U] \leq r n.
\]
Therefore for any estimator $\hat V$,
\begin{equation}
\Expect[(\hat V-V)^2] \geq \frac{1}{2}\Expect[(\hat V-\Expect[V])^2] - \frac{1}{2} \Var(V) \geq \frac{1}{2}\Expect[(\hat V-\Expect[V])^2] - r n.
\label{eq:EV}
\end{equation}
Define $\theta_x = n p_x$ and $h(\theta)=e^{-(r+1)\theta}$. Then $\Expect[V] = \sum_x h(\theta_x)$.

In order to apply the general result of \prettyref{th:highdim}, we introduce a restricted version of the species problem, where the number of distinct species is at most $n$. Thus any lower bound for the restricted species problem also holds for the original species problem.
Denote the parameters by $\btheta = (\theta_1,\cdots,\theta_n) \in \bTheta_c \triangleq \{\theta\in \reals^n_+: \sum_{i=1}^n \theta_i \leq n\}$.
%$d = n+m = (r+1)n$.  
Let the optimal risk for the restrictive problem be defined as usual:
\begin{equation}\label{eq:res_1}
	\mate^{(res)}_n(r) \eqdef \inf_{\hat V} \sup_{\btheta \in \bTheta_c} \frac{1}{n^2}\Expect[(\hat V-\Expect[V])^2]\,.
\end{equation}
Applying \prettyref{th:highdim} with $c(\theta)=\theta$, $P=\Poi(\cdot)$, 
$\Pi=\{\pi\in\calP(\reals_+): \int \theta\pi(d\theta)\leq 1 \}$ (which is weakly compact), 
and $h(\theta)=e^{-(r+1)\theta}$ (which is bounded), we obtain
\begin{equation}\label{eq:res_2}
\delchi\left({1\over \sqrt{n}}\right)^2 \ge \mate^{(res)}_n(r) \geq c\pth{\delchi\left({1\over \sqrt{n}}\right)^2 - \frac{1}{n} },
\end{equation}
for some absolute constant $c$.
Applying \prettyref{eq:rs_conv} in 
\prettyref{lmm:restricted} with $t=\frac{1}{\sqrt{n}}$ and $s=r+1$, we obtained the lower bound~$\delchi(\frac{1}{\sqrt{n}})\gtrsim 
n^{-\frac{1}{r+1}} \log^{-2}(n)$. The desired lower bound in \prettyref{eq:species-risk2} then follows from
$\mate^{(res)}_n(r) \le \mate_n(r)$, \prettyref{eq:EV}, and \prettyref{eq:res_2}.

\paragraph{Upper bound.}

%As mentioned earlier, due to the infinite-dimensional nature of the problem, we cannot directly invoke the general result of \prettyref{th:highdim}, 
%and it requires extra work to deal with the potential unbounded number of rare species.
We start with the construction of the estimator.
Let $n_0 = \frac{n}{\log n}$ and $n_1=n+n_0$.
For notational convenience, we work with Poisson sampling parameter $n_1=n(1+\frac{1}{\log n})$ in place of $n$.
Thus given observations $\{N_x\} \inddistr \Poi(n_1 p_x)$, the goal is to estimate $U = \sum_x \indc{N_x = 0, N_x' > 0}$ in \prettyref{eq:U}, 
the number of unseen symbols that would be present in the next $rn_1$ observations, where $N_x' \sim \Poi(r n_1 p_x)$
By Poisson splitting, we have access to two independent sets of Poisson observations 
$\{\tilde N_x\} \inddistr \Poi(\lambda_x)$ and $\{\tilde N_x'\} \inddistr \Poi(\lambda_x')$, where $\lambda_x \triangleq n p_x$, $\lambda_x' \triangleq n_0 p_x = \frac{\lambda_x}{\log n}$ and $\tilde N_x+\tilde N_x'=N_x$.
Let $\tilde U \triangleq \sum_x \indc{\tilde N_x = 0} - \indc{N_x=0, N_x' = 0}$. 
Since $U = \tilde U - \sum_x \indc{\tilde N_x = 0, \tilde N_x' > 0}$, where the last sum is observed, thus estimating $U$ is equivalent to estimating $\tilde U$.

To this end, fix a bounded sequence $f: \integers_+ \to \reals$ to be optimized later. 
Fix a large constant $C_0$ and set a threshold $b' =C_0 \log n$.
Consider an estimator of the following form
\begin{equation}
\hat U = \sum_x \hat T_x
\label{eq:species-ach1}
\end{equation}
where 
\begin{equation}
\hat T_x = \begin{cases}
0  &  \tilde N_x' \geq b'  \\
f(\tilde N_x)  &  \tilde N_x' < b'.  \\
\end{cases}
\label{eq:species-ach2}
\end{equation}
%Recall the goal is to estimate the expected number of unseen symbols that would be present in the next $rn$ observations:
Define
\begin{equation}
h(\lambda) \triangleq e^{-\lambda}-e^{-(1+\gamma) \lambda}, 
%e^{-\lambda}(1-e^{-\gamma \lambda}), 
\quad \gamma \triangleq (1+r) \pth{1+\frac{n_0}{n}} -1.
\label{eq:gamma}
\end{equation}
Note that
\[
%T\triangleq 
\Expect[\tilde U] = 
\sum_x (e^{-n p_x} - e^{- (1+r)(n+n_0) p_x}) =  \sum_x h(\lambda_x).
%\sum_x \underbrace{e^{-\lambda_x}(1-e^{-r\lambda_x})}_{\triangleq h(\lambda_x)}.
\]
Then
\[
\Expect[(\hat U-\tilde U)^2]=
\pth{\sum_x (\Expect[\hat T_x]-h(\lambda_x))}^2 + \Var(\hat U-\tilde U).
\]
A simple calculation shows that (cf.~\cite[Lemma 3]{OSW15})
\begin{equation}
\Var(\hat U-\tilde U) \leq n (\|f\|_\infty^2 + \gamma).
\label{eq:var0}
\end{equation}

To bound the bias, 
let $\epsilon=\frac{n_0}{n}=\frac{1}{\log n}$ and note that 
$\lambda_x'=\epsilon \lambda_x$. Set $b = b'/\epsilon = C_0 \log^2 n$.
Using the definition of $\hat T_x$ and the independence of $\{\tilde N_x\}$ and $\{\tilde N_x'\}$, we have:
\begin{align*}
 & ~ |\Expect[\hat T_x-h(\lambda_x)]|\\
= & ~ \left|\Expect\qth{(\hat T_x-h(\lambda_x)) 
\pth{\indc{\tilde N_x' \geq b', \lambda_x' \geq \frac{b'  }{2}}+\indc{\tilde N_x' \geq b' , \lambda_x' \leq \frac{b'  }{2}}+ \indc{\tilde N_x' \leq b' , \lambda_x' \leq 2b'  }+\indc{\tilde N_x' \leq b', \lambda_x' \geq 2b'}}} \right|  \\
\leq & ~ h(\lambda_x) \indc{\lambda_x' \geq \frac{b'}{2}} + h(\lambda_x) \pprob{\tilde N_x' \geq b'} \indc{\lambda_x' \leq \frac{b'  }{2}} \\
& ~ + |\eexpect{f(\tilde N_x)} - h(\lambda_x)| \indc{\lambda_x' \leq 2b' } + (\|h\|_\infty+1)\pprob{\tilde N_x' \leq b' } \indc{\lambda_x' \geq 2b'  }  \\
\leq & ~ \underbrace{h(\lambda_x) \indc{\lambda_x \geq \frac{b  }{2}}}_{\I} + \underbrace{h(\lambda_x) \exp(-b  \kappa) \indc{\lambda_x \leq \frac{b  }{2}}}_{\II} \\
%\log 2-\frac{1}{2})
& ~ + \underbrace{|\eexpect{f(\tilde N_x)} - h(\lambda_x)| \indc{\lambda_x \leq 2b  }}_{\III} + \underbrace{(\|f\|_\infty+1) \exp(-b  \kappa) \indc{\lambda_x \geq 2b  } }_{\IV}, 
%(1-\log 2
\end{align*}
where we used the Chernoff bound for Poisson distributions \cite[Theorem 4.4]{MU06}:
for any $\lambda>0$, $\prob{\Poi(\lambda/2)\geq \lambda} \leq \exp(-\kappa \lambda)$ and $\prob{\Poi(2\lambda)\leq \lambda} \leq \exp(-\kappa \lambda)$, with $\kappa \triangleq \log 2-\frac{1}{2}$.
Note that 
\begin{equation}
\sum _x \lambda_x = n.
\label{eq:lambdanormalize}
\end{equation}
 So
\[
\sum_x \I  \leq \sum_x e^{-\lambda_x}(1-e^{-r\lambda_x}) \indc{\lambda_x \geq \frac{b }{2}} 
\leq \sum_x e^{-\lambda_x} \gamma \lambda_x \indc{\lambda_x \geq \frac{b }{2}} 
\leq \gamma n n^{-\frac{C_0}{2}}.
\]
and
\[
\sum_x \II  \leq \sum_x e^{-\lambda_x}(1-e^{-r\lambda_x}) \exp(-b  \kappa) \leq \gamma n n^{-C_0  \kappa},
\]
and
\[
\sum_x \IV  \leq (\|f\|_\infty+1) n^{-C_0  \kappa} \frac{n}{2b }.
\]
By choosing $C_0$ to be large constant, we have 
\[
\sum_x \I+\II+\IV \leq  \gamma n^{-10} (\|h\|_\infty+1).
\]

Next to bound the main term \III, we choose the coefficient $f$ by solving an LP, which is directly related to the LP \prettyref{eq:deltast} in \prettyref{lmm:restricted}.
Let $f(k)=k g(k-1)$, where $g:\integers_+\to\reals$ is some sequence to be optimized later.
 %(and eventually will be chosen with bounded support).
Then by Stein's identity for Poisson distributions, we have $\eexpect{f(\tilde N_x)} = \lambda_x \eexpect{g(\tilde N_x)}$. 
Put
\[
S(\lambda) \triangleq \frac{h(\lambda)}{\lambda} = \frac{e^{-\lambda} -e^{-(\gamma+1)\lambda}}{\lambda}.
\]
Then we have $\eexpect{f(\tilde N_x)} - h(\lambda_x) = \lambda_x (\eexpect{f(\tilde N_x)} - S(\lambda_x))$.
Recall that the Poisson kernel $P$ acts as follows:
\begin{itemize}
	\item For any sequence $g:\integers_+\to\reals$, $Pg: \reals_+\to\reals$ is a function defined via $(Pg)(\lambda) \triangleq \Expect[g(\Poi(\lambda))]$;
\item For any distribution $\pi$ on $\reals_+$, $\pi P$ denotes the Poisson mixture whose probability mass function is given by $(\pi P)(k) = \int e^{-\lambda} \frac{\lambda^k}{k!} \pi(d\lambda), k \geq 0$.
\end{itemize}
%Then
%\[
%\sum_x \III \leq n \|Pg - S\|_{L_\infty(\reals_+)}
%\]
For any $t >0 $, define the following bias-variance tradeoff LP:
\begin{equation}
\delta(t) \triangleq \inf_g  \|S-Pg \|_{L_\infty(\reals_+)} + t \|g\|_{\ell_\infty(\integers_+)}.
\label{eq:species-bvLP}
\end{equation}
Next we bound $\delta(t)$ by the dual LP:
%We consider the dual LP as follows:
\begin{align}	
		%& ~  \inf_{g \in \ell_\infty(\integers_+)}  \|S-Pg\|_{L_\infty(\reals_+)} + t \|g\|_{\ell_\infty(\integers_+)} \\
\delta(t) 	\stepa{=} & ~ \inf_{g \in \ell_\infty(\integers_+)} \sup_{\|\Delta\|_{\TV} \leq 1, \|\nu\|_{\TV} \leq 1}  \int (S-Pg) d\Delta + t \int g d\nu \nonumber \\
	\stepb{=} & ~ \sup_{\|\Delta\|_{\TV} \leq 1, \|\nu\|_{\TV} \leq 1} \inf_{g \in\ell_\infty(\integers_+)} \int (S-Pg) d\Delta + t \int g d\nu \nonumber \\
	\stepc{=} & ~ \sup_{\|\Delta\|_{\TV} \leq 1, \|\nu\|_{\TV} \leq 1} \inf_{g \in \ell_\infty(\integers_+)} \int S d\Delta + \int g d(t \nu - \Delta P) \nonumber \\
	\stepd{=} & ~ \sup_{\Delta} \sth{\int S d\Delta: \|\Delta\|_{\TV} \leq 1, \|\Delta P\|_{\TV} \leq t}, \label{eq:deltat-dual}
	\end{align}
	where in (a) $\Delta$ and $\nu$ are finite signed measures on $\reals_+$ and $\integers_+$, respectively;
	(b) follows from Ky Fan's minimax theorem (\prettyref{thm:minimax}), since 
$\{\Delta: \|\Delta\|_{\TV}\leq 1\}$ and $\{\nu: \|\nu\|_{\TV}\leq 1\}$ are compact in their respective weak topology,  and
for every bounded $g$, $\nu \mapsto \int g d\nu$ and 
$\Delta \mapsto \int (S-Pg) d\Delta$ are both weakly continuous since both $S$ and $Pg$ are bounded;
	(c) follows from Fubini's theorem: $\int Pg d\Delta = \int g  d(\Delta P)$;
	(d) is because
\[
\inf_{g \in \ell_\infty(\integers_+)} \int S d\Delta + \int g d(t \nu - \Delta P) = 
\begin{cases}
- \infty  & t \nu \neq \Delta P \\
 0 & t \nu = \Delta P
\end{cases}
\]
To relate the LP \prettyref{eq:deltat-dual} to the LP \prettyref{eq:deltast} considered in \prettyref{lmm:restricted}, the key observation is the following integral representation:
\[
S(\lambda) = \int_1^{\gamma+1} e^{-\lambda s} ds.
\]
Interchanging the integral with the supremum in \prettyref{eq:deltat-dual}, we obtain the following upper bound
\begin{equation}
\delta(t) \leq \int_1^{\gamma+1} \delta(s,t) ds
\label{eq:deltadelta}
\end{equation}
where $\delta(s,t)$ is defined in \prettyref{eq:deltast}.
In view of \prettyref{eq:rs_ach} and \prettyref{eq:deltadelta}, we have
\begin{equation}
\delta(t) \leq \gamma t^{\frac{2\gamma}{1+\gamma}}.
\label{eq:species-deltat}
\end{equation}
Thus, for the specific value of $t = \frac{1}{\sqrt{n}}$, there exists $g^*: \integers_+\to\reals$, such that
%\|F_g - S\|_{L_\infty(\reals_+)} \leq 2 t t^{\frac{2}{1+t}}, \qquad 
%\|g\|_{\ell_\infty(\integers_+)} \leq t^{\frac{2}{1+t}-1}.
\begin{equation}
\sup_{\lambda \geq 0} |\Expect[g^*(\Poi(\lambda))] - S(\lambda)|  \leq \gamma n^{-\frac{1}{1+\gamma}}, \qquad 
%\sup_{k\geq 0} |g^*(k)| 
\|g^*\|_{\infty}
\leq \gamma n^{\frac{1}{2}-\frac{1}{1+\gamma}}.
\label{eq:gstar}
\end{equation}

Next, we truncate $g^*$. 
%so that the sequence $\{kg(k)\}$ has bounded sup-norm. 
Set 
$\lambda_0=2 b $ and 
$L = 2 \lambda_0 = 4 C_0 \log^2 n$ and define $g$ by 
\begin{equation}
g(k) =  g^*(k) \indc{k \leq L}.
\label{eq:gg}
\end{equation}
%
%\[
%h(k) = \begin{cases}
%k g(k)  &  0 \leq k \leq L\\
%0 & k > L\\
%\end{cases}
%\]
Since  $f(k)=k g(k-1)$, we have $\|f\|_\infty \leq  L \|g^*\|_\infty$. In view of \prettyref{eq:var0} and \prettyref{eq:gstar}, we have the variance bound
\[
\var(\hat U-U) \leq 4 \gamma L^2 n^{\frac{2\gamma}{1+\gamma}} = O(\gamma n^{\frac{2\gamma}{1+\gamma}} \log^4 n).
\]
	Furthermore, truncation incurs a small bias since 
\[
\left|\expect{g^*(\tilde N_x) \indc{\tilde N_x > L}} \right| \leq \|g^*\|_\infty \prob{\tilde N_x > L}.
\]
Note that $\prob{\Poi(\lambda) > L} \leq \lambda \sum_{i\geq L} \frac{\lambda^{i-1}}{(i-1)!} e^{-\lambda} = \lambda \prob{\Poi(\lambda) > L-1}$. Thus
\begin{align}
\sum_x |\Expect[g^*(\tilde N_x) \indc{\tilde N_x > L}] | \indc{\lambda_x \leq \lambda_0}
\leq & ~  n \|g^*\|_\infty  \prob{\Poi(\lambda_0) > 2 \lambda_0-1}  \nonumber \\
\overset{\prettyref{eq:gstar}}{\leq}& ~ 	\gamma n^{\frac{3}{2}-\frac{2\gamma}{1+\gamma}} \exp(- \kappa \lambda_0/2) \leq n^{-5}.
\label{eq:truncation}
\end{align}
%where the last step follows by $C_0$ being a large constant.
%where $\kappa' \geq \frac{\kappa}{2}$ is an absolute constant. 
%Thus
Thus
\begin{align}
\sum_x \III 
= & ~ \sum_x |\eexpect{f(\tilde N_x)} - h(\lambda_x)| \indc{\lambda_x \leq \lambda_0}	\nonumber \\
= & ~ \sum_x \lambda_x |\eexpect{g(\tilde N_x)} - S(\lambda_x)| \indc{\lambda_x \leq \lambda_0}	\\
\overset{\prettyref{eq:gg}}{\leq} & ~ \sum_x \lambda_x |\eexpect{g^*(\tilde N_x)} - S(\lambda_x)| \indc{\lambda_x \leq \lambda_0} + \sum_x |\Expect[g^*(\tilde N_x) \indc{\tilde N_x > L}] | \indc{\lambda_x \leq \lambda_0} 	\\
\leq & ~ \gamma n^{\frac{r}{1+\gamma}} + n^{-5},
\end{align}
where the last step follows from \prettyref{eq:lambdanormalize}, \prettyref{eq:gstar} and \prettyref{eq:truncation}.

Putting everything together, we have
\begin{align*}
\Expect[(\hat U-U)^2] 
\leq & ~  \pth{\sum_x \I+\II+\III+\IV }^2 + \var(\hat U-U) \\
= & ~ O(\gamma^2 n^{\frac{2\gamma}{1+\gamma}} \log^4 n)	=  O(n^{\frac{2r}{1+r}} \log^4 n),
\end{align*}
where the last step follows from the definition of $\gamma$ in \prettyref{eq:gamma} and that $\frac{2\gamma}{1+\gamma}= \frac{2r}{1+r} + \frac{2\epsilon}{1+\gamma}$ with $\epsilon = \frac{n_0}{n}=\frac{1}{\log n}$. 
%% NOTE
% this is the reason we need to choose $n_0=n/\log n$. The optimal choice for $\epsilon=n_0/n$ is obtained by optimizing $n^{\epsilon} 1/\epsilon^2$, achieved by $\epsilon=\frac{\log\log n}{\log n}$. This results in just a slightly better bound than before by a $(\log\log n)^2$ factor. So we omit it.
Recall that the above is proved for sample size $n_1=n(1+1/\log n) \asymp n$.
Dividing both sides by $n^2$ yields the upper bound in \prettyref{eq:species-risk2}.

Finally, we address the construction of the estimator and its computational complexity. From the above proof, combining \prettyref{eq:species-ach1}, \prettyref{eq:species-ach2}, \prettyref{eq:species-bvLP}, \prettyref{eq:gg} and \prettyref{eq:truncation}, we see that it suffices to choose an estimator of the following form
\begin{equation}
\hat U = \sum_x \tilde N_x \cdot g^*(\tilde N_x-1) \indc{\tilde N_x < L} \indc{\tilde N_x' < b'}
\label{eq:species-ach}
\end{equation}
where $g^*$ is the solution of the following infinite-dimensional LP:
\begin{equation}
\inf_{g}  \|S-Pg \|_{L_\infty([0,\lambda_0])} + \frac{1}{\sqrt{n}} \|g\|_{\ell_\infty},
\label{eq:species-comp1}
\end{equation}
with $(Pg)(\lambda) = \expects{g(N)\indc{N\leq L}}{N\sim \Poi(\lambda)}$.
Recall that $\lambda_0$, $L$ and $b$ are all $\Theta(\log^2 n)$.
Here the decision variable $g: \{0,\ldots,L\} \to \reals$ is finite-dimensional; however the objective function involves the $L_\infty$-norm and is equivalent to setting a continuum of constraints.
It remains to show that one can find a finite-dimensional LP whose solution is as good as \prettyref{eq:species-comp1}, statistically speaking. We do so by means of discretization.
%First of all, from \prettyref{eq:species-deltat} we know that the value of \prettyref{eq:species-comp1} is at least $O(n^{-c})$ for some constant $c$ depending only on $r$.
From \prettyref{eq:gstar} we see that it suffices to consider $\|g\|_\infty \leq \gamma n^{\frac{1}{2}-\frac{1}{1+\gamma}}$.
For some small $\varepsilon$ to be specified, let  $m = \floor{\lambda_0/\varepsilon}$ and $M \triangleq \varepsilon \{1,\ldots,m\}$. 
Consider the following discretized version of \prettyref{eq:species-comp1},
\begin{equation}
\inf_{g}  \|S-Pg \|_{L_\infty(M)} + \frac{1}{\sqrt{n}} \|g\|_{\ell_\infty},
\label{eq:species-comp2}
\end{equation}
To compare \prettyref{eq:species-comp1} and \prettyref{eq:species-comp2}, note that for any $\lambda \in [0,\lambda_0]$, there exists $\lambda' \in M$ such that $|\lambda-\lambda'| \leq \varepsilon$.
Note that
$S(\lambda) = \frac{h(\lambda)}{\lambda} = \frac{e^{-\lambda} -e^{-(\gamma+1)\lambda}}{\lambda}$ is $L$-Lipschitz in $\lambda$ for some $L$ depending only on $r$.
Therefore
$|S(\lambda)-S(\lambda')| \leq L\varepsilon$. Furthermore, since $D(\Poi(\lambda)\|\Poi(\lambda')) = \lambda \log \frac{\lambda}{\lambda'} + \lambda'-\lambda \leq \frac{(\lambda-\lambda')^2}{\lambda'}\leq \varepsilon$, 
by Pinsker's inequality, we have
$|(Pg)(\lambda) -(Pg)(\lambda')| \leq \|g\|_\infty \TV(\Poi(\lambda),\Poi(\lambda')) \leq \gamma \sqrt{n \varepsilon}$.
Choosing $\varepsilon = \frac{1}{n^2}$, we conclude that the value of \prettyref{eq:species-comp1} and \prettyref{eq:species-comp2} only differs by $O(n^{-1/2})$, and solving 
which is an LP with $O(\log^2 n)$ variables and $O(n^2)$ constraints, achieves the upper bound in \prettyref{eq:species-risk2}.
\end{proof}

To close this section, we prove \prettyref{lmm:restricted}. 
The proof relies on two key results from complex analysis: Hadamard's three-lines theorem and the Paley-Wiener theorem.
\begin{proof}
We follow the same program of $H^\infty$-relaxation
 as in the proof of Theorem~\ref{lmm:horo} in \cite{PSW17-colt}. 
For a complex valued function on
$U\subset \complex$ we define $\|f\|_{H^\infty(U)} = \sup_{z\in U} |f(z)|$. If $f$ is holomorphic on a domain $U$ then  $\|f\|_{H^\infty(U)} =\|f\|_{H^\infty(\partial U)}$ by the maximum principle.  The open unit disk is denoted below as $D$
and the unit circle as $\partial D$. To each finite signed measure $\Delta$ on $\mreals_+$
we associate its Laplace transform:
	$$ f_\Delta(z) \eqdef \int_{\mreals_+} e^{az} \Delta(da)\,,$$
which is a holomorphic function on $\{\Re \le 0\}$ and 
\begin{equation}
\|f_{\Delta}\|_{H^\infty(\Re \le 0)} = \|f_{\Delta}\|_{H^\infty(\Re = 0)}  \le \|\Delta\|_{\TV} \triangleq 
\int_{\reals} |\Delta|(da).
\label{eq:DeltaTV}
\end{equation}
 Similarly, to each finite signed measure $\nu$ on $\integers_+$ we associate its $z$-transform
	$$ f_\nu(z) \eqdef \sum_{m \in \integers_+} \nu(m) z^m\,.$$
Again, $f_\nu$ is holomorphic on a $D$ with 
\begin{equation}
\|f_{\nu}\|_{H^\infty(D)} = \|f_{\nu}\|_{H^\infty(\partial D)} \le \|\nu\|_{\TV} \triangleq \sum_{m \in\integers_+} |\nu(m)|.
\label{eq:nuTV}
\end{equation}
 Furthermore, if $f_\nu$
happens to be holomorphic on $rD$ for $r>1$, then we have from Cauchy integral formula
		\begin{equation}
		 |\nu(m)| \le r^{-m} \|f\|_{H^\infty(rD)}
		\label{eq:cauchy}
		\end{equation}
The important observation for this proof is the following identity: 
	\begin{equation}\label{eq:rs_A}
		f_{\Delta P}(z) = f_\Delta(z-1)\,,
\end{equation}	
where $\Delta$ and $\Delta P$ are measures on $\mreals_+$ and $\integers_+$, with the latter 
obtained by applying the Poisson kernel $P$ to $\Delta$, to wit, 
$\Delta P(m) = \int \frac{e^{-a} a^m}{m!} \Delta(da)$.
Indeed, \prettyref{eq:rs_A} simply follows from Fubini's theorem:
$f_{\Delta P}(z) = \int \sum_{m\geq 0} \frac{e^{-a} a^m}{m!} \Delta(da) = 
\int e^{a(z-1)}\Delta(da) = f_\Delta(z-1)$.

%We now proceed to the upper bound:
%\begin{align} \delchi(t) &\le \delta_{\TV}(t) \label{eq:rs1}\\
	      %&\le \sup_{\Delta}\{ \int \Delta(da) e^{-a}: \|\Delta P\|_{\TV} \le 2t, \|\Delta\|_{\TV} \le 2\}\label{eq:rs2}\\
	      %&\le \sup_{\Delta}\{ \int \Delta(da) e^{-a}: \|f_{\Delta P}\|_{H^\infty(D)} \le 2t, \|\Delta\|_{\TV} \le 2\}\label{eq:rs3}\\
	      %&= \sup_{\Delta}\{ f_\Delta(-1): \|f_{\Delta}\|_{H^\infty(-\bar\epsilon + \bar\epsilon D)} \le
	      %2t, \|\Delta\|_{\TV} \le 2\}\label{eq:rs4}\\
	      %&\le 2 \delta_{H^\infty}(t) \eqdef 2 \sup_{f}\{ f(-1): \|f\|_{H^\infty(-\bar\epsilon + \bar\epsilon D)} \le t,
	      %\|f\|_{H^\infty(\Re <0)} \le 1\} \label{eq:rs5}
%\end{align}
%where~\eqref{eq:rs1} is from~\eqref{eq:delta_all},~\eqref{eq:rs2} is by dropping the constraint $\pi,\pi'\in \Pi$ and
%taking $\Delta = \pi' - \pi$,~\eqref{eq:rs3} is by replacing the total variation constraint on $\Delta P$ with an
%$H^\infty$ one,~\eqref{eq:rs4} is expressing the objective function in terms of Laplace transform of $\Delta$,
%and~\eqref{eq:rs5} is by extending optimization from Laplace transforms $f_\Delta$ to all holomorphic functions on $\{\Re
%<0\}$ and using \prettyref{eq:nuTV}.

We now proceed to proving \prettyref{eq:rs_ach}:
\begin{align} 
\delta(s,t) 
	      & = \sup_{\Delta}\sth{ \int e^{-s\theta} \Delta(d \theta): \|\Delta P\|_{\TV} \le t, \|\Delta\|_{\TV} \le 1} \nonumber \\
	      &= \sup_{\Delta}\{ f_\Delta(-s): \|f_{\Delta}\|_{H^\infty(D-1)} \le t, \|f_\Delta\|_{H^\infty(\Re <0)} \le 1 \}\label{eq:rs4}\\
	      &\le \sup_{f}\{ f(-s): \|f\|_{H^\infty(D-1)} \le t,
	      \|f\|_{H^\infty(\Re <0)} \le 1\} \triangleq \delta_{H^\infty}(t) \label{eq:rs5}
\end{align}
where
\eqref{eq:rs4} is by expressing the objective function in terms of Laplace transform of $\Delta$,
and relaxing the total variation constraint on $\Delta P$ by the 
$H^\infty$-norm constraint, in view of \prettyref{eq:DeltaTV}, \prettyref{eq:nuTV} and \prettyref{eq:rs_A};
\eqref{eq:rs5} is by extending the optimization from Laplace transforms $f_\Delta$ to all holomorphic functions on $\{\Re
<0\}$.

To solve the optimization problem~\eqref{eq:rs5} we first notice that for $s \leq 2$, we have $-s \in D-1$ and thus $\delta_{H^\infty}(t)=t$ (achieved by taking $f(z)=t$). Next consider $s > 2$. Let us
reparameterize $f(z) = g(1+{s\over z})$.
Note that (cf.~\prettyref{fig:hadamard})
\begin{figure}[ht]%
\centering
%\begin{tikzpicture}[scale=1,transform shape,node distance=1cm,auto,>=latex']
\begin{tikzpicture}[scale=1.5,font=\scriptsize,>=latex']
\draw[red,thick] (-1,0) circle (1);
%\draw[black,thick] (0,-1.5)--(0,1.5);
\filldraw[thick] (-2.5,0) circle (0.02);
\node[below] at (-2.5,0) {$-s$};
\node[below] at (-1,0) {$-1$};
\node[below] at (-2,0) {$-2$};
%\draw (0.15,0) arc (0:45:0.15);
%\draw[-open triangle 45] (0,0) -- (0.707,0.707);
%\draw[dotted] (0.707,0.707) -- (0.707,0) node[below] {$X$};
%\node at (0.2,0.1) {$\Theta$};
%\draw[->] ([shift=(30:1cm)]2,1) arc (30:60:1cm);
\draw[-latex] (-3,0) -- (1,0) node[right] {$\text{Re}(z)$};
\draw[-latex,blue,thick] (0,-1.5) -- (0,1.5) node[right,black] {$\text{Im}(z)$};
%\draw[blue,fill=blue]  (360/9*\ang:2) circle (0.05);
%\draw[-latex] (-2.2,0) -- (2.2,0) node[right] {$\text{Re}$};
%\draw[-latex] (0,-2.2) -- (0,2.2) node[right] {$\text{Im}$};
 
\path[-latex] (0.7,0.5) edge[bend left] node [above] {$w=1+\frac{s}{z}$} (2.7,0.5);

\draw[-latex] (3,0) -- (6,0) node[right] {$\text{Re}(w)$};
\draw[-latex] (4,-1.5) -- (4,1.5) node[right] {$\text{Im}(w)$};
\draw[thick,blue] (5,-1.5) -- (5,1.5);
\draw[thick,red] (3.5,-1.5) -- (3.5,1.5);
\filldraw[thick] (4,0) circle (0.02);

\node[below] at (5.1,0) {$1$};
\node[below] at (4.1,0) {$0$};
\node[below] at (3.2,0) {$1-\frac{s}{2}$};

\end{tikzpicture}    
\caption{The function $w=1+\frac{s}{z}$ maps the circle $-1+\partial D$ to the line $\Re = 1-\frac{s}{2}$, 
$\Re=0$ to $\Re=1$, and the point $z=-s$ to $w=0$.}%
\label{fig:hadamard}%
\end{figure}
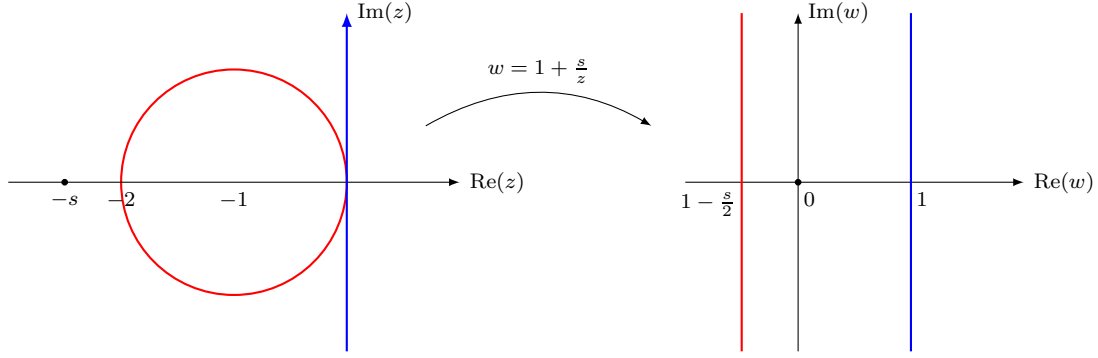
\[
\|f\|_{H^\infty(\Re <0)} = \sup_{\Re(z) < 0} |f(z)| = \sup_{\Re(z) < 0} \left|g\pth{1+\frac{s}{z}}\right|= \sup_{\Re(w) < 1} \left|g(w)\right| = 
\|g\|_{H^\infty(\Re < 1)}.
\]
Furthermore, since 
\begin{equation}
1 + \frac{1}{w} \in D \iff \Re(w) \leq - \frac{1}{2},
\label{eq:wz}
\end{equation}
 we have
\[
\|f\|_{H^\infty(D-1)} = 
\sup_{z \in D-1} \left|g\pth{1+\frac{s}{z}}\right|=
\sup_{1+\frac{x}{w-1} \in D} \left|g(w)\right|=
\sup_{\Re(w) <1- \frac{s}{2}} \left|g(w)\right| = \|g\|_{H^\infty(\Re < 1-\frac{s}{2})}.
\]
%where $\epsilon_1 \triangleq {2\epsilon-1\over 2-2\epsilon} > 0$. Indeed, this follows from the fact that $1+{1\over w} \in D$
%iff $\Re(w) \le -{1\over 2}$. Thus, with $z = \frac{1}{w-1}$, we have $z \in -\bar\epsilon + \bar \epsilon D$ iff $1+{1\over \bar \epsilon(w-1)} \in D$
%iff $\Re(w) \le -\epsilon_1$.
Hence, we have
\begin{equation}\label{eq:rs6}
	\delta_{H^\infty}(t) = \sup_{g}\{ g(0): \|g\|_{H^\infty(\Re < 1-\frac{s}{2})} \le t,
	      \|g\|_{H^\infty(\Re <1)} \le 1\} .
\end{equation}

For any $g$ feasible to \prettyref{eq:rs6}, which is bounded on the strip $\{z: 1-\frac{s}{2} \leq \Re(z) \leq 1\}$, 
by Hadamard's three-lines theorem (see, e.g., \cite[Theorem 12.3]{simon2011convexity}), 
$x \mapsto \log \|g\|_{H^\infty(\Re < x)} $ is convex. Since $(1-\frac{s}{2}) \frac{2}{s} + (1-\frac{2}{s})=0$, we get 
\[
|g(0)| \le
\|g\|_{H^\infty(\Re < 0)} 
\leq 
\pth{\|g\|_{H^\infty(\Re < 1-\frac{s}{2})} }^{\frac{2}{s}}
 \pth{\|g\|_{H^\infty(\Re <1)}  }^{1-\frac{2}{s}}
\leq t^{\frac{2}{s}},
\]
 for any $g$ feasible for \eqref{eq:rs6}. Furthermore, this is achieved by taking 
%$g(z) = e^{-2\bar\epsilon \ln{1\over t} \cdot (z-1)}$. 
$g(z) = t^{\frac{2}{s}(1-z)} $. 
So we have proved
	$$ \delta_{H^\infty}(t) = t^{\frac{2}{s}}, $$
	and the optimizer in~\eqref{eq:rs5} is
\begin{equation}
\label{eq:fstar}
		%f_*(z) = e^{c_t/z}, c_t = 2\bar\epsilon \ln {1\over t}  
		f_*(z) = t^{-\frac{2}{z}},
\end{equation}	
which turns out to not depend on $s$. This completes the proof of~\eqref{eq:rs_ach}.

\medskip
Next we prove \eqref{eq:rs_conv} for $s\geq 2$. 
The upper bound is clear:
\begin{align} \delchi(t) &\le \delTV(t) \label{eq:rs1}\\
	      &\le \sup_{\Delta}\sth{ \int \Delta(d\theta) e^{-s \theta}: \|\Delta P\|_{\TV} \le 2t, \|\Delta\|_{\TV} \le 2} \label{eq:rs2}\\
	      & = 2 \delta(s,t) \leq 2 t^{\frac{2}{s}} \label{eq:rs3}
\end{align}
where~\eqref{eq:rs1} is from~\eqref{eq:delta_all},~\eqref{eq:rs2} is by dropping the constraint $\pi,\pi'\in \Pi$ and
taking $\Delta = \pi' - \pi$, and \eqref{eq:rs3} is by \prettyref{eq:rs_ach}.

Finally, we prove the lower bound part of \eqref{eq:rs_conv}. To this end we need to produce a pair of distributions $\pi,\pi'$ that are feasible for $\delta_{\chi^2}(t)$. We could try
to take them to be positive and negative part of the measure $\Delta$ that whose Laplace transform coincides with \prettyref{eq:fstar}, i.e., $f_\Delta = f_*$; however, this approach does not directly work (for example, if $\Delta$ were a finite measure, its characteristic function would have been given by
$e^{ic_t\over \omega} 1\{\omega\neq 0\}$, which is discontinuous at $\omega=0$ and thus not the characteristic function of any finite
measure on $\mreals$).
% while being finite, does not have a density and does not
%have a finite moment; indeed the characteristic function of $\Delta$ is $e^{ic_t\over \omega}$). 
Instead, below we construct a sequence of measures approximating $\Delta$.

For each $0<\alpha<1$ (in the end we will take $\alpha \sim {1\over \log {1\over t}}$) define
%$$ f_\alpha(z) = {1\over (z-1)^2} e^{c_t/(z-\alpha)}\,.$$
$$ f_\alpha(z) = {1\over (z-1)^2} t^{-\frac{2}{z-\alpha}} = {1\over (z-1)^2} e^{c_t/(z-\alpha)}, \quad c_t \triangleq 2 \log \frac{1}{t}.$$
Let $G_\alpha$ be a real-valued function on $\mreals$ (whose existence is to be established), such that its Laplace transform is given by $f_\alpha$, i.e.
$$ \int_{\mreals} G_\alpha(a) e^{a z} da = f_\alpha(z) \qquad \forall z: \Re(z) \le 0\,.$$
Let $H_0$ be the following probability distribution on $\mreals_+$
	$$ H_0(dx) = (1-\lambda) \delta_0(dx) + \lambda \gamma e^{-\gamma x} 1\{x\ge 0\} \, dx\,,$$
	which is a mixture of a point mass at zero and an exponential distribution.
We then take
$$ \pi = H_0, \quad \pi' = (1-\tau_0) H_0 + \xi G_\alpha\,,$$
where
	\begin{equation}
	\tau_0 = \xi \int_{\mreals} G_\alpha(x) dx = \xi f_\alpha(0) = \xi e^{-{c_t\over \alpha}} 
	\label{eq:tau0}
	\end{equation}	
so that $\pi'$ is normalized.
To complete the proof we have to prove that a certain choice of $(\alpha,\xi,\gamma,\lambda)$ achieves the following six goals for all sufficiently small
$t$:
\begin{enumerate}
\item \label{item1} $G_\alpha$ is a real-valued density\footnote{Although not directly needed for the statistical lower bound, we require $G_\alpha$ to have a density in order to apply the Paley-Wiener theorem which ensures it is supported on $\reals_+$ and hence can be used as a valid prior.} supported on $\mreals_+$ ;
\item \label{item2} $\pi'$ is a probability measure (i.e.~it is a positive measure);
\item \label{item3} $\EE_\pi[\theta] \le 1$;
\item \label{item4} $\EE_{\pi'}[\theta] \le 1$;
\item \label{item5} The separation of means satisfies:
	$$ T(\pi')-T(\pi) \ge  {K\over (1+s)^2 \log^2 \frac{1}{t}} t^{\frac{2}{s}}\,,$$
	for some constant $K$ (here and below, $K$ denotes an absolute constant, possibly
different on different lines), where recall that $T(\pi) = \Expect_{\pi}[e^{-s\theta}]$;
%$T(\pi) = \int e^{-s\theta} \pi(d\theta)$;
\item \label{item6} The $\chi^2$-divergence satisfies:
	$$ \chi^2(\pi'P \| \pi P) \le t^2. $$
\end{enumerate}
We make the following choices of parameters:
	\begin{equation}\label{eq:rs_cho}
		\gamma = {\alpha\over 2}, \lambda = {\alpha \over 4}, \xi = {\alpha^2\over 16}, \alpha = {1\over c_t}
\end{equation}
Note that as $t\to 0$, all of the above vanish with polylog$(\frac{1}{t})$ speed.
%Furthermore, since $\alpha \to 0$ with $t\to 0$ we have
%that for any $m$: $\alpha^{-m} e^{-{c_t\over \alpha}} \le \alpha^{-m} e^{-{1\over \alpha}} \to 0$ with $t\to 0$.
%Similarly, 
%\begin{equation}\label{eq:rs9}
		%{c_t \over \alpha^2} e^{-{c_t\over \alpha}} \le {c_t^2 \over \alpha^2} e^{-{c_t\over \alpha}} \to 0 
%\end{equation}	
%since $y^2 e^{-y} \to 0$ as $y\to \infty$. 
%	So we
%may and will assume that $t$ is
%sufficiently small, i.e. $t\le t_1$, so that 
%	$$ {1\over \alpha^{2}} e^{-{c_t\over \alpha}}  \le {1\over 2}\,.$$

We start with item \ref{item1}. To get a formula for $G_\alpha$ we notice that the inverse Fourier transform is well-define. Indeed, 
since $|f_\alpha(i \omega)| = \frac{1}{1+\omega^2} \exp(-\frac{c_t \alpha}{\omega^2+\alpha^2})$, 
we have $\omega \mapsto f_\alpha(i\omega)$ is in $L_1(\mreals)$. Hence there exists a continuous bounded function
$G_\alpha$ on $\mreals$ whose Fourier transform is given by $f_\alpha(i\omega)$. 
Moreover, $G_\alpha$ is real-valued since $f_\alpha(-i\omega) =
(f_\alpha(i\omega))^*$, where $*$ denotes the complex conjugation. 
To ensure that $G_\alpha$ is supported on $\reals_+$, note that $f_\alpha$ is
holomorphic in $\{\Re \le 0\}$ and, furthermore, 
\[
|f_\alpha(x+i y)|  = \frac{1}{(1-x)^2 + y^2} \exp\pth{\frac{-c (x-\alpha)}{(x-\alpha)^2+y^2}},
\]
thus
\[
\sup_{x < 0} \int_{\reals} |f_\alpha(x+i y)|^2 dy 
\leq \int_{\reals} \frac{1}{1 + y^2}d y \exp\pth{\frac{c}{\alpha}}  < \infty.
\]
Then the Paley-Wiener theorem (cf.~\cite[Theorem 19.2]{RudinPapa}) implies that $G_\alpha$ is supported on $\mreals_+$. We also get an
estimate on the tail of $G_\alpha(a)$ for $a>0$ as follows:
By the inverse Fourier transform, 
	\begin{align} 
	G_\alpha(a) &= {1\over 2\pi} \int_{-\infty}^\infty e^{c_t\over i\omega - \alpha} {1\over (i\omega -1)^2}
	e^{-i\omega a} d\omega \nonumber \\
			&= {1\over 2\pi i} \int_{0-i\infty}^{0+i\infty} e^{c_t\over z - \alpha} {1\over (z -1)^2}
	e^{-z a} dz \nonumber \\
		&= {1\over 2\pi i} \int_{\tfrac{\alpha}{2}-i\infty}^{\tfrac{\alpha}{2}+i\infty} e^{c_t\over z - \alpha} {1\over (z -1)^2}
	e^{-z a} dz  \label{eq:rs7}\\
		&= {1\over 2\pi} \int_{-\infty}^\infty e^{c_t\over i\omega -\tfrac{\alpha}{2}} {1\over (i\omega + \tfrac{\alpha}{2}-1)^2}
	e^{-(i\omega +\tfrac{\alpha}{2})a} d\omega \nonumber,
\end{align}
where in~\eqref{eq:rs7} we shifted the contour of integration 
since the integrand is holomorphic in the strip $\{0 \leq \Re \leq \frac{\alpha}{2}\}$.
Thus
	\begin{align} 
	|G_\alpha(a)| 	&\le {e^{-a\tfrac{\alpha}{2}}\over 2\pi} \int_{-\infty}^\infty {1\over \omega^2 +
		(1-\tfrac{\alpha}{2})^2} d\omega  \nonumber \\
	&= {1\over 2(1-\tfrac{\alpha}{2})} e^{-a\tfrac{\alpha}{2}} \le e^{-a\tfrac{\alpha}{2}}\,,\label{eq:rs8}
\end{align}
where the last step follows from $\int_{-\infty}^\infty
{1\over K^2+x^2} dx = {\pi \over K}$ and the assumption that $\alpha \le 1$.

We proceed to item \ref{item2}. In view of~\eqref{eq:rs8}, to ensure the positivity of $\pi'$  we only need to verify 
	$$ (1-\tau_0) \lambda \gamma e^{-a \gamma} \ge \xi e^{-{a\alpha \over 2}} $$
Due to the choices in~\eqref{eq:rs_cho} this is equivalent to $1-\tau_0 \ge {1\over 2}$ which is satisfied for sufficiently small
$t$.

For item \ref{item3}, we have $\EE_\pi[\theta] = \lambda {1\over \gamma} = {1\over 2}$. 

For item \ref{item4}, 
%note that 
%Similarly, 
%\begin{equation}\label{eq:rs9}
		%{c_t \over \alpha^2} e^{-{c_t\over \alpha}} \le {c_t^2 \over \alpha^2} e^{-{c_t\over \alpha}} \to 0 
%\end{equation}	
%since $y^2 e^{-y} \to 0$ as $y\to \infty$. 
we can compute the first moment of
$G_\alpha$ from its Laplace transform as follows:
\[
	\int_0^\infty G_\alpha(a) a da = \left.{d\over dz}\right|_{z=0} f_\alpha(z) =
e^{-\tfrac{c_t}{\alpha}}(2-\tfrac{c_t}{\alpha^2}) = 
e^{-\tfrac{1}{\alpha^2}}(2-\tfrac{1}{\alpha^3})  \to 0, 
\]
since $\alpha\to0$ as $t\to0$.
Thus, we have $\EE_{\pi'}[\theta] = (1-\tau_0){1\over 2} + \xi \int a G_\alpha \to {1\over 2}$ as $t\to0$.
 %because of~\eqref{eq:rs10}.

For item \ref{item5}, note that 
\begin{equation}\label{eq:rs11}
	T(G_\alpha) = 
	\int e^{-s a} G_\alpha(a) da = f_\alpha(-s) = {1\over (s+1)^2} t^{{2\over s+\alpha}} \ge {1\over (s+1)^2}
	t^{\frac{2}{s}}\,,
\end{equation}
%where the latter inequality follows from ${c_t\over 1+\alpha} \ge c_t - c_t \alpha = c_t - 1$ due to~\eqref{eq:rs_cho}.
Since $T(H_0) = 1- {s\lambda \over s+\gamma} \in [0,1]$, by linearity, we have from~\eqref{eq:rs11}
\begin{align*}
T(\pi') - T(\pi) 
= & ~ -\tau_0 \pth{1- {s\lambda \over s+\gamma}} + \xi \int_0^\infty e^{-a} G_\alpha(a) da  
\ge -\tau_0 + {\xi \over (s+1)^2}
	t^{\frac{2}{s}} \\
\overset{\prettyref{eq:tau0}}{=}  & ~ 
	\xi\pth{ {1\over (s+1)^2}
	t^{\frac{2}{s}}  - e^{-4 \log^2 \frac{1}{t}}}	 \geq 
	{\xi \over 2(s+1)^2} t^{\frac{2}{s}},
\end{align*}
where the last step holds	for all sufficiently small $t$.
%Furthermore, since for sufficiently small $t$ we must have $2\tau_0 = 2\xi e^{-c_t/\alpha} \le {\xi\over 8e} e^{-c_t}$
%we conclude $T(\pi') - T(\pi) \ge {\xi \over 8e} t^{2\bar \epsilon}$ for all such $t$.

Finally, for item \ref{item6}, we have
\begin{align}
\chi^2((1-\tau_0) H_0 P + \xi G_\alpha P \| H_0 P) 
= & ~ \sum_{m\geq 0} { (\xi G_\alpha P (m) - \tau_0 H_0 P(m))^2 \over H_0
P(m)} 	\nonumber \\
= & ~ \xi^2 \sum_{m\geq 0} {{G_\alpha P (m)}^2 \over H_0 P(m)} - \tau_0^2	 \le \xi^2 \sum_{m\geq 0} {{G_\alpha P (m)}^2 \over H_0 P(m)}.
\label{eq:rs15}
\end{align}
For the denominator we have 
\begin{equation}\label{eq:rs14}
	H_0 P(m) = (1-\lambda) \indc{m=0} + \lambda (1-\beta) \beta^m, \quad \beta = {1 \over \gamma+1}.
\end{equation}
To bound the numerator, by \prettyref{eq:rs_A} the $z$-transform of $G_\alpha P$ is given by
\begin{equation}\label{eq:rs12}
	f_{G_\alpha P}(z) = f_\alpha(z-1) = {1\over ( z-2)^2} e^{c_t \over
z-1-\alpha}
\end{equation}
Our goal is to show that, for $r=1+{\alpha\over
2}$, we have 
 $\|f_{G_\alpha P}\|_{H^\infty(rD)} \le K t$ for some constant $K$. Indeed, the first factor in~\eqref{eq:rs12} 
is bounded by 
$\|{1\over ( z-2)^2} \|_{H^\infty(rD)} \leq {1\over ( 1-\alpha/2)^2} \leq 4$ for all sufficiently small $t$. 
For the second factor, 
in view of \prettyref{eq:wz}, for any $\rho>0$ we have
\begin{equation}
\|e^{\rho/z}\|_{H^\infty(D-1)} = e^{-\rho/2}.
\label{eq:erhoz}
\end{equation} 
Set $\rho = 1 + \frac{3\alpha}{4}$, we have
\[
\|f_{G_\alpha P}\|_{H^\infty(rD)} \stepa{\leq} 4 \|e^{c_t/z}\|_{H^\infty(rD-1-\alpha)}
\stepb{\leq}
4 \|e^{c_t/z}\|_{H^\infty(\rho(D-1)} \stepc{=}   4 e^{-{c_t\over 2\rho}}  \stepd{\leq} 10 t,
\]
where (a) is by \prettyref{eq:rs12}; (b) is because $rD-1-\alpha \subset \rho(D-1)$;
(c) is by \prettyref{eq:erhoz}; (d) is by the choices in \prettyref{eq:rs_cho}.

%Thus for any $r$ and $g$ such that
	%\begin{equation}\label{eq:rs13}
		%\bar \epsilon (rD-1) - \alpha \subset \rho(D-1) 
%\end{equation}	
	%we will have $\|f_{G_\alpha P}\|_{H^\infty(rD)} \le K e^{-{c_t\over 2\rho}}$. To have~\eqref{eq:rs13} we should satisfy
	%$\bar\epsilon r \le \rho$ and $-\bar \epsilon(r+1)-\alpha \ge -2\rho$. Taking $r$ as specified above and
	%$\rho=\bar\epsilon + {3\over 4}\alpha$ works. 
	From Cauchy's integral formula~\eqref{eq:cauchy} we obtain the estimate of the coefficients:
		\begin{equation}\label{eq:rs16}
			G_\alpha P (m) \le K r^{-m} t.
\end{equation}		
Using~\eqref{eq:rs14} and~\eqref{eq:rs16} we continue~\eqref{eq:rs15} to get 
$$ \chi^2((1-\tau_0) H_0 P + \xi G_\alpha P \| H_0 P) \le K t^2 {\xi^2 \over \lambda \gamma} \sum_{m\ge0} (r^2
\beta)^{-m}\,.$$
Since $r^2 \beta = 1+ {\alpha\over 2\bar\epsilon} + o(\alpha)$ we conclude
$$ \chi^2((1-\tau_0) H_0 P + \xi G_\alpha P \| H_0 P) \le K t^2 {\xi^2 \over \lambda \gamma \alpha}  \le t^2 $$
for all sufficiently small $t$ due to~\eqref{eq:rs_cho}.
This completes the proof of \prettyref{eq:rs_conv}.
\end{proof}

\subsection{Proof of \prettyref{th:exp}}
	\label{seca:pf-exp}

\begin{lemma}[Auxiliary convex analysis]\label{lem:cau} Let $X$ and $Y$ be a dual pair of finite-dimensional 
vector spaces and $\Pi$ a compact convex subset of $X$. 
Let $f(x,y)$ be a function on $\Pi \times Y$ concave in $x$ and convex in $y$. Assume in addition:
\begin{enumerate}
\item There exists $e_0\in Y$ such that $\iprod{x}{e_0}=1$ for any $x\in \Pi$.
\item We have $f(x,y+c e_0) = f(x,y)$ for any $c\in \mreals$.
\item For any $c\in\mreals$ we have\footnote{In particular, this implies that $f(x,y) = f(x,-y)$, $f(x,0)=0$ and $f\ge 0$.}
	$$ f(x,c y) = |c| f(x,y)\,.$$
\end{enumerate}
Fix $g\in Y$ and define the following quantities
\begin{align} d(x'\|x) &\eqdef \sup\{\iprod{x-x'}{y}: f(x,y) \le 1\},\\
   d_S(x', x) &\eqdef d(x', (x+x')/2),\\
   \delta_0(t) &\eqdef \inf_y \sup_{x\in \Pi} t f(x,y) + |\iprod{x}{g-y}|,\\
   \delta_1(t) &\eqdef \sup_{x,x' \in \Pi}\{\iprod{x-x'}{g}: d(x'\|x)\le t\},\\
   \delta_2(t) &\eqdef \sup_{x,x' \in \Pi}\{\iprod{x-x'}{g}: d_S(x', x)\le t\}.
\end{align}
We claim the following:
	\begin{enumerate}
	\item $d_S(x', x) = d_S(x, x')$
	\item $d_S(x',x) \le d(x\|x')$
	\item ${1\over 2} \delta_2(t) \le \delta_1(t) \le \delta_2(t)$
	\item And the key result:
		\begin{equation}\label{eq:cau}
			{1\over 2} \delta_1(t) \le \delta_0(t) \le \delta_1(t) .
\end{equation}		
	\end{enumerate}
\end{lemma}
\begin{proof}
\begin{enumerate}
\item This is clear. 
\item To prove $d(x'\|(x+x')/2) \le d(x\|x')$ just notice that $f((x+x')/2,y) \le 1$ implies $f(x,y)\le 2$ by concavity and
positivity. 
\item Implied from above.
\item For the lower bound notice
	$$ \delta_0(t) = \inf_y \sup_{x,x'\in \Pi} t {f(x,y)+f(x',y)\over 2} + {|\iprod{x}{g-y}| +
	|-\iprod{x'}{g-y}|\over 2} \,.$$
In the inner supremum we set $x,x'$ to be the ones achieving $\delta_1(t)$. Then we have $\iprod{x-x'}{g}\ge \delta_1$
and for any $y$ we have 
\begin{equation}\label{eq:cau_1}
	\iprod{x-x'}{y} \le t f(x,y)\,.
\end{equation}
We further lower bound
\begin{align} \delta_0(t) &\ge {1\over 2} \inf_y t (f(x,y)+f(x',y)) + |\iprod{x-x'}{g-y}| \\
			&\ge {1\over 2} \inf_y t f(x,y) + |\iprod{x-x'}{g-y}| \\
			&\ge {1\over 2} \iprod{x-x'}{g} + {1\over 2}\inf_y t f(x,y) - \iprod{x-x'}{y}\,,
\end{align}			
	where in the first step we used convexity of $|\cdot|$, in the second positivity of $f$ and in the last step $|a| \ge a$.
	From~\eqref{eq:cau_1} we conclude that $\delta_0 \ge {1\over 2} \delta_1$.

To prove an upper bound we denote the convex hull $\Pi_2 = \co\{0,2\Pi\} = \{\mu x: x\in \Pi, \mu\in[0,2]\}$ and notice
	$$ \delta_0(t) \le \inf_y \sup_{x\in \Pi, x' \in \Pi_2} t f(x,y) + \iprod{x-x'}{g-y} $$
	We now apply minimax theorem to get
	$$ \delta_0(t) \le \sup_{x\in \Pi, x' \in \Pi_2}  \inf_y  t f(x,y) + \iprod{x-x'}{g-y} $$
	We notice that the inner infimum is $-\infty$ unless $\iprod{x-x'}{h}=0$, i.e. that $x'\in \Pi$, and thus 
	\begin{equation}\label{eq:exp_ap1}
		\delta_0(t) \le \sup_{x\in \Pi, x' \in \Pi}  \inf_y  t f(x,y) + \iprod{x-x'}{g-y} 
\end{equation}	
	Due to the homogeneity of $f(x,\cdot)$ we see that further
		$$ \inf_y t f(x,y) - \iprod{x-x'}{y} = \begin{cases} -\infty, &d(x'\|x) > t\,,\\
							0, d(x'\|x) \le t
							\end{cases}\,. $$
	Consequently, the right-hand side of~\eqref{eq:exp_ap1} evaluates to exactly $\delta_1(t)$.
\end{enumerate}
\end{proof}

\begin{proof}[Proof of \prettyref{th:exp}]

Recall that $ D(P\|Q) =\int dP \log {dP \over dQ}$ denote the Kullback-Leibler (KL) divergence.
We need to introduce two other divergence-like quantities before proceeding. 
\begin{align*}
	 d_J(P,Q) &\eqdef D(P\|Q) + D(Q\|P) = \int dP \log {dP \over dQ} + dQ \log {dQ \over dP}\\
	 d(P_{\gamma'} \| P_{\gamma}) &\eqdef 
	\sup_{\phi\in\matf} \{ \EE_{P_\gamma}[\phi] - \EE_{P_{\gamma'}}[\phi]: 
		\Var_{P_\gamma}[\phi]\le 1\}. 
\end{align*}
We notice that $d_J$ is known as the Jeffreys divergence, while $d(P_\gamma\| P_{\gamma'})$ describes the 
dissimilarity between distributions $P_{\gamma'}$ and $P_{\gamma}$ in terms of the expectations of
unit-variance functions in $\matf$;\footnote{Note that without the restriction $\phi\in\calF$, the supremum coincides with $\chi(P_{\gamma'} \| P_{\gamma})$; see \prettyref{eq:chi_va}.}
an explicit expression for $d$ is given  in~\eqref{eq:exp_d2} below. The modulus of continuity of $T$ with respect to
$d_J$ and $d$ will also play a role:
\begin{align*} \omega_J(t) &\eqdef \sup_{\gamma,\gamma' \in \Gamma_0} \{T(\gamma)-T(\gamma'): d_J(P_\gamma,
		P_{\gamma'}) \le t^2\}\,,\\
	\omega_d(t) &\eqdef \sup_{\gamma,\gamma' \in \Gamma_0} \{T(\gamma)-T(\gamma'): d(P_{\gamma'}\|P_{\gamma}) \le
	 t\} 
\end{align*}

\apxonly{
\begin{enumerate}
\item Hellinger and affinity:
	$$ -\ln \Aff(P_\gamma,P_{\gamma'}) = {1\over2}(C(\gamma') + C(\gamma)) - C\left(\gamma+\gamma'\over
	2\right)$$
	$$ H^2(P_\gamma, P_{\gamma'}) = 2-2\Aff(P_\gamma, P_{\gamma'})$$
\item Chi-squared:
	$$ \ln(1+\chi^2(P_{\gamma'}\|P_\gamma)) = C(2\gamma'-\gamma) + C(\gamma)-2C(\gamma')\,.$$
\item Expression for $d$:
	$$ d( P_{\gamma'}\| P_{\gamma}) =   $$
}

We start by establishing the following comparison
	\begin{equation}\label{eq:bbj_1}
		\omega_J(t) \le \omega_H(t)\,.
\end{equation}	
Indeed, an application of Jensen inequality shows
$$ -2 \log(1-\tfrac{1}{2} H^2(P,Q)) = -2 \log \int \sqrt{dP dQ} \le D(P||Q) $$
and from symmetry we, thus, have
$$ -2 \log(1-\tfrac{1}{2} H^2(P,Q)) \le \tfrac{1}{2} d_J(P,Q)\,.$$
Lower bounding the left-hand side we get $ H^2(P,Q) \le \tfrac{1}{2} d_J(P,Q) \le d_J(P,Q)$
completing~\eqref{eq:bbj_1}.

A routine two-point argument yields the lower bound
	\begin{equation}\label{eq:bbj_2}
		\omega_H(c_3/\sqrt{n}) \le \sqrt{R_n^*(\Gamma_0)}\,
\end{equation}	
for some absolute constant $c_3$.

Our proof will be completed in the following steps:
\begin{itemize} 
\item First, we show by appealing to the minimax theorem the
constructive part:
\begin{equation}\label{eq:exp_p1}
		\sqrt{R_n^*(\Gamma_0)} \le \omega_d(1/\sqrt{n})\,. 
\end{equation}	
\item Next we will show that for some $c_2>0$ and all $t>0$ we have
\begin{equation}\label{eq:exp_p2}
		\omega_J(t) \ge \omega_d(c_2 t)\,.
\end{equation}
\item Subadditivity property of $\omega_d$:
\begin{equation}\label{eq:exp_p2b}
	\omega_d(ct) \ge c\omega_d(t), \quad  \forall 0\le c\le 1\,.
\end{equation}
\item Together~\eqref{eq:bbj_1},~\eqref{eq:bbj_2} and the three steps above imply
	\begin{equation}\label{eq:exp_main1}
		\omega_H({c_0\over \sqrt{n}}) \le \sqrt{R_n^*} \le \omega_d({c_2\over c_2\sqrt{n}}) \le {1\over c_2} \omega_H({1\over \sqrt{n}})\,.
\end{equation}	
	The proof will conclude by showing that for all $n\ge {2\over c_3^2}$ we have (for some constant $c_0>0$):
	\begin{equation}\label{eq:bbj_3}
		\omega_H({c_3\over \sqrt{n}}) \ge c_0 \omega_H({1\over\sqrt{n}}) 
\end{equation}	
\end{itemize}

We proceed to proving the above claims. Let us extend the family $\matf$ to $\matf^* =
\lspan\{\matf, 1\}$ by adding constants. Similarly, we extend $\phi$ to $\phi^*(x)=(1,\phi(x)) \in \reals^{d+1}$ by adding a
constant coordinate. (Note that as exponential family $\phi^*$ no longer satisfies non-degeneracy
condition~\eqref{eq:exp_nondeg}). 
We show an upper bound by considering estimators of the form 
	$$ \hat T = {1\over n} \sum_i g(X_i) = {1\over n} \sum_i \iprod{\gamma^*}{\phi^*(X_i)}\,,$$
where $g$ is an arbitrary (to be selected) element of $\matf^*$, which we represented (here and below) as $g(x) =
\iprod{\gamma^*}{\phi^*(x)}$ for some $\gamma^*\in\reals^{d+1}$. (So everywhere above $g$ and $\gamma^*$ are coupled by this relation.)
Recall that $\Expect_\gamma[\phi(X)]=\mu_f(\gamma)=\mu$ and 
the functional to be estimated is $T(\gamma)=\Iprod{h}{\mu}=\Expect_\gamma[\Iprod{h}{\phi(X)}]$.
Define $h^*=(0,h)$ and $\mu^*=(1,\mu)=\Expect[\phi^*(X)]$. Then we have $T(\gamma)=\Iprod{h^*}{\mu^*}=\Expect_\gamma[\Iprod{h^*}{\phi^*(X)}]$.
We have:
	\begin{align} \inf_{g \in \matf^*} \sup_{\gamma \in \Gamma_0} \sqrt{\EE_{\gamma}[(T-\hat T)^2]} &\le 
		\inf_{g  \in \matf^*}\sup_{\gamma \in \Gamma_0} 
		{1\over \sqrt{n}} \sqrt{\Var_{P_{\gamma}}[g(X)]} 
		+ |\EE_{P_\gamma}[g(X)] - T(\gamma)|\\
	&= \inf_{\phi  \in \matf^*} \sup_{\mu \in M_0} 
		{1\over \sqrt{n}} \sqrt{\Var_{\tilde P_{\mu}}[g(X)]} 
		+ |\iprod{\gamma^* - h^*}{\mu^*}|\\
	&=\delta_0(1/\sqrt{n})\,,\label{eq:exp_d1}
\end{align}		
where $\delta_0$ is defined as (similar to $\deltaa$ previously defined in \prettyref{eq:deltach_def})
$$ \delta_0(t) \eqdef \inf_{g  \in \matf^*} \sup_{\mu \in M_0} 
		 t \sqrt{\Var_{\tilde P_{\mu}}[g(X)]} 	+ |\iprod{\gamma^* - h^*}{\mu^*}|\,. $$
This definition coincides with $\delta_0$ defined in Lemma~\ref{lem:cau} if we set:
	\begin{itemize}
	\item $X=\mreals^{d+1}$, $\Pi = \{1\} \times \Mu_0$, $Y=\matf^*$
	\item Each element $g \in \matf^*$ can be written as $g(x) = y_0 + \sum y_i \phi_{i}(x) =
	\iprod{y}{\phi^*(x)}$, this identifies $Y=\matf^*$ with $\mreals^{d+1}$.
	\item We establish the dual pairing between $X$ and $Y$ as usual $\iprod{x}{y} = \sum_{i=0}^n x_i y_i$. Note
	that when $x=(1,\mu) \in \Pi$ and $y$ is identified with $g$, we have $\iprod{x}{y} = \EE_{\tilde P_\mu}[g(X)]$.
	\item For $x=(1,\mu) \in \Pi$ and $y$ identified with $g$, we set $f(x,y) = \sqrt{\Var_{\tilde P_\mu}[g(X)]}$
	\item $e_0=(1,0,\ldots,0)$ corresponds to the constant function $1$ in $\matf^*$.
	\end{itemize}
Clearly $\iprod{x}{e_0} = \EE_{P_\mu}[1]=1$ for any $x\in \Pi$. Note also that in the definition of $d(P_\gamma \| P_{\gamma'})$ we may extend the
supremum from $\matf$ to $\matf^*$ without change. With these settings, Lemma~\ref{lem:cau} shows
	$$ {1\over2} \omega_d(t) \le \delta_0(t) \le \omega_d(t)\,.$$
This completes the proof of~\eqref{eq:exp_p1}.

We proceed to proving~\eqref{eq:exp_p2}. We start with some preparatory remarks. A simple calculation reveals that
\begin{equation}\label{eq:exp_dj0}
		d_J(P_{\gamma_1}, P_{\gamma_2}) = \iprod{\gamma_1 - \gamma_2}{\mu_f(\gamma_1) - \mu_f(\gamma_2)}\,.
\end{equation}	
Similarly, we have the following expression for $d$:
\begin{align}  d(\tilde P_{\mu'} \| \tilde P_{\mu}) &=  \sup_{a\in \mreals^n}\{ \iprod{\mu-\mu'}{a}: \iprod{\tilde \Sigma(\mu)
a}{a}\le 1\}\\
		&= \sqrt{\iprod{\tilde \Sigma^{-1}(\mu) \Delta}{\Delta}}\,, \quad \Delta = \mu-\mu'\,, \label{eq:exp_d2}
\end{align}		
where we used the identity
\begin{equation}\label{eq:exp_cs}
		\sup_{y\in \mreals^n} \{\iprod{y}{b}: \iprod{Ay}{y}\le 1\} = \sqrt{\iprod{A^{-1} b}{b}}\,,
\end{equation}	
which follows from the Cauchy-Schwarz inequality: $\iprod{y}{b}^2 = \Iprod{A^{\frac{1}{2}}y}{A^{-\frac{1}{2}}b}^2 \le \iprod{A^{-1}
b}{b} \iprod{A y}{y}$.

Thus, we get a more explicit formula for $\omega_d$:
	\begin{equation}\label{eq:exp_d2x}
		\omega_d(t) = \sup_{\mu_1,\mu_2 \in M_0} \left\{ \iprod{\Delta}{h}: \iprod{\tilde \Sigma^{-1}(\mu_2)
	\Delta}{\Delta} \le t^2\,, \, \Delta = \mu_1 - \mu_2 \right\}\,.
\end{equation}	
This expression clearly shows~\eqref{eq:exp_p2b}.

We next establish a key inequality connecting the behavior of $\tilde \Sigma^{-1}(\lambda \mu_1 + \bar \lambda \mu_0)$ with
the assumption~\eqref{eq:exp_a2}. Consider the following chain of inequalities: for any $a \in \mreals^n$,
\begin{align} \iprod{\tilde \Sigma^{-1}(\lambda \mu_1 + \bar \lambda \mu_0) a}{a}^{1\over 2} &=
	\sup_y \left\{\iprod{y}{a}: \iprod{\tilde\Sigma(\lambda \mu_1 + \bar \lambda \mu_0) y}{y}^{1\over2} \le
	1\right\}\label{eq:exp_k1}\\
	&\le 
	\sup_y \left\{\iprod{y}{a}: \lambda\iprod{\tilde\Sigma(\mu_1) y}{y}^{1\over 2} + 
			\bar\lambda\iprod{\tilde\Sigma(\mu_2) y}{y}^{1\over 2}
				\le 1\right\}\label{eq:exp_k2}\\
	&\le \sup_y \left\{\iprod{y}{a}: \lambda\iprod{\tilde\Sigma(\mu_1) y}{y}^{1\over 2}
				\le 1\right\} \label{eq:exp_k3}\\
	&={1\over \lambda} \iprod{\tilde \Sigma^{-1}(\mu_1) a}{a}^{1\over 2} \label{eq:exp_key}\,,
\end{align}
where in~\eqref{eq:exp_k1} we used~\eqref{eq:exp_cs}, in~\eqref{eq:exp_k2} we applied~\eqref{eq:exp_a2},
in~\eqref{eq:exp_k3} we omitted the second term, which is non-negative by~\eqref{eq:exp_nondeg2}, and
in~\eqref{eq:exp_key} we used~\eqref{eq:exp_cs} again.

Next, we obtain an upper bound on $d_J(P_{\gamma_1}, P_{\gamma_2})$ by continuing from~\eqref{eq:exp_dj0}. We
denote $\mu_i = \mu_f(\gamma_i), i=1,2$ and $\Delta = \mu_1-\mu_2$. Notice 
$$ \gamma_1 - \gamma_2 = \int_0^1 \dot \gamma_\lambda d\lambda\,,$$
where with a slight abuse of notation we define $\gamma_\lambda \triangleq \gamma_r(\lambda \mu_1 + \bar \lambda \mu_2)$ and
\begin{equation}\label{eq:exp_djx}
	\dot \gamma_\lambda = {d\over
d\lambda}\gamma_\lambda = \sum_{j=1}^n {\partial \gamma_r \over \partial \mu_j} (\mu_{1,j} - \mu_{2,j}) 
\overset{\eqref{eq:exp_jacinv}}{=}
\Sigma(\gamma_\lambda)^{-1} \Delta\,.
\end{equation}
 Then we have
\begin{align} d_J(P_{\gamma_1}, P_{\gamma_2}) &= 
	\int_0^1 d\lambda \iprod{\tilde \Sigma^{-1}(\lambda \mu_1 + \bar\lambda
	\mu_2)\Delta}{\Delta}\label{eq:exp_dj1}\\
	&\le \int_0^{1/2} d\lambda {1\over \bar \lambda} \iprod{\tilde \Sigma^{-1}(\mu_2)\Delta}{\Delta} + 
		\int_{1/2}^{1} d\lambda {1\over \lambda} \iprod{\tilde \Sigma^{-1}(\mu_1)\Delta}{\Delta}
			\label{eq:exp_dj2}\\
	& = \ln 2 \cdot \iprod{(\tilde \Sigma^{-1}(\mu_2) + \tilde \Sigma^{-1}(\mu_1)) \Delta}{\Delta}
			\label{eq:exp_dj3}\,,
\end{align}
where~\eqref{eq:exp_dj1} is from~\eqref{eq:exp_djx},~\eqref{eq:exp_dj2} is from~\eqref{eq:exp_key}
and~\eqref{eq:exp_dj3} is by computing the integrals.

Finally, consider a pair $\mu_1,\mu_2 \in M_0$ in the optimization~\eqref{eq:exp_d2x}, i.e. such that
\begin{equation}\label{eq:exp_dj_pr}
	\iprod{\tilde \Sigma(\mu_2) \Delta}{\Delta} \le t^2\,,  
\end{equation}
where as usual $\Delta = \mu_1 - \mu_2$. We set 
\begin{equation}\label{eq:exp_dj_awesome}
	\mu_1' = {2\over 3}\mu_1 + {1\over 3} \mu_2\,, \quad \mu_2' = {1\over 3}\mu_1 + {2\over 3} \mu_2\,,
\end{equation}
 From convexity we have $\mu_1',\mu_2' \in M_0$ and also 
\begin{equation}\label{eq:exp_dj4a}
 	\iprod{\mu_1' - \mu_2'}{g} = {1\over 3} \iprod{\Delta}{g}\,.
\end{equation} 
We claim that for some constant $c'>0$ we have
\begin{equation}\label{eq:exp_dj4}
	d_J(\tilde P_{\mu_1'}, \tilde P_{\mu_2'}) \le c' t^2\,,
\end{equation}
which, together with~\eqref{eq:exp_dj4} would clearly establish~\eqref{eq:exp_p2}. Notice that
from~\eqref{eq:exp_dj_awesome} and~\eqref{eq:exp_key} we have
\begin{align} 
\iprod{\tilde \Sigma^{-1}(\mu_1') \Delta}{\Delta} &\le 3 \iprod{\tilde \Sigma^{-1}(\mu_2)
\Delta}{\Delta}	\label{eq:exp_dj5}\\
 \iprod{\tilde \Sigma^{-1}(\mu_2') \Delta}{\Delta} &\le {3\over 2} \iprod{\tilde \Sigma^{-1}(\mu_2) \Delta}{\Delta} 
	\label{eq:exp_dj6}\,.
\end{align}
Hence, the left-hand side in~\eqref{eq:exp_dj3} is upper-bounded by a constant multiple of $\Iprod{\tilde \Sigma(\mu_2)
\Delta}{\Delta}$, which, in view of~\eqref{eq:exp_dj_pr}, shows~\eqref{eq:exp_dj4} and, hence,~\eqref{eq:exp_p2}.

We complete the proof by showing~\eqref{eq:bbj_3}. Notice that $\omega_H(1/\sqrt{n}) \le \omega_H(c_0/\sqrt{\lfloor c_0^2 n \rfloor})$ and then
from~\eqref{eq:exp_main} we have for all $n \ge 2/c_0^2$ and $c_4 = {\sqrt{2}\over c_0}$:
$$ \omega_H(1/\sqrt{n}) \le \omega_d({1\over \sqrt{\lfloor c_0^2 n \rfloor}}) \le \omega_d(c_4/\sqrt{n}) 
\stepa{\le} {c_4\over c_2 c_3} \omega_d(c_2 c_3/\sqrt{n}) \stepb{\le} {c_4\over c_2 c_3} \omega_H(c_3/\sqrt{n})\,,$$
where (a) follows from \eqref{eq:exp_p2b} with $c={c_2c_3\over c_4}$;
(b) follows from~\eqref{eq:exp_p2} and~\eqref{eq:bbj_1}. 
%Now we apply~\eqref{eq:exp_p2b} with $c={c_2c_3\over c_4}$ to continue the chain
	%$$ \le {c_4\over c_2 c_3} \omega_d(c_2 c_3/\sqrt{n}) \le {c_4\over c_2 c_3} \omega_H(c_3/\sqrt{n})\,,$$
	%where in the last step we invoked~\eqref{eq:exp_p2} and~\eqref{eq:bbj_1}. 
	This completes the proof
	of~\eqref{eq:bbj_3}.
\end{proof}

\begin{proof}[Proof of \prettyref{eq:exp_ajn} $\iff$ \prettyref{eq:exp_a2star}]
		To show this equivalence, first notice the representation
	\begin{equation}\label{eq:aj_1}
			C(\gamma + a) - C(\gamma) = \iprod{\mu_f(\gamma)}{a} + \int_0^1 (1-s)
		a^T \Sigma(\gamma + sa) a \, ds 
	\end{equation}		
	since $\nabla C(\gamma) = \mu_f(\gamma)$ and $\mathrm{Hess}\, C(\gamma)=\Sigma(\gamma)$ as in
	\prettyref{eq:moments-exp}. Thus, from here \prettyref{eq:exp_a2star} clearly imply~\eqref{eq:exp_ajn} by virtue of
	\begin{align}\label{eq:aj_2}
		a^T \Sigma(\gamma) a = \Var_{P_\gamma}[\iprod{\phi(X)}{a}] \,.
	\end{align}	
	Conversely,~\eqref{eq:exp_ajn} implies that the function
		$$ \xi \mapsto f_\epsilon(\xi) \eqdef {1\over \epsilon} \{ C(A(\xi) + \epsilon a) - C(A(\xi))\} $$
	is concave for all $\epsilon > 0$. Taking the limit $\epsilon \to 0+$, cf.~\eqref{eq:aj_1}, 
	we conclude that $\xi \mapsto \iprod{\mu_f(A(\xi))}{a}$ is
	concave for any $a$ (in particular, for $-a$ as well), and hence $\xi \mapsto \mu_f(A(\xi))$ must be affine.
	Continuing, again from~\eqref{eq:exp_ajn} we must have that
		$$ \xi \mapsto g_\epsilon(\xi) \eqdef {1\over \epsilon} (f_\epsilon(\xi) - \iprod{\mu_f(A(\xi))}{a}) $$
	is concave for any $\epsilon\neq 0$. Taking the limit as $\epsilon\to 0$, cf.~\eqref{eq:aj_1}, we conclude that
	$ \xi \mapsto a^T \Sigma(A(\xi)) a $
	must be concave, which implies the second claim in \eqref{eq:exp_a2star} in view of \eqref{eq:aj_2}. 
\end{proof}

\section*{Acknowledgment}
Y.~Wu is supported in part by the NSF Grant CCF-1749241, an NSF CAREER award CCF-1651588, and an Alfred Sloan
fellowship. The research of Y. Polyanskiy was supported by the Center for Science of Information (CSoI),
an NSF Science and Technology Center, under grant agreement CCF-09-39370, by the MIT-IBM Watson AI Lab and the USAF-MIT
AI Accelerator.

We thank Prof.~A.~Rakhlin and Prof. A.~Tsybakov for pointing out~\cite{DL91} and~\cite{JN09}, respectively. 
We thank Prof.~A.~Juditsky for many discussions and suggestions. 
We also thank the Associate Editor and the anonymous referees for comments that helped rewrite our paper.

%We do not thank Max Raginsky.

\appendix

\section{Classical applications}
	\label{app:classical}

\subsection{Density estimation}
\label{app:density}

As an application of \prettyref{th:linear}, we consider the classical problem of density estimation under smoothness conditions. 
For simplicity, we focus on the one-dimensional setting where $\pi$ is a distribution on $[-1,1]$ with density 
$\rho$ belonging to the H\"older class $\calP(\beta,L)$ (with $0 < \beta \leq  1$), 
%$(\alpha,L)$-H\"older continuous for some $\alpha,L>0$, 
namely, $|\rho(x)-\rho(y)| \leq L |x-y|^\beta$ for any $x,y \in [-1,1]$. 
Given $n$ iid observations drawn from $\rho$, the goal is to estimate the value of the density at point zero $\rho(0)$. 
%So the minimax risk is given by
%$$ R^*(n) = \inf_{\hat \rho} \sup_{\rho} \EE_{X_i \simiid \rho} |\hat\rho(X_1,\ldots,X_n) - \rho(0)|^2\,.$$

We now verify that this setting fulfills the assumptions of \prettyref{th:linear}.
First, we have $\Theta=\matx=[-1,1]$ and $P$ is the identity kernel: $P(x,E) = 1\{x\in E\}$. We take $\matf=C[-1,1]$ to be all continuous
functions on $[-1,1]$. Note that by identifying a measure $\pi$ on $[-1,1]$ with its density $\rho$, we can set
$T(\pi)=\rho(0)$ and view $\Pi$ as a subset of $C[-1,1]$:
$$ \Pi = \{\rho \in C[-1,1]: |\rho(x)-\rho(y)| \le L|x-y|^\beta\}\,.$$
%In other words, $\Pi$ is the collection of probability distributions on $[-1,1]$ whose density belongs to $\calP(\beta,L)$.
If we endow $\Pi$ and $C[-1,1]$ with the topology of uniform convergence, then $\Pi$ becomes a closed convex subset of
$C[-1,1]$ and the Arzela-Ascoli theorem~\cite[IV.6.7]{DS58} implies that $\Pi$ is in fact compact. Finally, it is clear
that $\rho \mapsto \rho(0)$, $\rho \mapsto \int_{[-1,1]} \rho(x) f(x) dx$ and $\rho \mapsto \int_{[-1,1]} \rho(x) f^2(x)
dx$ are all continuous on $\Pi$ for any $f\in C[-1,1]$. 

So all assumptions A1-A4 of the theorem are satisfied and the minimax quadratic risk is determined within absolute constant factors by $\delchi(\frac{1}{\sqrt{n}})^2$.
It is well-known that the modulus continuity here satisfies the following:
%(a proof is given in \prettyref{app:lmm} for completeness):
\begin{lemma}
\label{lmm:delchi-density}	
There exist constants $c_0,c_1$ depending on $\beta$ and $L$, such that for all $t>0$,
	\[
	c_0 t^{\frac{2\beta}{2\beta+1}}\leq \delchi(t) \leq c_1 t^{\frac{2\beta}{2\beta+1}}.
	\]
\end{lemma}
\begin{proof}%[Proof of \prettyref{lmm:delchi-density}]
	For the upper bound, note that any $f\in \calP(\beta,L)$ is everywhere bounded from above by some constant $C=C(\alpha,L)$, thanks to the fact that $f \geq 0$ and $\int f=1$.
Thus, for any $f,g\in \calP(\beta,L)$ such that $|f(0)-g(0)|=\epsilon$ and $\chi^2(f\|g)\leq t^2$, we have $\|f-g\|_2^2 \leq Ct^2$.
Let $p=|f-g|$. Then $p\geq 0$ and $p$ is $(\beta,2L)$-H\"older continuous.
For sufficiently small $\epsilon$, define $h:[-1,1]\to \reals_+$ by 
$h(x) = \max\{\epsilon - 2L |x|^{\beta},0\}$. Then $p \geq h$ on $[-1,1]$ pointwise and hence
\[
Ct^2 \geq \|f-g\|_2^2 \geq \|h\|_2^2 = C' \epsilon^{2+\frac{1}{\beta}}
\]
for some constant $C'$ depending on $(\beta,L)$. This shows the upper bound.
The lower bound follows from choosing $f$ to be the uniform density, and $g(x) = f(x) + c |x|^{\beta} \sign(x) \indc{|x|^{\beta} \leq \epsilon}$, for some small constant $c$ depending on $(\beta,L)$ and $\epsilon = t^{\frac{2\beta}{2\beta+1}}$.	
\end{proof}

Applying \prettyref{th:linear}, we recover the classical result:
\begin{equation}
\inf_{\hat T} \sup_{\rho \in \calP(\beta,L)} \Expect_{X_1,\ldots,X_n \iiddistr f} |\hat T(X_1,\ldots,X_n) - \rho(0)|^2 \asymp n^{-\frac{2\beta}{2\beta+1}}.
\label{eq:densityrate}
\end{equation}
Furthermore, \prettyref{th:linear} ensures that empirical-mean estimators \prettyref{eq:hatT-emp} of the form $\hat T=\frac{1}{n} \sum_{i=1}^n g(X_i)$ are rate optimal for some appropriately chosen function $g$.
%, without demonstrating an explicit choice. 
Indeed, kernel density estimates are of this form, which achieve the minimax rate for suitably chosen kernel and bandwidth (cf.~e.g.~\cite[Section 1.2]{Tsybakov09}).

%\subsection{Motivating example: nonparametric estimation of linear functionals in Gaussian noise}
\subsection{White Gaussian noise model}
\label{app:whitenoise}

\apxonly{Juditsky says: Lepski figured out that you can get constant $1+o(1)$ as $\sigma\to0$ for estimating max of f.
Donoho~\cite{DL91}
had upper bound / lower bound within 1.25 (and perhaps this is the main contribution compared to~\cite{IH84}?
Juditski-Nemirovsky sweated to get same constant, if I understood correctly; even had to use non-2-pt priors?). 
Someone showed that non-lin estimators are strictly better.}

%Here we show how ideas similar to that behind Theorem~\ref{th:linear} can be used in a completely different problem.
%As a special (though infinite-dimensional) case of the exponential family model considered in \prettyref{sec:exp}, 
In this section we revisit the Gaussian white noise model and re-derive the classical result of \cite{IH84,donoho1994statistical} on the rate optimality (within constant factors) of
linear estimators from convex duality.
%To present a simple example for the duality perspective and give a self-contained exposition, let us revisit the classical Gaussian white noise model: 
Let
\begin{equation}
dX_t = f(t)dt + \sigma dB_t, \quad t \in [0,1],
\label{eq:whitenoise}
\end{equation}
where the unknown function $f$ belong to some \emph{convex} set $\calF$. 
%Denote the law of $X=\{X_t: t \in [0,1]\}$ by $P_f$. 
Given $X=\{X_t: t \in [0,1]\}$, the goal is to estimate some affine functional $T(f)$ (such as $T(f)=f(1/2)$). 
Define the minimax risk as
$$ R^*(\sigma) \eqdef  \inf_{\hat T} \sup_{f\in\calF}\EE_{f}[ (\hat T(X) - T(f))^2]. $$
Although this is a special case of the $n$-sample exponential family model considered in \prettyref{sec:exp} (with $\sigma=\frac{1}{\sqrt{n}}$),
\prettyref{th:exp} proved for finite dimensions cannot be directly applied. Nevertheless, due to the simple structure of the Gaussian model, Le Cam's lower bound can be dualized explicitly, leading to the rate-optimality of linear estimators. Next we carry out this calculation as a self-contained example.

Consider a linear estimator of the form
\[
\hat T = \int_0^1 g(t)dX_t,
\]
where $g$ is some continuous compactly-supported function to be optimized. 
(Denote all such functions by $C_c$.)
Then the bias and variance are given respectively by 
\begin{align*}
\Expect \hat T - T = & ~ \iprod{f}{g} - T(f) \\
\Var(\hat T) = & ~ 	\sigma^2 \|g\|_2^2.
\end{align*}
To bound the bias, note that, trivially,
\[
\inf_{f \in \calF}  \iprod{f}{g} - T(f) \leq \Expect \hat T - T \leq \sup_{f \in \calF}  \iprod{f}{g} - T(f).
\]
Without loss of generality, we can assume that $\sup_{f \in \calF} \iprod{f}{g} - T(f)\geq 0 \geq \inf_{f \in \calF}  \iprod{f}{g} - T(f)$.\footnote{Suppose $\sup_{f \in \calF}  \iprod{f}{g-h} = \epsilon < 0$, i.e., the estimator is always negatively biased, then replacing $g$ by $g-\epsilon$ improves the bias and retains the same variance.}
Therefore, we have
\[
%|\Expect \hat T - T| \leq \sup_{f \in \calF}  \iprod{f}{g-h} + \sup_{f \in \calF}  \iprod{f}{h-g} =  \sup_{f,f' \in \calF}  \iprod{f-f'}{g-h}.
|\Expect \hat T - T| \leq \sup_{f \in \calF}  \iprod{f}{g} - T(f) + \sup_{f \in \calF}  T(f)-\iprod{f}{g} =  \sup_{f,f' \in \calF}  \iprod{f-f'}{g} + T(f')-T(f).
\]
Optimizing the bias-variance tradeoff over $g$ leads to the following convex optimization problem:
%\footnote{The $L_2$-ball is weakly compact (Banach-Alaoglu), so minimax theorem is easy to justify.}
	\begin{align*}	
\sqrt{R^*(\sigma)}
\leq	& ~  \inf_{g \in C_c} \sup_{f,f' \in \calF}  \iprod{f-f'}{g} + T(f')-T(f) + \sigma \|g\|_2\\
	= & ~ \inf_{g\in C_c} \sup_{f,f' \in \calF, \|z\|_2 \leq 1} \iprod{f-f'}{g} + T(f')-T(f) + \sigma \iprod{g}{z} \\
	\overset{(a)}{=} & ~ \sup_{f,f' \in \calF, \|z\|_2 \leq 1} \sth{ T(f')-T(f) + \inf_{g\in C_c} \iprod{f-f' + \sigma z}{g}} \\
	\overset{(b)}{=}  & ~ \sup_{f,f' \in \calF, \|f-f'\|_2 \leq \sigma} T(f')-T(f) \\
	\overset{(c)}{\leq}  & ~ C \sqrt{ R^*(\sigma)},	
	\end{align*}	
	where (a) follows from the minimax theorem (see, e.g., \prettyref{thm:minimax} in \prettyref{sec:linear}); (b) is simply because $ \inf_{g\in C_c} \iprod{f}{g} = -\infty$ if $f\neq 0$ and $0$ if $f=0$; finally,
	(c) follows from Le Cam's two-point lower bound since the KL divergence in the white noise model is given by
\begin{equation}
D(P_f\|P_{f'}) = \frac{1}{2\sigma^2}\|f-f'\|_2^2,
\label{eq:KL-awgn}
\end{equation}
where $P_f$ denotes the law of $\{X_t: t\in[0,1]\}$ as in \prettyref{eq:whitenoise}, and $C$ is an absolute constant.
Thus we have shown that
\begin{equation}\label{eq:ih_gsn}
		\frac{\omega(\sigma) }{C} \le \sqrt{R^*(\sigma)} \le \omega(\sigma)
\end{equation}	
where $\omega(\sigma) \eqdef \sup_{f,f' \in \calF}\{T(f')-T(f): \|f-f'\|_2 \le \sigma\}$ is the modulus of continuity.

\section{Proof of technical results}
	\label{app:lmm}

\begin{proof}[Proof of \prettyref{lmm:horo}] In~\cite[Proposition 9]{PSW17-colt} it is shown for any $d \in \naturals \cup \{\infty\}$,
	\begin{equation}
	\delta_{\TV}(t,d) \le t^{\min(1, {1-\epsilon\over\epsilon})}\,. 
	\label{eq:horo-tv}
	\end{equation}	
	where $\delTV(t,d)$ is defined for the same problem as $\delchi$ but with TV-distance in place of
	$\chi^2$, cf.~\eqref{eq:deltv_def}. From the general relation $\delchi\leq\delTV$	in~\eqref{eq:delta_all} we get~\eqref{eq:horo1}.
	Furthermore, due to~\eqref{eq:delta_superl} and~\eqref{eq:delta_all}, for $\epsilon \le {1\over 2}$ 
	we conclude
			$$ {t\over 2\sqrt{2}} \le \delchi(t,d) \le t. $$

Next, we consider the case of $\epsilon > 1/2$.
The following was shown in \cite[Lemma 12]{PSW17-colt}: 
For every $\delta < \frac{1}{2e}$ and $d \geq \frac{2\epsilon}{1-\epsilon}\ln^2\frac{1}{\delta}$
 %there exist $\delta_0(\epsilon)>0$ and a constant $C=C(\epsilon)>0$
%satisfying the following. For all $\delta < \delta_0$
there exists a pair of probability distributions $\pi$ and $\pi'$ on $\{0,\ldots,d\}$ such that $|\pi(0) - \pi'(0)|\ge
\delta$ and
\begin{equation}\label{eq:bhel}
	H^2(\pi P, \pi' P) \le 36 \pth{e\delta \ln \frac{1}{\delta} }^{\frac{2\epsilon}{1-\epsilon}}.
\end{equation}
	%Now, if we take $s > 0$ and $x = {s\over 2\ln {2\over s}}$ then $x \ln {1\over x} \le s$. In
	%other words, for all $s \le 1/e$, the solution of $\delta \ln {1\over \delta} = s$ satisfies 
		%$$ \delta \ge {s\over 2 \ln {2\over s}}\,.$$
	%Then, for arbitrary $t\ge 0$ set $s = {1\over e} \left({t\over 6}\right)^{1-\epsilon\over \epsilon}$ and
	%conclude that, indeed, given $t \le t_0(\epsilon)$ and $d \ge C_\epsilon \ln^2 {1\over t}$ we must have
	Setting the RHS to $t^2$, we conclude that there exist 	$t_0=t_0(\epsilon)$ and $C=C(\epsilon)$ such that for all $t\leq t_0$ and $d \ge C \ln^2 {1\over t}$, we have
		$$ \delta_{H^2}(t,d) \ge C \left(t\over \ln {1\over t}\right)^{1-\epsilon \over
		\epsilon}\,.$$
	This implies the desired \eqref{eq:horo2} in view of the general inequality $\delta_{\chi^2} \geq \frac{1}{2} \delta_{H^2}$ in~\eqref{eq:delta_all}.
%
	%Finally, to show~\eqref{eq:horo3} we only need to invoke~\cite[Eqn.~(61)]{PSW17-colt} which shows that by
	%adding zeros to $\pi,\pi'$ to increase the dimension from $d =\Omega(\ln^2{1\over t})$ (as before) to $d = \Omega({1\over
	%t^2} \ln^4{1\over t})$ we get instead of~\eqref{eq:bhel}:
	%\begin{equation}\label{eq:bhel1}
		%H^2(\pi P^{(2)}, \pi' P^{(2)}) \le 144 \pth{e\delta \ln \frac{1}{\delta} }^{\frac{2\epsilon}{1-\epsilon}},
	%\end{equation}
	%By the same argument as above we conclude \prettyref{eq:horo3}.
	%%Arguing as for $\delchi$ we obtain a lower-bound on $\delchis$ in~\eqref{eq:horo3}.
\end{proof}

\begin{proof}[Proof of \prettyref{prop:ic}] We aim to apply Theorem~\ref{th:linear}. Assumptions A1 and A2 are verified. For A3 we take $\calF$ to be
the set of all continuous functions on $\calX$ (cf.~the second remark after Theorem~\ref{th:linear}). For A4 we endow 
$\Pi$ with the weak topology. Since $\Pi$ is a set of probability measures on the compact set $[0,1]$, it is tight and
hence weakly compact, establishing A4a. For
A4b we always have that $\pi \mapsto T_m(\pi)$ is weakly continuous, while for $\pi \mapsto T_c(\pi)$ weak continuity is
implied by the assumption (indeed, if $\pi_n \stackrel{w}{\to}\pi$ then $F_{\pi_n}(s_0) \to F_\pi(s_0)$ due to
assumption on $s_0$ being a point of continuity of $F_\pi$). To complete the verification of A4b, we 
need to verify that for any continuous $\phi$ on $\calX$ the functional $\pi P \phi$ is weakly continuous. For example, consider the
case of $i=1$, in which case we have $\phi:[0,1]\times \{0,1\}\to\reals$ and we can represent 
	\begin{equation}\label{eq:ic_j1}
		\pi P \phi = \int_{[0,1]} \pi(d\theta) f(\theta), \quad f(\theta)\triangleq 		
		\int_{[0,1]}da g_1(a) \phi(a,1)1\{\theta \le a\} + \phi(a,0)1\{\theta >
	a\}
\end{equation}	
	%Consider now the function 
	%$$ f(\theta) \eqdef \int_{[0,1]}da g_1(a) \phi(a,1)1\{\theta \le a\} + \phi(a,0)1\{\theta > a\}\,. $$
	We claim that $f$ is continuous. Indeed, if $\theta_n\to \theta$ then $1\{\theta_n > a\} \to
	1\{\theta > a\}$ for almost every $a \in [0,1]$. Hence, from the dominated convergence theorem we also have
	$f(\theta_n)\to f(\theta)$. Thus, the functional in~\eqref{eq:ic_j1} is weakly continuous as integral of a
	continuous function $f$.
\end{proof}

\begin{proof}[Proof of \prettyref{eq:delchi-CL}]
Let $H_k(x)$ denote the degree-$k$ Hermite polynomial and note the fact that for $X\sim N(a,1)$, we have	$\Expect[H_k(X)] = a^k$  and $\var(H_k(X)) = 
k!\sum_{j=0}^{k-1}\binom{k}{j}\frac{a^{2j}}{j!} $. Thus $\var(H_k(X)) \leq k! 2^k$ provided $|a|\leq 1$. Using the variational representation of the $\chi^2$-divergence \prettyref{eq:chi_va}, for any feasible solution $\pi,\pi'$ of \prettyref{eq:delchi-CL-def}, we have
$|m_k(\pi) - m_k(\pi')| \leq \sqrt{k! 2^k t}$, where $m_k(\pi) = \int \theta^k \pi(d\theta)$ denotes the $k$th moment of $\pi$.
By existing results in approximation theory (see \cite{CL11}), there exists a degree-$k$ polynomial $p(x)=\sum_{i=0}^k a_i x^i$ and a constant $C$, such that 
$|a_i| \leq C^k$ and $\sup_{|a|\leq 1} ||a|-p(a)|  \leq \frac{C}{k}$.
Therefore by the triangle inequality, we have
$|\int |\theta| \pi'(d\theta) -\int |\theta| \pi'(d\theta)| \leq \frac{C}{k} + \sqrt{t k! C^k}$. Choosing $k= c \frac{\log \frac{1}{t}}{\log\log\frac{1}{t}}$ for some small constant $c$ proves the upper bound of \prettyref{eq:delchi-CL}.

To show the lower bound part, by the duality between best polynomial approximation 
and moment matching (see e.g.~\cite[Appendix E]{WY14}), there exist $\pi,\pi' \in \calP([-1,1])$ such that $m_i(\pi)=m_i(\pi')$ for $i=1,\ldots,k$, and 
$\int |\theta| \pi'(d\theta) -\int |\theta| \pi'(d\theta) = 2 \inf_{\deg(p)=k} \sup_{|a|\leq 1} ||a|-p(a)| \geq \frac{c}{k}$, where the last inequality is well-known in the approximation theory literature \cite{CL11}. Furthermore, matching first $k$ moments implies that the corresponding Gaussian mixture are close in $\chi^2$-divergence \cite{CL11}: $\chi^2(\pi'*N(0,1)\|\pi*N(0,1)) \le \frac{C^k}{k!}$. Choosing $k= c \frac{\log \frac{1}{t}}{\log\log\frac{1}{t}}$ for some large constant $c$ proves the desired lower bound.	
\end{proof}

\section{Risks for Fisher's species problem with or without Poissonization}
\label{app:species-poisson}

For fixed sample sizes $(n,m)$, define the minimax quadratic risk for estimating $U=U_{n,m}$ as 
\[
\tilde R^*(n,m) \triangleq \inf_{\hat U} \sup_P \Expect_P[(\hat U-U)^2].
\]	
and $R^*(n,m)$ for the Poissonized model with sample sizes $(N,M)$ distributed independently as $\Poi(n)$ and $\Poi(m))$. 
Also denote their normalized version by $\tilde \calE_n(r) = \tilde R^*(n,m)/m^2$  and $\calE_n(r) = R^*(n,m)/m^2$ (with $r=m/n$), the latter of which is addressed by \prettyref{thm:species}.
Nevertheless, the next result shows that $\tilde \calE_n(r)$ satisfies the same upper and lower bounds \prettyref{thm:species} up to an additional $O\pth{\frac{\log n}{n}}$ term.
The proof of this lemma is standard (cf.~\cite[Appendix A]{WY14}).

\begin{lemma}
\label{lmm:species-poi}	
	Let $r=m/n$ be a constant. 
	Let $\alpha = 1/\log n$. Then for large $n$,
	\begin{equation}
	\frac{1}{2}\tilde R^*(n(1+\alpha),m) - O(n \log n) \leq  R^*(n,m) \leq 2 \tilde R^*(n(1-\alpha),m) + O(n \log n).
	\label{eq:species-poi}
	\end{equation}
\end{lemma}

\begin{proof}
	Note the following facts about the risk $	R^*(n,m)$:
	\begin{enumerate}
		\item 	$n\mapsto R^*(n,m)$ is non-increasing.
		
		\item  $0 \leq 	R^*(n,m) \leq 	R^*(0,m) \leq m^2$.
		
		\item $	|\sqrt{R^*(n,m)} - \sqrt{R^*(n,m')}| \leq |m-m'|$, due to the fact that $|U_{n,m}-U_{n,m'}| \leq |m-m'|$.
	\end{enumerate}
	
	For the left inequality, 
	let $N\sim \Poi((1+\alpha)n)$ and $M\sim \Poi(m)$. Set $\Delta = C \sqrt{n \log n}$ for some large constant $C$. 
	Using the Chernoff bound for Poisson, we have
    \begin{align*}
        \tilde R^*((1+\alpha)n, m)
        &\leq  \Expect[R^*(N,M)] \leq \Expect[R^*(n,M)] + O(m^2) \prob{N<n} \\
        &\leq  \sum_{|m'-m|\leq \Delta} R^*(n,m') \prob{M=m'} + O(m^2) (\prob{N>n} + \prob{|M-m|>\Delta}) \\
        &\le (\sqrt{R^*(n,m)} + \Delta)^2 + O(m^2 (e^{-\Omega(\alpha^2 n)} + e^{-\Omega(\Delta^2/n)})).
    \end{align*}
   
	For the right inequality,  it suffices to consider the Bayes risk. Note that for any fixed prior $\pi$, the Bayes risks with fixed or Poissonized sample size, denoted by $\tilde R_\pi^*(n,m)$ and $R_\pi^*(n,m)$, is related by the identity $R_\pi^*(n,m)  = \Expect[R_\pi^*(\Poi(n),\Poi(m))]$. Let $N\sim\Poi((1-\alpha)n)$ and $M\sim\Poi(m))$ be independent. Then
\begin{align*}
R_\pi^*((1-\alpha)n,m) 
\geq & ~ \Expect[R_\pi^*(N,M) \indc{N' \leq n, |M-m| \leq \Delta}  ] \\
\geq & ~ \Expect[R_\pi^*(n,M)\indc{|M-m| \leq \Delta}  ] \prob{N' \leq n} \\
\geq & ~ \max\{\sqrt{R_\pi^*(n,m)} - \Delta,0\}^2 (1-e^{-\Omega(\alpha^2 n)} - e^{-\Omega(\Delta^2/n)})).
\end{align*}
This completes the proof.	
\end{proof}

%\bibliographystyle{alpha}
%\bibliography{IEEEabrv,strings,wu_big,refs_new,reports}

\end{document}